\documentclass[11pt,twoside,a4paper]{article}
\usepackage{amsmath,amssymb,amsthm}
\usepackage{hyperref}
\usepackage{graphicx,psfrag}

\newtheorem{thm}{Theorem}[section]
\newtheorem{lem}[thm]{Lemma}
\newtheorem{pro}[thm]{Proposition}
\newtheorem{cor}[thm]{Corollary}
\newtheorem{con}[thm]{Conjecture}

\theoremstyle{definition}

\newtheorem{exa}[thm]{Example}

\theoremstyle{remark}
\newtheorem{rem}[thm]{Remark}

\newcommand{\R}{\mathbb{R}}
\newcommand{\Z}{\mathbb{Z}}
\newcommand{\N}{\mathbb{N}}
\newcommand{\bO}{\mathbb{O}}
\newcommand{\C}{\mathbb{C}}
\newcommand{\K}{\mathbb{K}}
\newcommand{\bH}{\mathbb{H}}

\newcommand{\cD}{\mathcal{D}}

\newcommand{\cL}{\mathcal{L}}
\newcommand{\cM}{\mathcal{M}}

\newcommand{\cP}{\mathcal{P}}

\newcommand{\cR}{\mathcal{R}}
\newcommand{\cS}{\mathcal{S}}

\newcommand{\al}{\alpha}
\newcommand{\be}{\beta}
\newcommand{\ga}{\gamma}
\newcommand{\Ga}{\Gamma}
\newcommand{\de}{\delta}
\newcommand{\De}{\Delta}
\newcommand{\ep}{\varepsilon}
\newcommand{\om}{\omega}

\newcommand{\si}{\sigma}
\newcommand{\Si}{\Sigma}

\newcommand{\la}{\lambda}
\newcommand{\La}{\Lambda}
\renewcommand{\phi}{\varphi}

\newcommand{\crt}{\operatorname{crt}}

\newcommand{\dist}{\operatorname{dist}}

\newcommand{\Hd}{\operatorname{Hd}}
\newcommand{\CAT}{\operatorname{CAT}}
\newcommand{\hyp}{\operatorname{H}}
\newcommand{\id}{\operatorname{id}}

\newcommand{\area}{\operatorname{area}}
\newcommand{\const}{\operatorname{const}}
\newcommand{\semi}{semi$\text{-}\C$}
\newcommand{\semik}{semi$\text{-}\K$}
\newcommand{\Semik}{Semi$\text{-}\K$}
\newcommand{\aut}{\operatorname{Aut}}

\newcommand{\lift}{\operatorname{lift}}
\newcommand{\sign}{\operatorname{sign}}
\newcommand{\grad}{\operatorname{grad}}
\newcommand{\im}{\operatorname{Im}}

\newcommand{\slope}{\operatorname{slope}}
\newcommand{\conj}{\operatorname{Conj}}

\newcommand{\es}{\emptyset}
\renewcommand{\d}{\partial}
\newcommand{\di}{\d_{\infty}}
\newcommand{\set}[2]{\{#1:\,\text{#2}\}}
\newcommand{\sm}{\setminus}
\newcommand{\sub}{\subset}
\newcommand{\sups}{\supset}

\newcommand{\ov}{\overline}
\newcommand{\wt}{\widetilde}
\newcommand{\wh}{\widehat}

\begin{document}

\title{M\"obius structures and Ptolemy spaces: boundary at infinity
of complex hyperbolic spaces}
\author{Sergei Buyalo\footnote{Supported by RFBR Grant
08-01-00079-a and SNF Grant 20-119907/1}
\ \& Viktor Schroeder\footnote{Supported by Swiss National
Science Foundation Grant 20-119907/1}}

\date{}
\maketitle

\begin{abstract}
The paper initiates a systematic study of M\"obius structures
and Ptolemy spaces. We conjecture that every compact Ptolemy 
space with circles and many space inversions is M\"obius equivalent 
to the boundary at infinity of a rank one symmetric space
$\K\hyp^n$
of noncompact type. We prove this conjecture for the class of complex hyperbolic spaces
$\C\hyp^n$
as our main result.
\end{abstract}

\section{Introduction}
\label{sect:introduction}

The paper initiates a systematic study of M\"obius structures
and Ptolemy spaces. A M\"obius structure on a set
$X$
is a class of M\"obius equivalent metrics. If a M\"obius structure
is fixed then
$X$
is called a M\"obius space. Ptolemy spaces are M\"obius spaces
with property that the inversion operation preserves the
M\"obius structure. A classical example of a Ptolemy space is
an extended
$\wh\R^n=\R^n\cup\infty=S^n$, $n\ge 0$,
where the M\"obius structure is generated by an Euclidean metric
on
$\R^n$,
and
$\R^n\cup\infty$
is identified with the unit sphere
$S^n\sub\R^{n+1}$
via the stereographic projection. For more detail see Section~\ref{sect:moebius_structures}.

Our motivation is to find a M\"obius characterization of the boundary
at infinity of rank one symmetric spaces 
$Y$
of noncompact type. In the case
$Y=\hyp^n$
this problem is solved in \cite{FS2} for every
$n\ge 1$: {\em every compact Ptolemy space such that through any three points there is
a Ptolemy circle is M\"obius equivalent to
$\wh\R^n=\di\hyp^{n+1}$.}
Here a Ptolemy circle is a subspace 
M\"obius equivalent to the Ptolemy space
$\wh\R=S^1$.

Given distinct points
$\om$, $\om'$
in a Ptolemy space
$X$,
there is a well defined notion of a {\em metric sphere}
$S$
between
$\om$, $\om'$,
and a notion of a {\em space inversion} w.r.t.
$\om$, $\om'$, $S$,
which is a M\"obius involution
$\phi_{\om,\om',S}:X\to X$,
see sect.~\ref{subsect:space_inversions}.

We consider a Ptolemy space
$X$
with the following basic properties. 

\noindent
(E) Existence: there is at least one Ptolemy circle in
$X$.

\noindent
(I) Inversion: for each distinct
$\om$, $\om'\in X$
and every metric sphere
$S\sub X$
between
$\om$, $\om'$
there is a unique space inversion
$\phi_{\om,\om',S}:X\to X$
w.r.t.
$\om$, $\om'$
and
$S$.

\begin{con}\label{con:boundary_rank_one} Let
$X$
be a compact Ptolemy space with properties 
($E$) and (I). Then
$X$
is M\"obius equivalent to the boundary at infinity of a rank one
symmetric space
$\K\hyp^k$
of noncompact type taken with a canonical M\"obius structure.
\end{con}

Our main result is the proof of the following particular case of Conjecture~\ref{con:boundary_rank_one},
which gives a M\"obius characterization of the boundary at infinity
$\di\hyp^k$
of a real hyperbolic space
$\hyp^k$
as well as of the boundary at infinity
$\di\C\hyp^k$
of a complex hyperbolic space 
$\C\hyp^k$.

\begin{thm}\label{thm:complex_hyperbolic} Let
$X$
be a compact Ptolemy space with properties 
($E$) and (I). Then
$X$
is homeomophic to a sphere
$S^n$, $n\ge 1$,
and for every
$\om\in X$
there is a fibration
$\pi_\om$
of
$X_\om=X\sm\om$
with fibers homeomophic to
$\R^p$
for some
$p$, $0\le p<n$,
such that for
$\om$, $\om'$
the fibrations
$\pi_\om$, $\pi_{\om'}$
are transformed to each other by any space inversion
$\phi:X\to X$
with 
$\phi(\om)=\om'$.
So the number
$p$
is a M\"obius invariant of
$X$.
In the case
$p=0$
the space 
$X$
is M\"obius equivalent to
$\di\hyp^{n+1}=\wh\R^n$.
In the case
$p=1$
the space
$X$ 
is M\"obius equivalent to
$\di\C\hyp^k$
with
$n=2k-1$, $k\ge 2$,
taken with a canonical M\"obius structure. 
\end{thm}

In sections~\ref{sect:moebius_structures} and \ref{sect:properties_ptolemy} 
we give an introduction to M\"obius structures and Ptolemy spaces.
We emphasize that a Ptolemy space is not just a metric space, and 
there is no distinguished metric in its M\"obius structure.
This is a source of a duality phenomenon between Busemann and
distance functions which takes place in any Ptolemy space and which
cannot be even formulated for an individual metric space, see
Section~\ref{subsect:duality_dist_busemann}. The duality plays
an important role in our paper. 

The proof of Theorem~\ref{thm:complex_hyperbolic} consists of
two parts. In the first part, which occupies
sections~\ref{sect:many_circles_auto} -- \ref{sect:topology_space},
we prove Theorem~\ref{thm:basic_ptolemy}. That theorem
gives a much more detailed information about Ptolemy spaces
discussed in Conjecture~\ref{con:boundary_rank_one} than it is
formulated in Theorem~\ref{thm:complex_hyperbolic}. The second part
of the proof occupies 
sections~\ref{sect:semi_c_planes} -- \ref{sect:model_space} 
and it is dedicated to a particular case when fibers of fibrations
$\pi_\om$
are homeomorphic to 
$\R$.

{\em Acknowledgements.} We are thankful to
J.-H.~Eschenburg for consulting us on isometry groups
acting transitively on spheres. The first author
is also grateful to the University of Z\"urich for 
hospitality and support.

\tableofcontents

\section{M\"obius structures and Ptolemy spaces}
\label{sect:moebius_structures}
In this section we discuss basic notions of M\"obius geometry.

\subsection{M\"obius structures}

A quadruple
$Q=(x,y,z,u)$
of points in a set
$X$
is said to be {\em admissible} if no entry occurs three or
four times in 
$Q$.
Two metrics 
$d$, $d'$
on 
$X$ 
are {\em M\"obius equivalent} if for any admissible quadruple
$Q=(x,y,z,u)\sub X$
the respective {\em cross-ratio triples} coincide,
$\crt_d(Q)=\crt_{d'}(Q)$,
where
$$\crt_d(Q)=(d(x,y)d(z,u):d(x,z)d(y,u):d(x,u)d(y,z))\in\R P^2.$$
We actually consider {\em extended} metrics on
$X$
for which existence of an {\em infinitely remote} point
$\om\in X$
is allowed, that is,
$d(x,\om)=\infty$
for all
$x\in X$, $x\neq\om$.
We always assume that such a point is unique if exists, and that
$d(\om,\om)=0$.
We use notation
$X_\om:=X\sm\om$
and the standard conventions for the calculation with 
$\om=\infty$.
If 
$\infty$ 
occurs once in 
$Q$, 
say 
$u=\infty$,
then
$\crt_d(x,y,z,\infty)=(d(x,y):d(x,z):d(y,z))$.
If 
$\infty$ 
occurs twice, say 
$z=u=\infty$, 
then
$\crt_d(x,y,\infty,\infty)=(0:1:1)$.

A {\em M\"obius structure} on a set
$X$
is a class 
$\cM=\cM(X)$
of metrics on
$X$
which are pairwise M\"obius equivalent.

The topology considered on 
$(X,d)$ 
is the topology with the basis consisting of all open distance balls 
$B_r(x)$
around points in 
$x\in X_\om$ 
and the complements $X\sm D$ 
of all closed distance balls 
$D=\overline{B}_r(x)$. 
M\"obius equivalent metrics define
the same topology on
$X$.
When a M\"obius structure 
$\cM$
on
$X$
is fixed, we say that
$(X,\cM)$
or simply
$X$
is a {\em M\"obius space.}

A map
$f:X\to X'$
between two M\"obius spaces
is called {\em M\"obius}, if 
$f$ 
is injective and for all admissible quadruples
$Q\sub X$
$$\crt(f(Q))=\crt(Q),$$
where the cross-ratio triples are taken with respect to
some (and hence any) metric of the M\"obius structures
of
$X$, $X'$.
M\"obius maps are continuous. If a M\"obius map
$f:X\to X'$
is bijective, then 
$f^{-1}$
is M\"obius,
$f$
is homeomorphism, and the M\"obius
spaces
$X$, $X'$
are said to be {\em M\"obius equivalent}.

In general different metrics in a M\"obius structure
$\cM$ 
can look very different. However if two metrics have the same infinitely remote
point, then they are homothetic. Since this result is crucial 
for our considerations, we state it as a lemma. 

\begin{lem}\label{lem:homothety_infinite}
Let 
$\cM$ be a M\"obius structure on a set
$X$, 
and let
$d$, $d'\in\cM$
have the same infinitely remote point
$\om\in X$.
Then there exists 
$\la>0$,
such that
$d'(x,y)=\la d(x,y)$ 
for all 
$x$, $y\in X$.
\end{lem}

\begin{proof}
Since otherwise the result is trivial, we can assume that
there are distinct points
$x$, $y\in X_\om$.
Take 
$\la>0$ 
such that
$d'(x,y)=\la d(x,y)$.
If 
$z\in X_\om$, 
then
$\crt_d(x,y,z,\om)=\crt_{d'}(x,y,z,\om)$, 
hence
$(d'(x,y):d'(x,z):d'(y,z))=(d(x,y):d(x,z):d(y,z)).$
Since 
$d'(x,y)=\la d(x,y)$
we therefore obtain
$d'(x,z)=\la d(x,z)$ 
and
$d'(y,z)=\la d(y,x)$.
\end{proof}

In what follows we always consider
$X_\om=X\sm\om$
as a metric space with a metric from the M\"obius structure
for which the point 
$\om$
is infinitely remote.

A classical example of a M\"obius space is the extended
$\wh\R^n=\R^n\cup\infty=S^n$, $n\ge 1$,
where the M\"obius structure is generated by some extended
Euclidean metric on
$\wh\R^n$,
and
$\R^n\cup\infty$
is identified with the unit sphere
$S^n\sub\R^{n+1}$
via the stereographic projection. Note that Euclidean metrics 
which are not homothetic to each other generate different 
M\"obius structures by the lemma above, which however are M\"obius equivalent.

\subsection{Ptolemy spaces}
\label{subsect:Ptolemy_spaces}

A M\"obius space
$X$
is called a {\em Ptolemy space}, if it satisfies the
Ptolemy property, that is, for all admissible quadruples  
$Q\sub X$
the entries of the respective cross-ratio triple
$\crt(Q)\in\R P^2$
satisfies the triangle inequality.
We can reformulate this property as follows.

Let 
$\Si$ 
be the subset of points 
$(a:b:c)\in \R P^2$, 
where
all entries 
$a$, $b$, $c$ 
are nonnegative or all entries are nonpositive. Note that 
$\Si$ 
can be identified with the standard 2-simplex,
$\set{(a,b,c)\in\R^3}{$a,b,c\ge 0,\, a+b+c=1$}$.
Let 
$\De\sub\Si$ 
be the set of points 
$(a:b:c)\in\Si$ 
such that the entries 
$a$, $b$, $c$ 
satisfy the triangle inequality. This is obviously well defined.
If we identify 
$\Si\sub\R P^2$ 
with the standard 2-simplex, i.e. the convex hull
of the unit vectors 
$e_1$, $e_2$, $e_3$, 
then 
$\De$ 
is the convex subset spanned by 
$(0,\frac{1}{2},\frac{1}{2})$,
$(\frac{1}{2},0,\frac{1}{2})$ 
and
$(\frac{1}{2},\frac{1}{2},0)$. 

The importance of the Ptolemy property comes from the following fact.
Given a metric
$d\in\cM(X)$
possibly with infinitely remote point
$\om\in X$ 
and a point
$z\in X_\om$,
the {\em metric inversion}, or m-inversion for brevity, of
$d$
of radius
$r>0$
with respect to
$z$
is a function
$d_z(x,y)=\frac{r^2d(x,y)}{d(z,x)d(z,y)}$
for all
$x$, $y\in X$
distinct from
$z$, $d_z(x,z)=\infty$
for all
$x\in X\sm\{z\}$
and
$d_z(z,z)=0$.
Using the standard convention we also have
$d_z(x,\om)=\frac{r^2}{d(x,z)}$.
A direct computation shows that
$d_z$
is M\"obius equivalent to
$d$.

\begin{rem}\label{rem:radius_one_inversion} When saying
about an m-inversion of a metric without specifying its radius, 
we mean that the radius is 1.
\end{rem}

In general
$d_z$
is not a metric because the triangle inequality may not
be satisfied. However, we have

\begin{pro}\label{pro:moeb_ptolemy}
A M\"obius structure
$\cM$ 
on a set
$X$ 
is Ptolemy if and only if for all
$z\in X$ 
there exists a metric
$d_z\in\cM$ 
with infinitely remote point
$z$. 
\end{pro}

\begin{proof}
Assume that 
$\cM$ 
is Ptolemy and that
$z\in X$.
Choose some 
$d\in\cM$.
If 
$z$
is infinitely remote with respect to
$d$
then 
$d$
is our desired metric. If not we define
$d_z$
as the m-inversion (of radius
$r=1$)
of
$d$
with respect to
$z$.
Since for 
$x$, $y$, $u\in X\sm z$
$$(d_z(x,y):d_z(y,u):d_z(x,u))=\crt_{d_z}(x,y,z,u)
=\crt_d(x,y,z,u)\in\De$$
we see that 
$d_z$
satisfies the triangle inequality and hence
$d_z\in\cM$. 

If on the other hand for every
$z\in X$
there is a metric
$d_z\in\cM$ 
with infinitely remote point
$z$,
then for all
$x$, $y$, $u\in X\sm z$
and all 
$d\in\cM$
$$\crt_d(x,y,z,u)=\crt_{d_z}(x,y,z,u)
=(d_z(x,y):d_z(y,u):d_z(x,u))\in\De,$$
which implies the Ptolemy property.
\end{proof}

The classical example of Ptolemy space is
$\wh\R^n$
with a standard M\"obius structure as it follows from 
the proposition above. Here is the list some known results 
on metric spaces with Ptolemy property.
A real normed vector space, which is ptolemaic, is an inner product space
(Schoenberg, 1952, \cite{Sch});
a Riemannian locally ptolemaic space is nonpositively curved (Kay, 1963, \cite{Kay});
all Bourdon and Hamenst\"adt metrics on 
$\di Y$,
where
$Y$
is CAT($-1$), generate a Ptolemy space (Foertsch-Schroeder, 2006, \cite{FS1});
a geodesic metric space is CAT(0) if and only if it is ptolemaic and 
Busemann convex, a ptolemaic proper geodesic metric space is uniquely geodesic 
(Foertsch-Lytchak-Schroeder, 2007, \cite{FLS}); any Hadamard space ptolemaic, 
a complete Riemannian manifold is ptolemaic if and only if it is a Hadamard manifold, 
a Finsler ptolemaic manifold is Riemannian (Buckley-Falk-Wraith, 2009, \cite{BFW}). 
These results allow to suggest that the Ptolemy property is a sort
of a M\"obius invariant nonpositive curvature condition.

\subsection{Circles in Ptolemy spaces}

A Ptolemy circle in a Ptolemy space 
$X$ 
is a subset 
$\si\sub X$ 
homeomorphic to 
$S^1$  
such that for every quadruple
$(x,y,z,u)\in\si$
of distinct points the equality 
\begin{equation}\label{eq:PT_eq}
d(x,z)d(y,u)=d(x,y)d(z,u)+d(x,u)d(y,z)
\end{equation}
holds for some and hence for any metric 
$d$
of the M\"obius structure , where it is supposed that the pair
$(x,z)$
separates the pair
$(y,u)$,
i.e.
$y$
and
$u$
are in different components of
$\si\sm\{x,z\}$.
Recall the classical Ptolemy theorem that four points 
$x$, $y$, $z$, $u$ 
of the Euclidean plane lie on a circle (in this order) if and
only if their distances satisfy
the Ptolemy equality (\ref{eq:PT_eq}).
One can reformulate this via the cross ratio triple. A subset
$\si$ 
homeomorphic to
$S^1$ 
is a Ptolemy circle, if and only if for all admissible quadruples 
$(x,y,z,u)$ 
of points in
$\si$ 
we have
$\crt(x,y,z,u)\in\d\De$,
where the set 
$\De$
is defined in sect.~\ref{subsect:Ptolemy_spaces}.

Let
$\si$ 
be a Ptolemy circle passing through the infinitely remote point
$\om$
for some metric
$d\in\cM$
and let
$\si_\om=\si\sm\om$.
Then 
$\crt(x,y,z,\om)\in\d\De$
says that for 
$x$, $y$, $z\in\si_\om$
(in this order)
$d(x,y)+d(y,z)=d(x,z)$,
i.e. it implies that
$\si_\om$
is a geodesic, actually a complete geodesic isometric to
$\R$.

We recall the following facts from \cite{FS2}. Let 
$\si$
be a Ptolemy circle in a Ptolemy space and
let
$x_1$, $x_2$, $x_3\in\si$
be distinct points, then the map
$\si\to\d\De$, $t\mapsto\crt(x_1,x_2,x_3,t)$
is a homeomorphism. The inverse of this map
gives a canonical parametrization of
$\si$ 
(for given points
$x_1$, $x_2$, $x_3\in\si$).
By composing two of these canonical parameterizations we have:

\begin{pro} \label{pro:moebch-circ}
Let 
$\si$ 
and 
$\si'$ 
be Ptolemy circles. Let 
$x_1$, $x_2$, $x_3$ 
and 
$x'_1$, $x'_2$, $x'_3$ 
be distinct points on 
$\si$ 
respectively on 
$\si'$.
Then there exists a unique M\"obius homeomorphism
$\varphi:\si\to \si'$ 
with 
$\varphi(x_i)=x'_i$.
\qed
\end{pro}

In particular all Ptolemy circles are M\"obius equivalent.
The standard metric models of a circle are 
$(\wh\R,d)$,
where 
$d$
is the standard Euclidean metric, or
$(S^1,d_c)$,
where
$d_c$ 
is the chordal metric on 
$S^1$,
i.e. the metric induced by the standard embedding
$S^1\sub\R^2$
as a unit circle. These two standard realizations of a circle
are M\"obius equivalent via the  stereographic projection.
Note that by Lemma~\ref{lem:homothety_infinite} there is up to homothety only 
one metric on a circle with a infinitely remote point, while there are plenty 
of bounded metrics (for a description of all Ptolemy metrics on
$S^1$ 
see \cite{FS2}).

\section{Properties of Ptolemy spaces}
\label{sect:properties_ptolemy}

In this section we discuss various properties of Ptolemy spaces
which include duality between Busemann and distance functions,
and Busemann flat Ptolemy spaces.

\subsection{Duality between Busemann and distance functions}
\label{subsect:duality_dist_busemann}

Let
$X$ 
be a Ptolemy space,
$d$
a metric of the M\"obius structure with infinitely remote point
$\om$, $X_\om=X\sm\om$.
Recall that every Ptolemy circle
$\si\sub X$
that passes through
$\om$
is isometric w.r.t.
$d$
to a geodesic line. Such a line
$l=\si_\om$ 
is called a {\em Ptolemy} line. We fix
$\om'\in l$,
and let
$d'$
be the m-inversion of
$d$
w.r.t.
$\om'$.
Then
$d'$
is a metric of the M\"obius structure with infinitely remote point
$\om'$.
In particular,
$l'=\si_{\om'}$
is a Ptolemy line in
$X_{\om'}$.
One easily checks that
$d$
is the m-inversion of
$d'$
w.r.t.
$\om$,
that is, the inversion operation (of radius 1) is involutive.

With every oriented Ptolemy line 
$l\sub X_\om$
and every point 
$\om'\in l$
we associate a function
$b:X_\om\to\R$,
called a {\em Busemann function} of
$l$, 
as follows. Given
$x\in X_\om$,
the difference
$|xy|-|\om' y|$
is nonincreasing by triangle inequality as
$y\in l$
goes to infinity according the orientation of
$l$, $y>\om'$.
Thus the limit
$b(x)=\lim_{l\ni y\to\infty}(|xy|-|\om'y|)$
exists. Note that
$b(\om')=0$
and
$b(x)=-|\om'x|$
for all
$l\ni x>\om'$.

For any Ptolemy space
$X$
there is a remarkable duality between Busemann and distance functions
which is described as follows. Let
$c:\R\to X_{\om'}$
be a unit speed parameterization of
$l'$
with
$c(0)=\om$, $b^\pm:X_\om\to\R$
the {\em opposite} Busemann functions of
$l$,
that is, associated with opposite ends of
$l$,
which are normalized by
$b^\pm(\om')=0$
and
$b^+\circ c(t)<0$
for
$t>0$, $b^-\circ c(t)<0$
for
$t<0$.
Since
$d(x,\om')\cdot d'(x,\om)=1$
for every
$x\in X\sm\{\om,\om'\}$,
we have
$b^\pm\circ c(t)=\mp 1/t$
for all
$t\neq 0$.

\begin{pro}\label{pro:duality_dist_busemann} For all
$x\in X\sm\{\om,\om'\}$
we have
\begin{equation}\label{eq:duality}
b^\pm(x)=\frac{d^\pm}{dt}\ln d'(x,c(t))|_{t=0},
\end{equation} 
where
$\frac{d^\pm}{dt}$
is the right/$-$left derivative.
\end{pro}

\begin{proof} We first note that the function
$t\mapsto d'(x,c(t))$
is convex by the Ptolemy condition, and thus it has
the right and the left derivatives at every point. Hence, the
right hand side of Equation~(\ref{eq:duality}) is well
defined. By definition,
$d(x,y)=\frac{d'(x,y)}{d'(\om,x)d'(\om,y)}$
and
$d(x,\om')=\frac{1}{d'(x,\om)}$
for all
$x$, $y\in X_\om$.
Now, we compute 
\begin{eqnarray*}
d(x,c(t))-d(\om',c(t))&=&\frac{d'(x,c(t))}{d'(\om,x)d'(\om,c(t))}
  -\frac{1}{d'(\om,c(t))}\\
&=&\frac{1}{|t|d'(x,c(0))}
   \left(d'(x,c(t))-d'(x,c(0)\right)
\end{eqnarray*}
for all
$t\neq 0$,
because
$d'(x,\om)=d'(x,c(0))$
and
$d'(\om,c(t))=|t|$.
Since
$b^\pm(x)=\lim_{t\to\pm 0}(d(x,c(t))-d(\om',c(t)))$,
we obtain
$$b^\pm(x)=\frac{d^\pm}{dt}\ln d'(x,c(t))|_{t=0}.$$
\end{proof}

Given a Ptolemy circle
$\si\in X$
and distinct points 
$\om$, $\om'\in\si$,
we denote with
$D_{\si,\om}^{\om'}$
the subset in
$X_{\om'}$
which consists of all
$x$
such that
$\om$
is a closest to 
$x$
point in the geodesic line
$\si_{\om'}$
(w.r.t. the metric of
$X_{\om'}$).

\begin{lem}\label{lem:omega_closest_subset} Let
$X$
be a Ptolemy space. Then for every Ptolemy circle
$\si\sub X$
and each pair of distinct points
$\om$, $\om'\in\si$
we have
\begin{equation}\label{eq:dist_busemann}
D_{\si,\om}^{\om'}\cup\om'=B_{\si,\om'}^\om\cup\om,
\end{equation}
where
$B_{\si,\om'}^\om=\set{x\in X_\om}{$b^+(x)\ge 0\ \text{and}\ b^-(x)\ge 0$}$,
$b^\pm:X_\om\to\R$
are the opposite Busemann functions of the Ptolemy line
$\si_\om\sub X_\om$
with
$b^\pm(\om')=0$.
\end{lem}

\begin{proof} Denote with
$d'$
the metric of
$X_{\om'}$
and let
$c:\R\to X_{\om'}$
be the unit speed parameterization of the Ptolemy line
$\si_{\om'}\sub X_{\om'}$
such that
$c(0)=\om$
and
$b^\pm\circ c(t)=\mp 1/t$,
see the paragraph preceding Proposition~\ref{pro:duality_dist_busemann}.
For every
$x\in D_{\si,\om}^{\om'}$
we have 
$\frac{d^+}{dt}d'(x,c(t))_{t=0}\ge 0$
for the right derivative, and 
$-\frac{d^-}{dt}d'(x,c(t))_{t=0}\le 0$
for left derivative because
$t=0$
is a minimum point of the convex function
$t\mapsto d'(x,c(t))$.
Equation~(\ref{eq:duality}) implies that 
$x\in B_{\si,\om'}^\om$.

Assume that
$b^+(x)\ge 0$
and
$b^-(x)\ge 0$
for some
$x\in X\sm\{\om,\om'\}$.
Equation~(\ref{eq:duality}) implies that the right derivative
$\frac{d^+}{dt}d'(x,c(t))_{t=0}\ge 0$
and the left derivative
$-\frac{d^-}{dt}d'(x,c(t))_{t=0}\le 0$.
Thus
$t=0$
is a minimum point of the convex function
$t\mapsto d'(x,c(t))$
and hence
$x\in D_{\si,\om}^{\om'}$.
\end{proof}

\subsection{Busemann flat Ptolemy spaces}
\label{subsect:busemann_flat_ptolemy}

A Ptolemy space
$X$
is said to be {\em (Busemann) flat} if for every Ptolemy
circle
$\si\sub X$
and every point
$\om\in\si$,
we have
\begin{equation}\label{eq:busemann_flat}
 b^++b^-\equiv\const
\end{equation}
for opposite Busemann functions
$b^\pm:X_\om\to\R$
associated with Ptolemy line
$\si_\om$.
This property is equivalent to that any horospheres of
$b^+$, $b^-$
coincide whenever they have a common point. Thus the horosphere
$H_{\si,\om'}^\om\sub X_\om$
of
$\si_\om$
through
$\om'\in\si_\om$
is well defined in a flat Ptolemy space.

\begin{pro}\label{pro:busemann_flat} A Ptolemy space
$X$
is flat if and only for every
$\om\in X$
and every
$x\in X_\om$
the distance function
$d(x,\cdot)$
is $C^1$-smooth along any Ptolemy line
$l\sub X_\om$, $l\not\ni x$.
\end{pro}

\begin{proof} Assume that distance functions are $C^1$-smooth
along Ptolemy lines. We fix
$\om\in X$, 
a Ptolemy line
$l\sub X_\om$,
and let
$b^\pm$
be opposite Busemann functions of
$l$.
We suppose W.L.G. that
$b^\pm(\om')=0$
for some point
$\om'\in l$.
Then
$b^++b^-=0$
along
$l$.
Equation~(\ref{eq:duality}) implies that in fact
$b^+(x)+b^-(x)=0$
for every
$x\in X_\om$. 
Thus
$X$
is flat.

Conversely, assume that 
$X$
is flat. Given
$\om'\in X$,
a Ptolemy line
$l'\in X_{\om'}$
and
$x\in X_{\om'}\sm l'$,
we show that the distance function
$d'(x,\cdot)$
in
$X_{\om'}$
is $C^1$-smooth along
$l'$
at every point
$\om\in l'$.

Let
$c:\R\to X_{\om'}$
by a unit speed parameterization of
$l'$
with
$c(0)=\om$, $b^\pm:X_\om\to\R$
the opposite Busemann function associated 
with the Ptolemy line
$l=(l'\cup\om')\sm\om\sub X_\om$
such that
$b^\pm(\om')=0$, $b^+\circ c(t)<0$
for all
$t>0$.
Then 
$b^++b^-\equiv 0$
by the assumption, and
by Proposition~\ref{pro:duality_dist_busemann} we have 
$\frac{d^+}{dt}d'(x,c(t))|_{t=0}=-\frac{d^-}{dt}d'(x,c(t))|_{t=0}$,
where
$\frac{d^+}{dt}$
is the right derivative
and
$-\frac{d^-}{dt}$
is the left derivative. Hence
$d'(x,\cdot)$
is $C^1$-smooth.
\end{proof}

By Proposition~\ref{pro:busemann_flat}, the duality equation~(\ref{eq:duality}) 
in a flat Ptolemy space
$X$
takes the following form
\begin{equation}\label{eq:smooth_duality}
b^\pm(x)=\pm\frac{d}{dt}\ln d'(x,c(t))|_{t=0}.
\end{equation}

\begin{exa}\label{exa:hyp} The Ptolemy space
$\wh\hyp^n$, $n\ge 2$,
generated by the real hyperbolic space
$\hyp^n$,
is not flat because the equality
$b^++b^-\equiv\const$
is violated in
$\hyp^n$.
(Recall that
$\hyp^n$
possesses the Ptolemy property and thus it generates
a Ptolemy space by taking all metrics on
$\wh\hyp^n$
which are M\"obius equivalent to the metric of
$\hyp^n$.)
Note that the distance function
$d(x,\cdot):\hyp^n\to\R$
is smooth for every
$x\in\hyp^n$
along any geodesic line
$l$, $x\not\in l\sub\hyp^n$.
This does not contradict Proposition~\ref{pro:busemann_flat} 
because the m-inversion of
$d$
with respect to any point
$x\in\hyp^n$
has a singularity at the infinity point of
$\wh\hyp^n$.
\end{exa}

In flat Ptolemy spaces, the duality between distance and 
Busemann functions is as follows.

\begin{lem}\label{lem:flat_duality} Let
$X$
be a flat Ptolemy space,
$\si\sub X$
a Ptolemy circle, and
$\om$, $\om'\sub\si$
distinct points. Let
$H_{\si,\om'}^\om\sub X_\om$
be the horosphere through
$\om'$
of the Ptolemy line
$\si_\om\sub X_\om$,
$D_{\si,\om}^{\om'}\sub X_{\om'}$
the set of all
$x\in X_{\om'}$
such that
$\om$
is the closest to
$x$
point in the Ptolemy line
$\si_{\om'}$.
Then
\begin{equation}\label{eq:horosphere_distance}
H_{\si,\om'}^\om\cup\om=D_{\si,\om}^{\om'}\cup\om'.
\end{equation}
\end{lem}

\begin{proof} In a flat Ptolemy space we have
$H_{\si,\om'}^\om=B_{\si,\om'}^\om$
because level sets of opposite Busemann functions associated with 
a Ptolemy line coincide when they have a common point. On the other hand, by duality, Lemma~\ref{lem:omega_closest_subset}, we have
$B_{\si,\om'}^\om\cup\om=D_{\si,\om}^{\om'}\cup{\om'}$.
\end{proof}

\section{Ptolemy spaces with circles and many space inversions}
\label{sect:many_circles_auto}

We begin this section with discussion of what is a space
inversion of an arbitrary Ptolemy space.

\subsection{Space inversions}
\label{subsect:space_inversions}

A M\"obius automorphism
$\phi:X\to X$
of a Ptolemy space induces a map
$\phi^\ast:\cM\to\cM$, $(\phi^\ast d)(x,y)=d(\phi(x),\phi(y))$
for every metric
$d\in\cM$
and each
$x$, $y\in X$,
where
$\cM$
is the M\"obius structure of
$X$.
Note that a metric inversion of a bounded metric 
cannot be induced by any M\"obius automorphism
$X\to X$,
because a metric inversion w.r.t.
$\om\in X$
has
$\om$
as the infinitely remote point.

Given distinct
$\om$, $\om'\in X$,
we say that a subset
$S\sub X$
is a {\em metric sphere between}
$\om$, $\om'$,
if
$$S=\set{x\in X}{$d(x,\om)=r$}=S_r^d(\om)$$
for some metric
$d\in\cM$
with infinitely remote point 
$\om'$
and some
$r>0$.
Recall that any two such metrics
$d$, $d'\in\cM$
are proportional to each other,
$d'=\la d$
for some 
$\la>0$,
see Lemma~\ref{lem:homothety_infinite}. Then
$S_r^d(\om)=S_{\la r}^{d'}(\om)$.
Moreover, this notion is symmetric w.r.t. 
$\om$, $\om'$,
because any metric 
$d'\in\cM$
with infinitely remote point 
$\om$
is proportional to the m-inversion of
$d$
w.r.t.
$\om$,
and we can assume that
$d'$
is the m-inversion itself. Then
$S=\set{x\in X}{$d'(x,\om')=1/r$}$.

We define a {\em space inversion}, or s-inversion for brevity,
w.r.t. distinct
$\om$, $\om'\in X$
and a metric sphere 
$S\sub X$
between
$\om$, $\om'$
as a M\"obius automorphism
$\phi=\phi_{\om,\om',S}:X\to X$
such that
\begin{itemize}
 \item[(1)] $\phi$
is an involution,
$\phi^2=\id$,
without fixed points;
 \item[(2)] $\phi(\om)=\om'$ (and thus
$\phi(\om')=\om$);
 \item[(3)] $\phi$
preserves
$S$,
$\phi(S)=S$;
 \item[(4)] $\phi(\si)=\si$
for any Ptolemy circle
$\si\sub X$
through
$\om$, $\om'$.
\end{itemize}

\begin{rem}\label{rem:sinversion_motivation} Motivation of this
definition comes from the fact that in the case
$X=\di M$,
where
$M$
is a symmetric rank one space of noncompact type, any central symmetry
$f:M\to M$
with a center 
$o\in M$, $f(o)=o$,
induces a space inversion
$\di f=\phi_{\om,\om',S}:X\to X$,
where a geodesic line 
$l=\om\om'\sub Y$
with the end points
$\om$, $\om'$
passes through
$o$,
and
$S\sub X$
is a metric sphere between
$\om$, $\om'$,
see Proposition~\ref{pro:rank_one_basic_axioms}.
\end{rem}

\begin{rem}\label{rem:weak_unique} In general, there is no reason 
that an s-inversion
$\phi=\phi_{\om,\om',S}$
is uniquely determined by its data
$\om$, $\om'$, $S$.
However, if 
$\phi'$
is another s-inversion with the same data, then
it coincides with 
$\phi$
along any Ptolemy circle through
$\om$, $\om'$
because any M\"obius automorphism of a Ptolemy circle
is uniquely determined by values at three distinct points,
see Proposition~\ref{pro:moebch-circ}.
\end{rem}

\begin{lem}\label{lem:sinversion_minversion} Given distinct
$\om$, $\om'\in X$
and a metric sphere
$S\sub X$
between
$\om$, $\om'$,
for any metric 
$d\in\cM$
with infinitely remote point 
$\om'$, 
an s-inversion
$\phi=\phi_{\om,\om',S}$
induces the m-inversion of
$d$
w.r.t.
$\om$
of radius
$r=r(d)>0$, $(\phi^\ast d)(x,y)=\frac{r^2d(x,y)}{d(x,\om)d(y,\om)}$,
where 
$r$
is determined by
$S=S_r^d(\om)$,
and
$x$, $y\in X$
are not equal to
$\om$
simultaneously. The similar property holds true
for any metric 
$d'\in\cM$
with the infinitely remote point
$\om$.
\end{lem}

\begin{proof} Since
$\phi(\om)=\om'$,
the point
$\om$
is infinitely remote for the metric
$\phi^\ast d$.
Thus
$\phi^\ast d=\la d'$
for some 
$\la>0$,
where
$d'$
is the m-inversion of
$d$
w.r.t.
$\om$,
$$(\phi^\ast d)(x,y)=\frac{\la d(x,y)}{d(x,\om)d(y,\om)}$$
for each
$x$, $y\in X$
which are not equal to
$\om$
simultaneously. We compute
$\la$
by taking
$x\in S$, $y=\phi(x)$.
Then
$\phi(y)=x$
by (1), and since
$(\phi^\ast d)(x,y)=d(x,y)$, $d(x,\om)=r=d(y,\om)$,
we have
$\la=r^2$. 
\end{proof}

Contrary to metric inversions which always exist, in general 
there is no reason for a space inversion to exist. If however
an s-inversion
$\phi=\phi_{\om,\om',S}:X\to X$
exists, and
$S=S_r^d(\om)$
for a metric 
$d\in\cM$
with infinitely remote point 
$\om'$,
then
$S=S_r^{\phi^\ast d}(\om')$.
This follows from Lemma~\ref{lem:sinversion_minversion}. Moreover, 
Lemma~\ref{lem:sinversion_minversion} implies that
$\phi^\ast(\phi^\ast d)=d$
for any metric 
$d\in\cM$
with infinitely remote point
$\om$
or
$\om'$.
Thus the property (1) agrees with Lemma~\ref{lem:sinversion_minversion}, 
and actually (1) refines the property
$\phi^\ast(\phi^\ast d)=d$.

\begin{lem}\label{lem:sphere_sinversion} For any metric sphere
$S'\sub X$
between
$\om$, $\om'$, 
we have
$\phi(S')$
is a metric sphere between
$\om$, $\om'$
for every s-inversion
$\phi=\phi_{\om,\om',S}:X\to X$.
More precisely, if
$S=S_r^d(\om)$, $S'=S_{r'}^d(\om)$,
then
$\phi(S')=S_{r^2/r'}^d(\om)$. 
\end{lem}

\begin{proof} For every
$x\in S'$,
by Lemma~\ref{lem:sinversion_minversion} we have
$$d(\phi(x),\om)=(\phi^\ast d)(x,\om')
  =\frac{r^2}{d(x,\om)}=r^2/r',$$
hence the claim.   
\end{proof}

In support of Conjecture~\ref{con:boundary_rank_one} we prove the following theorem which recovers
some basic features of
$\di\K\hyp^n$.
To formulate it, we introduce another important property which is
useful for many things.

\noindent
(${\rm E}_2$) Extension: any M\"obius map between any  Ptolemy circles in
$X$
extends to a M\"obius automorphism of
$X$.

\begin{thm}\label{thm:basic_ptolemy} Let
$X$
be a compact Ptolemy space with properties 
($E$) and (I) (see sect.~\ref{sect:introduction}). Then 
$X$
is homeomorphic to a sphere
$S^n$
for some
$n\ge 1$,
possesses the extension property (${\rm E}_2$), and for every
$\om\in X$
there is a 1-Lipschitz submetry
$\pi_\om:X_\om\to B_\om$
with the base
$B_\om$
isometric to an Euclidean space
$\R^k$, $0<k\le n$,
such that any M\"obius automorphism
$\phi:X\to X$
with
$\phi(\om)=\om'$
induces a homothety
$\ov\phi:B_\om\to B_{\om'}$
with
$\pi_{\om'}\circ\phi=\ov\phi\circ\pi_\om$.

The fibers of
$\pi_\om$
also called
$\K$-lines
are homeomorphic to
$\R^p$
for some
$p\ge 0$, $k+p=n$,
and for them the following properties hold

\begin{itemize}
 \item[($1_\K$)] given a $\K$-line
$F\sub X_\om$
and
$x\in X\sm F$,
there is a unique Ptolemy line
$l\sub X_\om$
through
$x$
that intersects
$F$;
 \item[($2_\K$)] given distinct $\K$-lines
$F$, $F'\sub X_\om$
and two  Ptolemy line that intersect both
$F$, $F'$,
if any other $\K$-line
$F''\sub X_\om$
intersects one of the Ptolemy lines, then it
necessarily intersects the other.
\end{itemize}

Furthermore, if
$k=1$,
then
$X=\wh\R$
is a Ptolemy circle.
\end{thm}

\begin{rem}\label{rem:real_case} In the case
$p=0$
the space
$X$
from Theorem~\ref{thm:basic_ptolemy} is M\"obius
equivalent to
$\wh\R^n=\di\hyp^{n+1}$
with
$n=\dim X$.
This proves Conjecture~\ref{con:boundary_rank_one}
for real hyperbolic spaces.
\end{rem}

\begin{rem}\label{rem:submetry} Recall that a map
$f:X\to Y$
between metric spaces is called a {\em submetry}
if for every ball 
$B_r(x)\sub X$
of radius 
$r>0$
centered at
$x$
its image
$f(B_r(x))$
coincides with the ball
$B_r(f(x))\sub Y$. 
\end{rem}

In what follows, we always consider the weak topology on
the group 
$\aut X$
of M\"obius automorphisms of
$X$,
i.e. a sequence
$\phi_i\in\aut X$
converges to
$\phi\in\aut X$, $\phi_i\to\phi$,
if and only if
$\phi_i(x)\to\phi(x)$
for every 
$x\in X$.

\subsection{M\"obius automorphisms of $X$}
\label{subsect:moebius_automorphisms}

In this section we establish some important additional
properties of a Ptolemy space
$X$
which follow from (E) and (I).

Given two distinct points
$\om$, $\om'\in X$,
we denote with
$C_{\om,\om'}$
the set of all the Ptolemy circles
$\si\sub X$
through
$\om$, $\om'$,
and with
$\Ga_{\om,\om'}$
the group of M\"obius automorphisms
$\phi:X\to X$
such that
$\phi(\om)=\om$, $\phi(\om')=\om'$, $\phi(\si)=\si$
and
$\phi$
preserves an orientation of
$\si$ 
for every
$\si\in C_{\om,\om'}$.

\begin{pro}\label{pro:homothety_property} Any Ptolemy space
$X$ 
with properties (E) and (I) possesses the following property

\noindent
(H) Homothety: for each distinct
$\om$, $\om'\in X$
the group
$\Ga_{\om,\om'}$
acts transitively on every arc of
$\si\sm\{\om,\om'\}$
for every circle
$\si\in C_{\om,\om'}$.
\end{pro}

\begin{rem}\label{rem:homothety} If one of the points
$\om$, $\om'$
is infinitely remote for a metric 
$d$
of the M\"obius structure, then every 
$\ga\in\Ga_{\om,\om'}$
is a homothety w.r.t.
$d$.
This is why we use (H) for the notation of the property above.
\end{rem}

\begin{proof} We assume that 
$\om'$
is infinitely remote for a metric 
$d\in\cM$.
Then for any
$\si\in C_{\om,\om'}$
the curve
$\si_{\om'}=\si\sm\om'$
is a Ptolemy line w.r.t.
$d$,
and any
$\ga\in\Ga_{\om,\om'}$
acts on
$\si_{\om'}$
as a homothety preserving an orientation.

Composing s-inversions
$\phi=\phi_{\om,\om',S}$, $\phi'=\phi_{\om,\om',S'}$
of
$X$,
where  
$S$, $S'\sub X$
are spheres between
$\om$, $\om'$,
we obtain a M\"obius automorphism
$\ga=\phi'\circ\phi$
with properties
$\ga(\om)=\om$, $\ga(\om')=\om'$
and
$\ga(\si)=\si$
for any Ptolemy circle
$\si\in C_{\om,\om'}$.
Having no fixed point, both
$\phi$, $\phi'$
preserve orientations of
$\si$.
Hence, 
$\ga$
preserves its orientations, thus
$\ga$
acts on every arc of
$\si\sm\{\om,\om'\}$
as a homothety. That is,
$\ga\in\Ga_{\om,\om'}$.

Let
$r$, $r'>0$
be the radii of
$S$, $S'$
respectively w.r.t. the metric 
$d$, $S=S_r^d(\om)$, $S'=S_{r'}^d(\om)$.
Then for every
$x\in X\sm\{\om,\om'\}$
we have
$d(\phi(x),\om)=\frac{r^2}{d(x,\om)}$
and
$d(\ga(x),\om)=d(\phi'\circ\phi(x),\phi'(\om'))
=\frac{r'^2}{d(\phi(x),\om)}=(r'/r)^2d(x,\om)$.
Therefore, the dilatation coefficient of
$\ga$
equals
$\la:=(r'/r)^2$,
and it can be chosen arbitrarily by changing
$S$, $S'$
appropriately.
\end{proof}

\begin{cor}\label{cor:weak_unique} Any two distinct Ptolemy 
circles in a Ptolemy space with properties (E) and (I)
have in common at most two points.
\end{cor}

\begin{proof} Assume
$\om$, $\om'$, $x\in\si\cap\si'$
are distinct common points of Ptolemy circles
$\si$, $\si'\sub X$.
We have
$\ga(x)\in\si\cap\si'$
for every
$\ga\in\Ga_{\om,\om'}$.
Then by property (H), the arcs of
$\si$
and 
$\si'$
between
$\om$, $\om'$
which contain
$x$
coincide. Taking
$\om''$
inside of this common arc and applying the same
argument to
$\om'$, $\om''$, $x=\om$,
we obtain 
$\si=\si'$. 
\end{proof}

\subsection{Busemann parallel lines, pure homotheties and shifts}
\label{subsect:parallel_lines_pure_homothethies_shifts}

In this section we assume that the Ptolemy space
$X$
possesses the properties (E) an (I). A some point, we
also assume that
$X$
is compact.

We say that Ptolemy lines
$l$, $l'\sub X_\om$
are {\em Busemann parallel} if 
$l$, $l'$
share Busemann functions, that is, any Busemann function
associated with
$l$
is also a Busemann function associated with
$l'$
and vice versa.

\begin{lem}\label{lem:unique_line} Let
$l$, $l'\sub X_\om$
be Ptolemy lines with a common point,
$o\in l\cap l'$, $b:X_\om\to\R$
a Busemann function of
$l$
with
$b(o)=0$.
Assume 
$b\circ c(t)=-t=b\circ c'(t)$
for all 
$t\ge 0$
and for appropriate unit speed parameterizations
$c$, $c':\R\to X_\om$
of
$l$, $l'$
respectively
with
$c(0)=o=c'(0)$.
Then
$l=l'$.
In particular, Busemann parallel Ptolemy lines coincide if
they have a common point.
\end{lem}

\begin{proof} We show that the concatenation of
$c|(-\infty,0]$
with
$c'|[0,\infty)$
is also a Ptolemy line. Then
$l=l'$
by Corollary~\ref{cor:weak_unique}. It suffices to show	that for
$s$, $t\ge 0$
we have
$|c(-s)c'(t)|=t+s$.
By triangle inequality we have
$|c(-s)c'(t)|\le t+s$.
Letting 
$t_i\to \infty$ 
we have
$|c'(t)c(t_i)|-t_i\to b\circ c'(t)=-t$.
Thus by triangle inequality again, we have
$$|c(-s)c'(t)| \ge |c(-s)c(t_i)| - |c'(t)c(t_i)| = (t_i+s)-|c'(t)c(t_i)|\to t+s.$$
Thus
$|c(-s)c'(t)|=t+s$.
\end{proof}

Next, we show that a sublinear divergence of Ptolemy lines
is equivalent for them to be Busemann parallel.

\begin{lem}\label{lem:busparallel_sublinear}
If two Ptolemy lines 
$l$, $l'\sub X_\om$
are Busemann parallel, then
they diverge at most sublinearly, that is
$|c(t)c'(t)|/|t|\to 0$
as
$|t|\to\infty$
for appropriate unit speed parameterizations
$c$, $c'$
of 
$l$, $l'$.

Conversely, if
$|c(t_i)c'(t_i)|/|t_i|\to 0$
for some sequences
$t_i\to\pm\infty$,
then the lines
$l$, $l'$
are Busemann parallel.
\end{lem}

\begin{proof} Let
$c$, $c':\R\to X_\om$
be unit speed parameterizations of Busemann parallel lines
$l$, $l'\sub X_\om$
respectively, and a common Busemann function
$b:X_\om\to\R$
such that
$b\circ c(t)=b\circ c'(t)=-t$
for all
$t\in\R$.
Let
$\mu(t):=|c(t)c'(t)|$.
We claim that
$\mu(t)/|t|\to 0$
for
$t\to\pm\infty$.
Assume to the contrary, that W.L.G. there exists a sequence
$t_i\to\infty$ 
with
$\mu(t_i)/t_i\ge a >0$.

By the homothety property~(H) there exists a homothety
$\phi_i$
of
$X_\om$
with factor
$1/t_i$
such that
$\phi_i\circ c(s)=c(s/t_i)$
for all
$s\in\R$.
Note that
$c'_i(s)=\phi_i\circ c'(t_i s)$
is a unit speed parameterization of
the Ptolemy line
$\phi_i(l')$.
For fixed
$i$
we calculate
\begin{align*}
b\circ c'_i(t)) &= \lim_{s\to \infty}(|c'_i(t) c(s)|-s)
= \lim_{s\to \infty}(|\phi_i(c'(tt_i))c(s)|-s) \\
&= \lim_{s\to \infty}(|\phi_i(c'(tt_i))c(s/t_i)|-s/t_i)
= \lim_{s\to \infty}(|\phi_i(c'(tt_i))\phi_i(c(s))|-s/t_i)\\
&= \lim_{s\to \infty}\frac{1}{t_i}(|c'(tt_i)c(s)|-s)
=\frac{1}{t_i}b(c'(tt_i))=\frac{1}{t_i}(-tt_i )=-t
\end{align*}
for all 
$t\in\R$.
The Ptolemy lines
$\phi_i(l')$
subconverge to a Ptolemy line
$l''$
through
$c(0)$.
If
$c'':\R\to X$
is the limit unit speed parameterization of
$l''$,
then
$b\circ c''(t)=-t$
for all 
$t\in\R$,
and
$|c''(1)c(1)|\ge a>0$.
This contradicts Lemma \ref{lem:unique_line} by which
$l=l'$
and thus
$c''(t)=c(t)$
for all 
$t\in\R$.

Conversely, assume 
$c$, $c':\R\to X_\om$
are unit speed parameterizations of Ptolemy lines
$l$, $l'\sub X_\om$
with
$c(0)=o$, $c'(0)=o'$
such that
$b(o)=b(o')=0$
for the Busemann function
$b:X_\om\to\R$
of
$l$
with
$b\circ c(t)=-t$, $t\in\R$,
and
$\mu(t_i)/t_i\to 0$
for some sequence
$t_i\to\infty$,
where
$\mu(t)=|c(t)c'(t)|$.
Let
$b':X_\om\to\R$
be the Busemann function of
$l'$
with
$b'\circ c'(t)=-t$.
Applying the Ptolemy inequality to the cross-ration
triple
$\crt(Q_i)$
of the quadruple
$Q_i=(o,c(t_i),c'(t_i),o')$,
we obtain 
$$\left||oc'(t_i)||o'c(t_i)|-|oc(t_i)||o'c'(t_i)|\right|
  \le|oo'||c(t_i)c'(t_i)|.$$
Using
$|oc(t_i)|=t_i=|o'c'(t_i)|$,
$|o'c(t_i)|=b(o')+t_i+o(1)$,
$|oc'(t_i)|=b'(o)+t_i+o(1)$,
and
$|c(t_i)c'(t_i)|=\mu(t_i)=o(1)t_i$,
we obtain
$$|(b'(o)+t_i+o(1))(b(o')+t_i+o(1))-t_i^2|\le|oo'|o(1)t_i,$$
thus
$|b'(o)|\le o(1)$
and hence
$b'(o)=b(o')=0$.

Finally, for an arbitrary
$x\in X_\om$
consider the quadruple
$Q_{x,i}=(x,c(t_i),c'(t_i),o)$.
By the same argument as above, we have
$$\left||xc'(t_i)|t_i-|xc(t_i)||oc'(t_i)|\right|\le
  |ox|\mu(t_i).$$
Using
$|oc'(t_i)|=b'(o)+t_i+o(1)=t_i+o(1)$,
$|xc'(t_i)|=b'(x)+t_i+o(1)$,
$|xc(t_i)|=b(x)+t_i+o(1)$,
we finally obtain 
$|b'(x)-b(x)|\le o(1)$
and hence
$b(x)=b'(x)$.
Therefore, the lines
$l$, $l'$
are Busemann parallel.
\end{proof}

Now, we assume that our Ptolemy space
$X$
is compact. Given
$x$, $x'\in X_\om$,
we construct an isometry
$\eta_{xx'}:X_\om\to X_\om$
called a {\em shift} as follows. We take a sequence
$\la_i\to\infty$
and using the homothety property (H) for every 
$i$
consider homotheties
$\phi_i\in\Ga_{\om,x}$, $\psi_i\in\Ga_{\om,x'}$
with coefficient
$\la_i$.
Then
$\eta_i=\psi_i^{-1}\circ\phi_i$
is an isometry of
$X_\om$
for every
$i$
because the coefficient of the homothety
$\eta_i$
is 1. Furthermore, we have
$|\eta_i(x)x'|=\la_i^{-1}|xx'|\to 0$
as 
$i\to\infty$.
Since
$X$
is compact, the sequence
$\eta_i$
subconverges to an isometry
$\eta=\eta_{xx'}$
with
$\eta(x)=x'$.
The term shift for 
$\eta$
is justified by the following

\begin{lem}\label{lem:shift_busemann_parallel} A shift 
$\eta_{xx'}$
moves any Ptolemy line
$l$
through
$x$
to a Busemann parallel Ptolemy line 
$\eta_{xx'}(l)$
through
$x'$.
\end{lem}

\begin{proof} We show that the line 
$l'=\eta_{xx'}(l)$
cannot have a linear divergence with 
$l$.
Assume to the contrary that
$\mu(t)\ge at$
for some 
$a>0$
and all
$t>0$,
where
$\mu(t)=|c(t)c'(t)|$, $c:\R\to X_\om$
is a unit speed parameterization of
$l$
with 
$c(0)=x$, $c'=\eta_{xx'}\circ c$.

Recall that
$\eta_{xx'}=\lim\eta_i$,
where
$\eta_i=\psi_i^{-1}\circ\phi_i$,
and
$\phi_i\in\Ga_{\om,x}$, $\psi_i\in\Ga_{\om,x'}$
are homotheties with the same coefficient
$\la_i\to\infty$.
By definition of the groups
$\Ga_{\om,x}$, $\Ga_{\om,x'}$,
we have
$\phi_i(l)=l$, $\psi_i(l')=l'$.
We take
$y=c(1)$, $y'=c'(1)$.
Then for 
$y_i=\phi_i(y)=c(\la_i)$
we have
$|y_ic'(\la_i)|=\mu(\la_i)\ge a\la_i$.
Thus for 
$y_i'=\psi_i^{-1}(y_i)$
the estimate
$|y_i'y'|=|\psi_i^{-1}(y_i)\psi_i^{-1}\circ c'(\la_i)|\ge a$
holds for all 
$i$
in contradiction with 
$y_i'\to y'$
as
$i\to\infty$.

Therefore, there are sequences
$t_i\to\pm\infty$, 
with 
$\mu(t_i)=o(1)|t_i|$.
By Lemma~\ref{lem:busparallel_sublinear} the lines
$l$, $l'$
are Busemann parallel. 
\end{proof}

From Lemma~\ref{lem:unique_line} and 
Lemma~\ref{lem:shift_busemann_parallel} we immediately obtain

\begin{cor}\label{cor:busparallel_foliation} Given a Ptolemy line
$l\sub X_\om$,
through any point 
$x\in X_\om$
there is a unique Ptolemy line
$l(x)$
Busemann parallel to
$l$.
\qed
\end{cor}

Recall that any M\"obius map 
$\phi:X\to X$
with 
$\phi(\om)=\om$
for
$\om\in X$
acts on
$X_\om$
as a homothety. A homothety
$\phi:X_\om\to X_\om$
is said to be {\em pure} if it preserves any
foliation of
$X_\om$
by Busemann parallel Ptolemy lines.

\begin{lem}\label{lem:pure_homothety} For every
$o\in X_\om$
the group
$\Ga_{\om,o}$
consists of pure homotheties. In particular,
every shift of
$X_\om$
preserves any foliation of
$X_\om$
by Busemann parallel Ptolemy lines. 
\end{lem}

\begin{proof} Let
$l\sub X_\om$
be a Ptolemy line through
$o$, $b:X_\om\to\R$
a Busemann function of
$l$
with
$b(o)=0$.
Then
$b\circ\phi=\la b$
for every homothety
$\phi\in\Ga_{\om,o}$,
where
$\la>0$
is the coefficient of
$\phi$.
By Corollary~\ref{cor:busparallel_foliation}, any Busemann
function of any Ptolemy line 
$l(x)$
through
$x\in X_\om$
is a Busemann function of a line
$l$
through
$o$.
Therefore, every 
$\phi\in\Ga_{\om,o}$
preserves any Busemann function
$b$
of
$l(x)$
with
$b(o)=0$
in the sense that
$\la^{-1}b\circ\phi=b$,
where
$\la>0$
is the coefficient of
$\phi$. 
Since
$\la^{-1}b\circ\phi$
is a Busemann function of the Ptolemy line
$\phi^{-1}(l(x))$, 
we see that this line is Busemann parallel to
$l(x)$.
Thus
$\phi$
preserves the foliation
$l(x)$, $x\in X_\om$
by Busemann parallel Ptolemy lines. 
\end{proof}

A construction of a homothety from the group
$\Ga_{\om,\om'}$
given in Proposition~\ref{pro:homothety_property} is not
uniquely determined because to obtain a homothety
with the same coefficient 
$\la$
one can take a composition of different pairs of 
s-inversions. Thus for given
$x$, $x'\in X_\om$
a shift
$\eta_{xx'}$
is not uniquely determined. We give a refined construction
of shifts with property
$\eta_{xx'}\to\id$
as
$x\to x'$
which will be used in the proof of Lemma~\ref{lem:zigzag_change_basepoint}
below.

\begin{lem}\label{lem:shift_identity} For
$x$, $x'\in X_\om$
there is a shift
$\eta_{xx'}:X_\om\to X_\om$
with 
$\eta_{xx'}(x)=x'$
such that
$\eta_{xx'}\to\id$
as 
$x\to x'$.
\end{lem}

\begin{proof} For
$\la_i\to\infty$
we denote with
$S=S_1(x)$, $S_i=S_{\la_i}(x)$
the metric spheres in
$X_\om$
centered at
$x$
of radius 1 and
$\la_i$
respectively. Similarly we put
$S'=S_1(x')$, $S_i'=S_{\la_i}(x')$.
Then
$\phi_i=\phi_{x,\om,S_i}\circ\phi_{x,\om,S}\in\Ga_{\om,x}$, 
$\psi_i=\phi_{x',\om,S_i'}\circ\phi_{x',\om,S'}\in\Ga_{\om,x'}$
are homotheties with the same coefficient
$\la_i^2$,
see the proof of Proposition~\ref{pro:homothety_property}.
The sequence of isometries
$\eta_i=\psi_i^{-1}\circ\phi_i:X_\om\to X_\om$
converges to a shift
$\eta:X_\om\to X_\om$
with
$\eta(x)=x'$.
We have
$$\eta_i=\phi_{x',\om,S'}\circ\phi_{x',\om,S_i'}\circ
   \phi_{x,\om,S_i}\circ\phi_{x,\om,S}$$
and
$\phi_{x',\om,S_i'}\circ\phi_{x,\om,S_i}\to\id$,
$\phi_{x',\om,S'}\circ\phi_{x,\om,S}\to\id$
as
$x\to x'$
because
$S_i\to S_i'$, $S\to S'$
in the Hausdorff metric, every s-inversion is uniquely determined by
its data according to our assumption, and by 
Lemma~\ref{lem:sphere_sinversion}, s-inversions preserve the family
of metric spheres between data points. Thus
$\eta_i\to\id$
for every
$i$
as
$x\to x'$.
Moreover, the convergence of metric spheres around
$x$
to metric spheres around
$x'$
in the Hausdorff metric is uniform in radius as
$x\to x'$, $\Hd(S_r^d(x),S_r^d(x'))\le|xx'|$
for every
$r>0$.
Therefore, 
$\eta_i\to\id$
uniformly in
$i$
as
$x\to x'$.
This implies
$\eta\to\id$
as
$x\to x'$.
\end{proof}

\subsection{Enhancing the existence property (E)}
\label{subsect:enhancing_E}

By property (E) formulated in sect.~\ref{sect:introduction}
we know that the space
$X$
contains at least one Ptolemy circle.

\begin{pro}\label{pro:two_point_homogeneous} Any compact Ptolemy
space with the inversion property (I) is two-point homogeneous,
that is, for each (ordered) pairs
$(x,y)$, $(x',y')$
of distinct points in
$X$
there is a M\"obius automorphism
$f:X\to X$
with
$f(x)=x'$, $f(y)=y'$. 
\end{pro}

\begin{proof} Applying an inversion, we can map 
$x$
to
$x'$.
Let
$y''$
be the image of
$y$
under the inversion. Then
$y''\neq x'$
by the assumption. We consider a metric of the M\"obius
structure with infinitely remote point 
$x'$.
By discussion above, there is a shift w.r.t. that metric which maps 
$y''$
to
$y'$.
The resulting composition gives a required M\"obius
automorphism.
\end{proof}

It immediately follows from Proposition~\ref{pro:two_point_homogeneous}
that the property (E) in any compact Ptolemy space with (I) 
is promoted to

\noindent
(E) Enhanced existence: through any two points in
$X$
there is a Ptolemy circle.

In what follow, we use this property under the name (E).

\subsection{Busemann functions on $X_\om$}
\label{subsect:busemann_functions}

The proof of Theorem~\ref{thm:basic_ptolemy} is based on
study of Busemann functions on
$X_\om$.
In this section we assume that a compact Ptolemy space
$X$
possesses the properties (E) and (I).

\begin{lem}\label{lem:limit_circle} Assume that
$x_i\to x$
in
$X$,
and a point
$\om\in X$
distinct from
$x$
is fixed. Then any Ptolemy circle
$l\sub X$
through
$\om$, $x$
is the (pointwise) limit of a sequence of Ptolemy circles
$l_i\sub X$
through
$\om$, $x_i$.
\end{lem}

\begin{proof} In the space
$X_\om$
the circle
$l$
is a Ptolemy line (with infinitely remote point 
$\om$)
through
$x$.
Then the sequence
$l_i=\eta_i(l)$
of Ptolemy lines with
$x_i\in l_i$
converges to
$l$,
where
$\eta_i:X_\om\to X_\om$
is a shift with 
$\eta_i(x)=x_i$,
because the lines
$l_i$
are Busemann parallel to 
$l$
by Lemma~\ref{lem:pure_homothety}, and any sublimit
of the sequence
$\{l_i\}$
coincides with 
$l$
by Lemma~\ref{lem:unique_line}.
\end{proof}

We fix
$\om\in X$
and a metric 
$d$
of the M\"obius structure
such that
$\om$
is the infinitely remote point. Then every Ptolemy circle in
$X$
through
$\om$
is a Ptolemy line with respect to that metric. It immediately 
follows from the Ptolemy inequality that the distance function 
$d(z,\cdot)$
to a point 
$z\in X_\om$
is convex along any Ptolemy line, see
\cite{FS2}. 
Under the homothety property (H) we prove that
in fact 
$d(z,\cdot)$
is $C^1$-smooth.

\begin{lem}\label{lem:smooth_convex} In any compact Ptolemy space
$X$
with the homothety property~(H), the distance function
$d_z=d(z,\cdot):X_\om\to\R$
is convex and $C^1$-smooth along any Ptolemy line
$l\sub X_\om$
for any
$\om\in X$, $z\in X_\om\sm l$.
Therefore,
$X$
is Busemann flat.
\end{lem}

\begin{proof}
Assume that 
$d_z$
is not $C^1$-smooth at some point
$x\in l$.
We fix an arclength parameterization 
$c:\R\to X_\om$
of
$l$
such that
$x=c(0)$.
Since
$f=d_z\circ c$
is convex, it has the left and the right derivatives at every point.
By assumption, these derivatives are different at
$t=0$.
It follows that
$$\liminf_{t\to 0}\frac{f(t)-2f(0)+f(-t)}{t}>0.$$
Now, using property~(H), we find for every
$\la>0$
a homothety
$h_\la:X_\om\to X_\om$
with coefficient
$\la$
that preserves the point
$x$
and the Ptolemy line
$l$, $h_\la(x)=x$, $h_\la(l)=l$.
Then
$d(x,h_\la(z))\to\infty$
as
$\la\to\infty$
and thus
$\om_\la=h_\la(z)\to\om$.
By Lemma~\ref{lem:limit_circle}, there is a Ptolemy circle
$l_\la$
through
$x$, $\om_\la$
such that
$l_\la\to l$
as
$\la\to\infty$
(maybe after passing to a subsequence). We put
$\la=1/t$
and consider points
$x_t^+$, $x_t^-\in l_\la$
separated by
$x$
with
$d(x,x_t^\pm)=1$.
Then, W.L.G.,
$x_t^\pm\to c(\pm 1)$
as
$t\to 0$.
The points
$x_t^+$, $x$, $x_t^-$, $\om_\la$
lie on the Ptolemy circle
$l_\la$
(in this order), thus
$$2d(x,\om_\la)\ge d(x,\om_\la)d(x_t^+,x_t^-)=d(x_t^+,\om_\la)+d(x_t^-,\om_\la)$$
by the Ptolemy equality. On the other hand,
$f(0)/t=d(x,\om_\la)$
and
$f(\pm t)/t=d(c(\pm 1),\om_\la)$.
Thus
$|d(x_t^\pm,\om_\la)-f(\pm t)/t|\le d(x_t^\pm,c(\pm 1))\to 0$
as
$t\to 0$.
Therefore,
$(f(t)-2f(0)+f(-t))/t\to 0$
as
$t\to 0$
in contradiction with our assumption. Now, 
$X$
is flat by Proposition~\ref{pro:busemann_flat}.
\end{proof}

A similar idea is used in the proof of the following

\begin{pro}\label{pro:busemann_affine} Given two Ptolemy lines
$l$, $l'\in X_\om$,
the Busemann functions of
$l$
are affine functions on
$l'$.
\end{pro}

\begin{proof} Indeed let
$b$
be a Busemann function
of
$l$.
Thus we can write
$b(x)=\lim_{i\to\infty}(|x\om_i|-|o\om_i|)$
for every
$x\in X_\om$,
where
$o\in l$
is some fixed point,
$\om_i\in l$,
and
$\om_i\to\om$.

Let
$x$, $y$, $m\in l'$,
such that
$|xm|=|my|=\frac{1}{2}|xy|$.
We have to show that
$b(m)=\frac{1}{2}(b(x)+b(y))$.
Since Busemann functions are convex, see \cite{FS2},
we have
$b(m)\le \frac{1}{2}(b(x)+b(y))$.

By Lemma~\ref{lem:limit_circle} there
exists a Ptolemy circle
$l_i$
through the points
$x$
and
$\om_i$,
such that the sequence
$l_i$
converges pointwise to the Ptolemy line
$l'$.
Thus there are points
$y_i\in l_i$,
such that
$y_i\to y$.
The points
$x,y_i$
divide
$l_i$
into two segments.
Choose
a point
$m_i$
in the segment which does
not contain
$\om_i$
in a way such that
$|xm_i|=|m_iy_i|$.
One easily sees
that
$m_i\to m$.

Since
the points
$x,m_i,y_i,\om_i$
are on a Ptolemy circle
(in this order),
we have
$$|\om_im_i|\cdot|xy_i|=|\om_ix|\cdot|m_iy_i|+|\om_iy_i|\cdot|m_ix|$$
and since
$|xm_i|=|m_iy_i|\ge\frac{1}{2}|xy_i|$
we see
$|\om_im_i|\ge \frac{1}{2}(|\om_i x|+|\om_i y_i|)$.
This implies in the limit
$b(m)\ge \frac{1}{2}(b(x)+b(y))$.
\end{proof}

\begin{lem}\label{lem:geoconvex_horosphere} For any Busemann function
$b:X_\om\to\R$
of any Ptolemy line
$l\sub X_\om$,
every horosphere
$H_t=b^{-1}(t)$, $t\in\R$,
is geodesically convex, that is, any Ptolemy line
$l'\sub X_\om$
having two distinct points 
$z$, $z'$ 
in common with
$H_t$
is contained in
$H_t$, $l'\sub H_t$. 
\end{lem}

\begin{proof} We put
$b^+=b$
and assume W.L.G. that
$b^+(z)=0=b^-(z)$,
where the Busemann function 
$b^-$
of
$l$
is opposite to
$b^+$. 
Then
$b^++b^-\equiv 0$
because
$X$
is Busemann flat, see Lemma~\ref{lem:smooth_convex}. Thus
$H_0=(b^+)^{-1}(0)=(b^-)^{-1}(0)$
is a common horosphere for
$b^+$, $b^-$.
Since horoballs, i.e. sublevel sets of Busemann functions, 
are convex, the geodesic segment
$zz'\sub l$
lies in
$H_0$.

By Proposition~\ref{pro:busemann_affine}, the function
$b$
is affine along
$l'$,
that is, 
$b\circ c(t)=\al t+\be$
for any arclength parameterization 
$c:\R\to l'$
of
$l'$
and some
$\al$, $\be\in\R$, $|\al|\le 1$.
We choose
$c$
so that
$c(0)=z$, $c(|zz'|)=z'$.
Then
$\be=0$
by the assumption
$b(z)=0$,
and
$0=b(z')=b\circ c(|zz'|)=\al|zz'|$.
Hence
$\al=0$
and
$b|l'\equiv 0$.
This shows that
$l'\sub H_0$.
\end{proof}

\subsection{Slope of two Ptolemy lines}
\label{subsect:slope}

By Proposition~\ref{pro:busemann_affine}, a Busemann
function associated with a Ptolemy line is affine along
any other Ptolemy line. We introduce a quantity which measures
a mutual position of Ptolemy lines in the space.

Let
$l$, $l'\sub X_\om$
be oriented Ptolemy lines.
We define the {\em slope} of
$l'$
w.r.t.
$l$
as the coefficient of a Busemann function
$b$
associated with
$l$
when restricted to
$l'$, $\slope(l';l)=\al$
if and only if
$b\circ c'(t)=\al t+\be$
for some
$\be\in\R$
and all
$t\in\R$,
where
$c':\R\to X_\om$
is a unit speed parameterization of
$l'$
compatible with its orientation. The quantity
$\slope(l';l)\in[-1,1]$
is well defined, i.e. it depends of the choice neither the Busemann
function
$b$ 
nor the parameterization
$c'$
(we assume that
$b$
is defined via a parameterization of
$l$
compatible with its orientation).
Note that the slope changes the sign when the orientation of
$l$
or
$l'$
is changed,
$$\slope(-l';l)=-\slope(l';l)\quad\text{and}\quad
  \slope(l';-l)=-\slope(l';l).$$
The first equality is obvious, while the second one holds because
$X$
is Busemann flat by Lemma~\ref{lem:smooth_convex}. 

By definition, we have
$\slope(l;l)=-1$
for any oriented Ptolemy line
$l\sub X_\om$.
More generally, let
$l$, $l'\sub X_\om$
be Busemann parallel Ptolemy lines. If an orientation of
$l$
is fixed, then a {\em compatible} orientation of
$l'$
is well defined. Indeed, we take a Busemann function 
$b$
of
$l$
such that
$b\to -\infty$
along
$l$
in the chosen direction. Since
$b$
is also a Busemann function of
$l'$,
the respective direction of
$l'$
such that
$b\to-\infty$
along
$l'$
is well defined, and it is independent of the choice of
$b$.

Now, if orientations of
$l$, $l'$
are compatible, then 
$\slope(l';l)=-1=\slope(l;l')$.

\begin{lem}\label{lem:paraline_busemann} Let
$l$, $l'\sub X_\om$
be Busemann parallel Ptolemy lines with compatible orientations. Then 
for any oriented Ptolemy line
$l''\sub X_\om$
we have
$\slope(l;l'')=\slope(l';l'')$.
\end{lem}

\begin{proof}
Let
$b:X_\om\to\R$
be a Busemann function associated with
$l''$.
Since
$b$
is affine along Ptolemy lines, there are
unit speed parameterizations
$c:\R\to l$, $c':\R\to l'$
compatible with the orientations of
$l$, $l'$
such that
$b\circ c(0)=b\circ c'(0)=:\be$.
Then
$b\circ c(t)=\al t+\be$, $b\circ c'(t)=\al' t+\be$
for some
$|\al|$, $|\al'|\le 1$
and all
$t\in\R$.
We show that
$\al=\al'$.

Since the orientations of
$l$, $l'$
are compatible, we have
$|c(t)c'(t)|=o(1)|t|$
as
$|t|\to\infty$
by Lemma~\ref{lem:busparallel_sublinear}. Let
$c'':\R\to l''$
be a unit speed parameterization such that
$b(x)=\lim_{s\to\infty}|c''(s)x|-s$, $x\in X_\om$.
Since
$||c''(s)c(t)|-|c''(s)c'(t)||\le|c(t)c'(t)|=o(1)t$
as
$t\to\infty$,
we have
$|b\circ c(t)-b\circ c'(t)|=o(1)t$
and hence
$$\al=\lim_{t\to\infty}b\circ c(t)/t
=\lim_{t\to\infty}b\circ c'(t)/t=\al'.$$
\end{proof}

Proposition~\ref{pro:busemann_affine} combined with duality gives
rise to a first variation formula to describe which we use
the following agreement. Let
$\si$, $\si'\sub X$
be Ptolemy circles meeting each other at two distinct points
$\om$
and
$\om'$, $\si\cap\si'=\{\om,\om'\}$,
which decompose the circles into (closed) arcs
$\si=\si_+\cup\si_-$
and
$\si'=\si_+'\cup\si_-'$.
The choice of
$\om$
as an infinitely remote point automatically introduces the orientation
of 
$\si$
as well as of
$\si'$
such that 
$\om'$
is the initial point of the arcs
$\si_+$, $\si_+'$,
while
$\om$
the final point of
$\si_+$, $\si_+'$,
and the similar agreement holds true for the choice of
$\om'$
as an infinitely remote point. Then the 
$\slope(\si_\om';\si_\om)$
of the Ptolemy line
$\si_\om'\sub X_\om$
w.r.t. the Ptolemy line
$\si_\om\sub X_\om$
is well defined.

\begin{lem}\label{lem:first_variation} Let
$\si$, $\si'\sub X$
be Ptolemy circles meeting each other at two distinct points
$\om$
and
$\om'$, $\si\cap\si'=\{\om,\om'\}$,
which decompose the circles into (closed) arcs
$\si=\si_+\cup\si_-$
and
$\si'=\si_+'\cup\si_-'$.
Let
$c:\R\to X_{\om'}$, $c':\R\to X_\om$
be the unit speed parameterizations of the oriented Ptolemy lines
$\si_{\om'}\sub X_{\om'}$, $\si_\om'\sub X_\om$
respectively compatible with the orientations such that
$c(0)=\om$, $c'(0)=\om'$.
Then
\begin{equation}\label{eq:first_variation}
\frac{d}{dt}d'(c'(s),c(t))|_{t=0}=\al\sign s
\end{equation}
for all
$s\neq 0$,
where
$d'$
is the metric of
$X_{\om'}$,
and 
$\al=\slope(\si_\om';\si_\om)$.
\end{lem}

\begin{rem}\label{rem:first_variation} We emphasize that
(\ref{eq:first_variation}) is a typical duality equality
where the left hand side is computed in the space
$X_{\om'}$,
while the right hand side is computed in the opposite space
$X_\om$.
\end{rem}

\begin{proof} By Proposition~\ref{pro:busemann_affine}, 
the Busemann function
$b$
of
$\si_\om$
with
$b(\om')=0$
is affine along the Ptolemy line
$\si_\om'\sub X_\om$.
Thus
$b\circ c'(s)=\al s$
for 
$\al=\slope(\si_\om';\si_\om)$
and all
$s\in\R$
because
$b\circ c'(0)=b(\om')=0$.
Since
$X$
is Busemann flat by Lemma~\ref{lem:smooth_convex},
Equation~(\ref{eq:smooth_duality}) applied to
$b^+=b$
gives
$$\al s=b\circ c'(s)=\frac{d}{dt}\ln d'(c'(s),c(t))|_{t=0}
 =\frac{1}{d'(c'(s),\om)}\frac{d}{dt}d'(c'(s),c(t))|_{t=0}$$
for all
$s\neq 0$. 
Using 
$d'(c'(s),\om)=1/|s|$,
we obtain the required equality.
\end{proof}

In the situation with two Ptolemy circles intersecting at two
distinct points as in Lemma~\ref{lem:first_variation}, we have
four a priori different slopes. The duality and existence of 
s-inversions allows to reduce this number to one.

\begin{lem}\label{lem:opposite_slopes} Let
$\si$, $\si'\sub X$
be Ptolemy circles meeting each other at two distinct points
$\om$
and
$\om'$, $\si\cap\si'=\{\om,\om'\}$,
which decompose the circles into (closed) arcs
$\si=\si_+\cup\si_-$
and
$\si'=\si_+'\cup\si_-'$.
Then
$$\slope(\si_\om';\si_\om)=\slope(\si_{\om'};\si_{\om'}').$$
\end{lem}

\begin{proof} Denote
$\al=\slope(\si_\om';\si_\om)$
and
$\al'=\slope(\si_{\om'};\si_{\om'}')$. 
As in Lemma~\ref{lem:first_variation} let
$c:\R\to X_{\om'}$, $c':\R\to X_\om$
be the unit speed parameterizations of the oriented Ptolemy lines
$\si_{\om'}\sub X_{\om'}$, $\si_\om'\sub X_\om$
respectively compatible with the orientations such that
$c(0)=\om$, $c'(0)=\om'$.
Let
$b':X_{\om'}\to\R$
be the Busemann function associated with 
$\si_{\om'}'$
such that
$b'(\om)=0$
(according to our agreement,
$b'$
is computed via a parameterization of
$\si_{\om'}'$
which is opposite in orientation to that of
the parameterization 
$c'$). Then
$b'\circ c(t)=\al' t$
for all 
$t\in\R$
by the definition of
$\al'=\slope(\si_{\om'};\si_{\om'}')$.
Thus
$$\al' t=b'\circ c(t)\le d'(c'(s),c(t))-d'(c'(s),\om)$$
for all
$s>0$
and all (sufficiently small)
$t\in\R$,
where
$d'$
is the metric of
$X_{\om'}$.
The last inequality holds because the right hand side 
decreases monotonically to
$b'\circ c(t)$
for every fixed
$t$
as
$s\to 0$.
Applying Equality~(\ref{eq:first_variation}), we obtain
$\al'\le\al$.
Interchanging
$\om$
with
$\om'$
and
$\si$
with
$\si'$,
we have
$\al\le\al'$
by the same argument. Hence, the claim.
\end{proof}

Using Lemma~\ref{lem:opposite_slopes}, the first variation
formula~(\ref{eq:first_variation}) can be rewritten as follows
\begin{equation}\label{eq:re_first_variation}
\frac{d}{dt}d'(c'(s),c(t))|_{t=0}=\al'\sign s
\end{equation}
for all
$s\neq 0$,
where
$\al'=\slope(\si_{\om'};\si_{\om'}')$.
Now, the both sides of (\ref{eq:re_first_variation})
are computed in the same space
$X_{\om'}$.

Lemma~\ref{lem:opposite_slopes} implies the symmetry of the slope
w.r.t. the arguments. 

\begin{lem}\label{lem:slope_symmetry} For any oriented Ptolemy lines
$l$, $l'\sub X_\om$
we have
$\slope(l';l)=\slope(l;l')$.
\end{lem}

\begin{proof} We assume W.L.G. that
$l\cap l'=\om'$
and represent
$l=\si_\om$, $l'=\si_\om'$
for Ptolemy circles
$\si=l\cup\om$, $\si'=l'\cup\om$.
Then
$\si\cap\si'=\{\om,\om'\}$.
Let
$S\sub X$
be a sphere between
$\om$, $\om'$, $\phi=\phi_{\om,\om',S}:X\to X$
the s-inversion w.r.t.
$\om$, $\om'$, $S$.
Then
$\phi$
preserves any Ptolemy circle though
$\om$, $\om'$
and its orientations. In particular,
$\phi(\si_\om)=\si_{\om'}$
and
$\phi(\si_\om')=\si_{\om'}'$.
We assume that
$S=S_1^d(\om')$,
where
$d\in\cM$
is the metric of
$X_\om$.
Then the metric 
$d'=\phi^\ast d$
is the m-inversion of
$d$,
and vice versa, see Lemma~\ref{lem:sinversion_minversion}.
It follows that the map
$\phi:(X_\om,d)\to (X_{\om'},d')$
is an isometry.

Now, we have
$$\slope(\si_{\om'}';\si_{\om'})=
  \slope(\phi(\si_\om');\phi(\si_\om))=\slope(\si_\om';\si_\om)
  =\slope(l';l).$$
Using Lemma~\ref{lem:opposite_slopes} we obtain
$\slope(l;l')=\slope(\si_\om;\si_\om')
=\slope(\si_{\om'}';\si_{\om'})=\slope(l';l)$. 
\end{proof}

From now on, we use notation
$\slope(l,l')$
for the slope instead of
$\slope(l;l')$.
We say that Ptolemy lines
$l$, $l'\sub X_\om$
are {\em orthogonal} if
$\slope(l',l)=0$.
By Lemma~\ref{lem:slope_symmetry} this is a symmetric
relation. For orthogonal lines we also use notation
$l\bot l'$.

\subsection{Tangent lines to a Ptolemy circle}
\label{subsect:tangent_lines}

A Ptolemy line 
$l\sub X_\om$
is {\em tangent} to a Ptolemy circle 
$\si\sub X_\om$
at a point
$x\in\si$
if for every
$y\in\si$
sufficiently close to
$x$
we have
$\dist(y,l)=o(|xy|)$.

\begin{pro}\label{pro:tangent_rcircle} Every Ptolemy circle 
$\si\sub X_\om$
possesses a unique tangent Ptolemy line
$l$ 
at every point 
$x\in\si$. 
\end{pro}

\begin{proof} Let
$d$
be the metric of
$X_\om$, $d'$
the metric inversion of
$d$
w.r.t.
$x$, $d'(y,z)=\frac{d(y,z)}{d(y,x)d(z,x)}$.
Then
$x$
is infinitely remote for
$d'$,
and
$\si\sm x\sub X_x$
is a Ptolemy line w.r.t. the metric
$d'$
on
$X_x$.
By Corollary~\ref{cor:busparallel_foliation}, there is
a unique Ptolemy line 
$\wt l\sub X_x$
through
$\om$
which is Busemann parallel to the line
$\si\sm x$.
Then
$l=(\wt l\cup x)\sm\om\sub X_\om$
is a Ptolemy line through
$x$.
We show that
$l$
is tangent to
$\si$
at
$x$.
 
We fix on
$\si\sm x$
and
$\wt l$
compatible orientations, see sect.~\ref{subsect:slope},
and choose
$y\in\si$, $y'\in l$
with sufficiently small positive
$t=d(x,y)=d(x,y')$
according the orientations. Recall that
$d$
is also the metric inversion of
$d'$
w.r.t.
$\om$,
and that
$d(x,z)=1/d\;'(\om,z)$
for every
$z\in X\sm\{x,\om\}$.
Then
$$d(y,y')=\frac{d\;'(y,y')}{d\;'(\om,y)d\;'(\om,y')}
         =t\frac{d\;'(y,y')}{1/t}.$$
By Lemma~\ref{lem:busparallel_sublinear},
$\frac{d\;'(y,y')}{1/t}\to 0$
as
$t\to 0$,
hence
$d(y,y')=o(t)$,
and thus
$l$
is tangent to
$\si$
at
$x$.

If 
$l'\sub X_\om$
is another tangent line to
$\si$
at
$x$,
then reversing the argument above we observe that
the Ptolemy lines
$\wt l$, $\wt l'=(l'\cup\om)\sm x\sub X_x$
through 
$\om$
diverge sublinearly and thus they are Busemann parallel
again by Lemma~\ref{lem:busparallel_sublinear}. It follows
that
$\wt l=\wt l'$
and
$l=l'$. 
\end{proof}

Now, we reformulate Corollary~\ref{cor:busparallel_foliation}
as follows.

\begin{cor}\label{cor:tangent_unique} Given a Ptolemy line 
$l\sub X_\om$
and a point 
$x\in l$,
for any other point
$y\in X_\om$
there exists a unique Ptolemy circle
$\si\sub X_\om$
through
$y$
tangent to
$l$
at 
$x$.
In particular, if
$y\in l$,
then
$\si=l$. 
\end{cor}

\begin{proof} Consider a metric of the M\"obius structure on
$X$
with the infinitely remote point 
$x$
and apply Corollary~\ref{cor:busparallel_foliation}. 
\end{proof}

\section{Fibration 
$\pi_\om:X_\om\to B_\om$}
\label{sect:fibration}

Given
$x\in X_\om$,
we define
$$F_x=\bigcap_{l\ni x}H_l,$$
where the intersection is taken over all the Ptolemy lines
$l\sub X_\om$
through
$x$,
$H_l$
is the horosphere through 
$x$
of a Busemann function associated with
$l$
(since
$X$
is Busemann flat,
$H_l$
is independent of choice of a Busemann function).

\begin{lem}\label{lem:fiber_def} For any
$y\in F_x$
we have
$F_y=F_x$.
\end{lem}

\begin{proof}
By Corollary~\ref{cor:busparallel_foliation}, for every Ptolemy line
$l$
through
$x$
there is a unique Ptolemy line
$l'$
through
$y$
such that
$l$, $l'$
are Busemann parallel. Let
$b$
be a Busemann function of
$l$
such that
$b(x)=0$.
Then
$b(y)=0$
because
$y$
lies in the horosphere through
$x$
of
$b$.
Hence
$H_l=H_{l'}$
because
$b$
is a Busemann function also of
$l'$
and thus
$F_y=F_x$. 
\end{proof}

By Lemma~\ref{lem:fiber_def}, the sets
$F_x$, $F_{x'}$
coincide or are disjoint for any
$x$, $x'\in X_\om$.
We let
$B_\om=\set{F_x}{$x\in X_\om$}$
and define
$\pi_\om:X_\om\to B_\om$
by
$\pi_\om(x)=F_x$.
Therefore, the fibers
$F_b=\pi_\om^{-1}(b)$, $b\in B_\om$,
form a partition of
$X_\om$, $B_\om$
is the factor-space of this partition, and
$\pi_\om$
is the respective factor-map. A fiber
$F$
of
$\pi_\om$
is also called a $\K$-{\em line}. 

\begin{lem}\label{lem:induced_base_map}
For any
$\om$, $\om'\in X$,
any M\"obius automorphism
$\phi:X\to X$
with
$\phi(\om)=\om'$
induces a bijection
$\ov\phi:B_\om\to B_{\om'}$
such that
$\pi_{\om'}\circ\phi=\ov\phi\circ\pi_\om$.
\end{lem}

\begin{proof}
It follows from Lemma~\ref{lem:homothety_infinite} that
for any metrics
$d$, $d'$
of the M\"obius structure with infinitely remote points
$\om$, $\om'$
respectively, the map
$\phi:(X_\om,d)\to(X_{\om'},d')$
is a homothety. Thus
$\phi$
maps any Ptolemy line 
$l\sub X_\om$
to the Ptolemy line
$l'=\phi(l)\sub X_{\om'}$,
and
$b\circ\phi^{-1}$
is proportional a Busemann function of
$l'$
for any Busemann function
$b$
of
$l$.
It follows that
$\phi$
induces a bijection
$\ov\phi:B_\om\to B_\om$
such that
$\ov\phi\circ\pi_\om=\pi_\om\circ\phi$. 
\end{proof}

\begin{proof}[Proof of property ($1_\K$): uniqueness]
Let
$F\sub X_\om$
be a $\K$-line, and let
$x\in X\sm F$.
We show that there is at most one Ptolemy line in
$X_\om$
through
$x$
that meets
$F$.
Assume that there are Ptolemy lines
$l$, $l'\in X_\om$
through
$x$
that intersect
$F$.
Let
$c:\R\to l$, $c':\R\to l'$
be unit speed parameterizations such that
$c(0)=x=c'(0)$,
and
$c(s)\in F$, $c'(s')\in F$
for some
$s$, $s'>0$.
For the Busemann function
$b:X_\om\to\R$, $b(y)=\lim_{t\to-\infty}|yc(t)|-|t|$,
of
$l$
we have
$b(c(s))=s=b(c'(s'))$
because
$c(s)$, $c'(s')\in F$
and 
$F$
is a fiber of the fibration
$\pi_\om:X_\om\to B_\om$.
The function
$t\mapsto |c'(s')c(t)|-(|t|+s)$
is nonincreasing and it converges to
$b(c'(s'))-s=0$
as
$t\to-\infty$,
thus
$s'=|c'(s')c(0)|\ge s=|c(s)x|$.
Interchanging 
$l$ 
and
$l'$
we obtain 
$s\ge s'$
by the same reason. Hence
$s=s'$.
Since the Busemann function
$b$
is affine along
$l'$
by Proposition~\ref{pro:busemann_affine}
and it takes the equal values
$0=b\circ c(0)$
and
$b\circ c(s)=s=b\circ c'(s)$
along
$l$, $l'$
at two different parameter points, we have
$b\circ c(t)=b\circ c'(t)$
for every
$t\in\R$.
By Lemma~\ref{lem:unique_line},
$l=l'$.
\end{proof}

\subsection{\Semik-planes}
\label{subsect:semi_kplanes}

We fix
$\om\in X$
and a metric from the M\"obius structure with infinitely remote point
$\om$.
For a Ptolemy line
$l\sub X_\om$
we put
$M=M_l:=\cup F$,
where the union is taken over all the fibers
$F\sub X_\om$
of the fibration
$\pi_\om:X_\om\to B_\om$
which intersect
$l$, $F\cap l\not=\es$
(the Ptolemy line
$l$
has at most one point in common with any fiber
$F\sub X_\om$
because it intersects only once any its horosphere). The set
$M_l\sub X_\om$
is called a {\em\semik-plane} over
$l$.
Since different fibers of
$\pi_\om$
are disjoint, we have if
$M_l\cap F\not=\es$
for some fiber
$F$
of 
$\pi_\om$,
then
$F\sub M_l$
and there is a uniquely determined point
$x\in l$
such that
$x\in F$,
i.e.
$F$
is a member of the family of fibers that form
$M_l$.

\begin{lem}\label{lem:rfoliation_semik} Through every point
$x$
of a \semik-line
$M_l$
there is a uniquely determined Ptolemy line
$l'$
that meets every $\K$-line of
$M_l$
and moreover
$l'\sub M_l$
is Busemann parallel to
$l$.
Furthermore, any two $\K$-lines $F$, $F'$
of 
$M_l$
are equidistant in the sense that the segments of any two
Ptolemy lines
$l$, $l'\sub M_l$
between
$F$, $F'$
have equal lengths.
\end{lem}

\begin{proof}
By Corollary~\ref{cor:busparallel_foliation}, 
there is a unique Ptolemy line
$l'$
through
$x$
which is Busemann parallel to
$l$.
Consider compatible unit speed parameterizations
$c:\R\to l$, $c':\R\to l'$
such that
$c(0)\in F_x$, $c'(0)=x$,
where
$F_x\sub M_l$
is the $\K$-line through
$x$.

Let
$F$
be a $\K$-line of
$M_l$.
Then by definition
$c(t)\in F$
for some
$t\in\R$.
We show that
$c'(t)\in F$.
Let
$l''$
be a Ptolemy line through
$c(t)$, $b''$
a Busemann function of
$l''$
with
$b''(c(t))=0$.
We show that
$c'(t)$
lies in the zero level set of
$b''$, $b''(c'(t))=0$.

By Corollary~\ref{cor:busparallel_foliation}, there is
a Ptolemy line through
$c(0)$
for which
$b''$
is a Busemann function. Then the $\K$-line
$F_x$
lies in a level set of
$b''$,
in particular,
$b''(c(0))=b''(c'(0))=:\be$.
By Lemma~\ref{lem:paraline_busemann} we have
$b''\circ c(s)=\al s+\be=b''\circ c'(s)$
for all
$s\in\R$.
In particular,
$b''(c'(t))=b''(c(t))=0$,
hence
$c'(t)\in F$.
This also shows that the $\K$-lines
$F$, $F_x$
are equidistant. Moreover, this argument shows that for every
$s\in\R$,
the point
$c'(s)$
lies in the $\K$-line
$F_s$
through
$c(s)$,
thus
$l'\sub M_l$.
\end{proof}

\begin{lem}\label{lem:geoconvex_semik} Every \semik-plane
$M\sub X_\om$
is geodesically convex, i.e., every Ptolemy line
$l'\sub X_\om$
that meets
$M$
in two different points is contained in
$M$, $l'\sub M$.
\end{lem}

\begin{proof} Let
$x$, $x'\in l'\cap M$
be different points. By Lemma~\ref{lem:rfoliation_semik} there is
a Ptolemy line 
$l\sub M$
through
$x$.
Both
$l$, $l'$
meet the fiber 
$F_{x'}\sub M$
of the fibration
$\pi_\om$
through
$x'$.
Then
$l=l'$ 
by the uniqueness part of property ($1_\K$).
\end{proof}

\begin{proof}[Proof of property ($2_\K$)]
Let
$F$, $F'\sub X_\om$
be distinct fibers of the fibration
$\pi_\om:X_\om\to B_\om$.
Assume that Ptolemy lines
$l$, $l'\sub X_\om$
intersect both
$F$, $F'$,
and let
$F''$
be a $\K$-line that intersects
$l$.
We show that
$F''$
intersects also
$l'$.

Let
$M_l$
be the \semik-plane over 
$l$.
Then
$F$, $F'$, $F''\sub M_l$
by our assumption. We have
$l'\sub M_l$
by Lemma~\ref{lem:geoconvex_semik}. Hence
$l'$
intersects
$F''$
by Lemma~\ref{lem:rfoliation_semik}.
\end{proof}

\subsection{Zigzag curves}
\label{subsect:zigzag}

We fix
$\om\in X$
and consider a metric on
$X_\om$
with infinitely remote point
$\om$.
Let
$l\sub X_\om$
be an oriented Ptolemy line. By Corollary~\ref{cor:busparallel_foliation},
there is a foliation
$l(x)$, $x\in X_\om$
of the space
$X_\om$
by Ptolemy lines, which are Busemann parallel to
$l$.
Moreover, every member
$l(x)$
of the foliation has a well defined orientation compatible
with that of
$l$,
see sect.~\ref{subsect:slope}.

\begin{lem}\label{lem:parallel_couple} Let
$l_1$, $l_2\sub X_\om$
be oriented Ptolemy lines which induce respective
foliations of
$X_\om$.
We start moving from
$x\in X_\om$
along
$l_1(x)$
by some distance
$s_1\ge 0$
up to a point
$y$,
and then switch to
$l_2(y)$
and move along it by some distance
$s_2\ge 0$
up to a point
$z$.
Next, we move from
$x'\in X_\om$
along
$l_2(x')$
by the distance
$s_2$
up to a point
$y'$,
and then switch to
$l_1(y')$
and move along it by the distance
$s_1$
up to a point
$z'$,
where we always move in the directions prescribed by 
the orientations. If
$x$, $x'$
lie in a $\K$-line
$F\sub X_\om$,
then
$z$, $z'$
also lie in one and the same $\K$-line
$F'\sub X_\om$.
\end{lem}

\begin{proof} Let
$c_1$, $c_2:\R\to X_\om$
be the unit speed parameterizations of
$l_1(x)$, $l_2(x')$
respectively compatible with the orientations such that
$c_1(0)=x$, $c_2(0)=x'$.
We also consider the unit speed parameterizations
$c_1'$, $c_2':\R\to X_\om$
of
$l_1(y')$, $l_2(y)$
respectively compatible with the orientations such that
$c_1'(0)=y'$, $c_2'(0)=y$.
Then
$c_1(s_1)=y=c_2'(0)$, $c_2(s_2)=y'=c_1'(0)$
and
$c_2'(s_2)=z$, $c_1'(s_1)=z'$.

Let
$b$
be a Busemann function of a Ptolemy line
$l\sub X_\om$
which vanishes along
$F$,
in particular,
$b(x)=0=b(x')$.
By Proposition~\ref{pro:busemann_affine},
$b$
is affine along any Ptolemy line in
$X_\om$,
in particular,
$b\circ c_1(t)=\al_1t$, $b\circ c_2(t)=\al_2t$
for some
$\al_i$
which by Lemma~\ref{lem:paraline_busemann}
only depends on
$l_i$, $i=1,2$,
and for all
$t\in\R$. 
Thus we have
$b(z')=b\circ c_1'(s_1)=\al_1s_1+\al_2s_2$
and similarly
$b(z)=b\circ c_2'(s_2)=\al_2s_2+\al_1s_1$.
Hence any Busemann function on
$X_\om$
takes the same value at the points
$z$
and 
$z'$,
i.e. these points lie in a common $\K$-line
$F'$.
\end{proof}

Given a base point
$o\in X_\om$,
a finite ordered collection
$\cL=\{l_1,\dots,l_k\}$
of oriented Ptolemy lines in
$X_\om$,
and a collection
$S=\{s_1,\dots,s_k\}$
of nonnegative numbers with
$s_1+\dots+s_k>0$,
we construct a sequence
$\ga_p=\ga_p(o,\cL,S)\sub X_\om$, $p\ge 1$,
of piecewise geodesic curves through
$o$
as follows. Recall that we have 
$k$ 
foliations of
$X_\om$
by oriented Ptolemy lines 
$l_1(x),\dots,l_k(x)$, $x\in X_\om$,
which are Busemann parallel with compatible orientations to
$l_1,\dots,l_k$
respectively.

The curve
$\ga_p$
starts at
$o=v_p^0$
for every
$p\ge 1$.
We move along
$l_1(o)$
by the distance
$s_1/2^{p-1}$
up to the point
$v_p^1\in l_1(v_p^0)$,
then switch to the line
$l_2(v_p^1)$
and move along it by the distance
$s_2/2^{p-1}$
up to the point
$v_p^2$
etc. On the
$i$th 
step, for
$1\le i\le k$, 
we move along the line
$l_i(v_p^{i-1})$
by the distance
$s_i/2^{p-1}$
in the direction prescribed by the orientation of the line
up to the point
$v_p^i\in l_i(v_p^{i-1})$.
Starting with the point
$v_p^k$
we then repeat this procedure only taking the subindices for
$l_i$, $s_i$
modulo
$k$
for all integer
$i\ge k+1$.

This produces the sequence
$v_p^n$
of vertices of
$\ga_p$
for all
$n\ge 0$.
For integer
$n<0$
the vertices
$v_p^n$
are determined in the same way with all the orientations
of the lines
$l_1,\dots,l_k$
reversed, with the starting line
$l_k(o)$,
and with the ordered collections
$\ov{\cL}=\{l_k,\dots,l_1\}$
of lines, and
$\ov S=\{s_k,\dots,s_1\}$
of numbers.

Every curve
$\ga_p$
receives the arclength parameterization, for which
we use the same notation
$\ga_p:\R\to X_\om$,
with
$\ga_p(0)=o$.
Then for every
$m\in\Z$, $1\le i\le k$,
we have
$\ga_p(t_p^n)=v_p^n$
is a vertex of
$\ga_p$,
where
$n=k(m-1)+i$, 
$t_p^n=[(s_1+\dots+s_i)m+(s_{i+1}+\dots+s_k)(m-1)]/2^{p-1}$
(the sum
$(s_{i+1}+\dots+s_k)$
is assumed to be zero for
$i=k$).

It follows from Lemma~\ref{lem:parallel_couple}
by induction that for every 
$n=km\in\Z$,
the vertices
$v_p^n=\ga_p(t_p^n)$
of
$\ga_p$
and
$v_{p+1}^{2n}=\ga_{p+1}(t_{p+1}^{2n})$
of
$\ga_{p+1}$
lie in a common $\K$-line in
$X_\om$
for every
$p\ge 1$.
From this one easily concludes that the sequence of the projected curves
$\pi_\om(\ga_p)\sub B_\om$
converges (pointwise in the induced topology). At this stage,
we do not have tools to prove that the sequence
$\ga_p$
itself converges in
$X_\om$.
However, we need a limiting object of
$\ga_p$.
Thus, for instance, we fix a nonprincipal ultra-filter on
$\Z$
and say that
$\ga=\lim\ga_p$
w.r.t. that ultra-filter. By this we mean that
$\ga(t)=\lim\ga_p(t)$
for every 
$t\in\R$.
The curve
$\ga=\ga(o,\cL,S)$
is called a {\em zigzag} curve, and it is obtained together with 
the limiting parameterization
$\ga:\R\to X_\om$, $\ga(0)=o$,
which in general is not an arclength parameterization.

\begin{lem}\label{lem:busemann_affine_zigzag} Every Busemann function
$b:X_\om\to\R$
is affine along any zigzag curve
$\ga$,
that is, the function
$b\circ\ga:\R\to\R$
is affine. Moreover, if
$\ga=\ga(o,\cL,S)$
for a base point
$o\in X_\om$,
some ordered collection
$\cL=\{l_1,\dots,l_k\}$
of oriented Ptolemy lines in
$X_\om$, 
and a collection
$S=\{s_1,\dots,s_k\}$
of nonnegative numbers with
$s_1+\dots+s_k>0$,
and
$b(o)=0$,
then
$b\circ\ga(t)=\be t$
for all
$t\in\R$,
where
$\be=\sum_i\al_is_i/\sum_is_i$,
$\al_i=\slope(l_i,l)$, $i=1,\dots,k$,
and
$l\sub X_\om$
is the oriented Ptolemy line for which the function
$b$
is associated.
\end{lem}

\begin{proof} We assume that
$\ga=\lim\ga_p$.
Since
$\ga_p$
is piecewise geodesic for every
$p\ge 1$,
the function
$b\circ\ga_p:\R\to\R$
is piecewise affine. Recall that the points
$v_p^n=\ga_p(t_p^n)$
are vertices of
$\ga_p$,
where
$t_p^n=[(s_1+\dots+s_i)m+(s_{i+1}+\dots+s_k)(m-1)]/2^{p-1}$
for 
$n=k(m-1)+i\in\Z$.
Thus we have by induction
$$b\circ\ga_p(t_p^n)=
  [(\al_1s_1+\dots+\al_is_i)m
  +(\al_{i+1}s_{i+1}+\dots+\al_ks_k)(m-1)]/2^{p-1}$$
for
$n=k(m-1)+i\in\Z$.
Hence,
$b\circ\ga_p(t_p^n)=\be t_p^n+o(1)$
as
$p\to\infty$.
Since the step
$t_p^{n+1}-t_p^n\le\max_is_i/2^{p-1}\to 0$
as
$p\to\infty$,
we conclude that
$b\circ\ga_p\to b\circ\ga$
pointwise as
$p\to\infty$,
and
$b\circ\ga(t)=\be t$
for all
$t\in\R$.
\end{proof}

Lemma~\ref{lem:busemann_affine_zigzag} gives a strong evidence
in support of the expectation that a zigzag curve under natural
assumptions actually is a Ptolemy line. However we need additional 
arguments for the proof of this.

For
$\om$, $o\in X$,
the group
$\Ga_{\om,o}$
consists of homotheties
$\phi:X_\om\to X_\om$
with
$\phi(o)=o$
such that
$\phi(l)=l$
for every Ptolemy line
$l\sub X_\om$
through
$o$
preserving an orientation of
$l$,
and moreover by property (H),  
$\Ga_{\om,o}$
acts transitively on the open rays of
$l$
with the vertex
$o$,
see Proposition~\ref{pro:homothety_property}.

\begin{lem}\label{lem:zigzag_preserved} The homothety
$\phi\in\Ga_{\om,o}$
with the coefficient
$\la=1/2$
leaves invariant a zigzag curve
$\ga=\ga(o,\cL,S)$
for any base point 
$o\in X_\om$,
any ordered collection of oriented Ptolemy lines
$\cL=\{l_1,\dots,l_k\}$
in
$X_\om$,
and any collection
$S=\{s_1,\dots,s_k\}$
of nonnegative numbers with
$s_1+\cdots+s_k>0$.
\end{lem}

\begin{proof} Let 
$\ga_p$, $p\ge 1$,
be the sequence of piecewise geodesic curves in
$X_\om$
used in the construction of
$\ga=\lim\ga_p$.
For
$p\ge 1$,
we let
$v_p^{km}$, $m\in\Z$,
be the sequence of vertices of
$\ga_p$, $\ov v_p^{km}=\pi_\om(v_p^{km})$ 
the sequence of respective fibers of the fibration
$\pi_\om:X_\om\to B_\om$.
Recall that the sequences
$\set{\ov v_p^{km}}{$m\in\Z$}$, $p\ge 1$,
approximate the projection
$\pi_\om(\ga)$
of 
$\ga$,
that is,
$\pi_\om(\ga)$
coincides with the closure of the union
$\cup_p\set{\ov v_p^{km}}{$m\in\Z$}$.

We have
$\phi(\ga_p)=\ga_{p+1}$
and
$\phi(v_p^{km})=v_{p+1}^{km}$
by the construction of
$\ga_p$
and Lemma~\ref{lem:pure_homothety}. For every dyadic number
$q=m/2^r$, $m\in\Z$, $r\ge 0$,
the sequence
$v(q)=\set{v_p^{q_p}}{$p\ge r+1$}$,
where
$q_p=2^{p-(r+1)}\cdot km$, 
lies in a common fiber
$F=F(q)$
of
$\pi_\om$.
Thus 
$\phi$
maps this sequence into the sequence 
$v'(q)=\set{v_{p+1}^{q_p'}}{$p\ge r+1$}\sub\phi(F)=F(q/2)$,
where
$q_p'=2^{p-r}(km/2)$,
shrinking the mutual distances by the factor
$1/2$.
Hence for the limit point
$x=\lim v(q)\in\ga$
of any limiting procedure giving
$\ga=\lim\ga_p$
we have
$\phi(x)=\lim v'(q)\in\ga$.
The points of type
$x=\lim v(q)$
with dyadic
$q$
are dense in
$\ga$,
thus 
$\phi$
preserves
$\ga$, $\phi(\ga)=\ga$.
\end{proof}

\begin{lem}\label{lem:zigzag_change_basepoint} Let
$\ga=\ga(o,\cL,\cS)\sub X_\om$
be a zigzag curve with base point
$o\in X_\om$,
where
$\cL=\{l_1,\dots,l_k\}$, $\cS=\{s_1,\dots,s_k\}$, 
$s_1+\dots+s_k>0$.
Then for any
$o'\in\ga$
we have
$\ga(o',\cL,\cS)=\ga$,
i.e. any zigzag curve 
$\ga$
is independent of a choice of
its base point
$o$.
\end{lem}

\begin{proof} We first consider the case
$o'=\ga(t_q)$
is a {\em dyadic} point with dyadic
$q=m/2^r$, $m\in\Z$, $r\ge 0$,
and 
$t_q=(s_1+\dots+s_k)q$,
for the canonical parameterization
$t\mapsto\ga(t)$
of
$\ga$.
Then 
$o'$
is an accumulation point of the vertices
$v_p=v_p^{q_p}=\ga_p(t_p^{q_p})$, $p\ge r+1$,
where
$q_p=2^{p-(r+1)}\cdot km$
and
$t_p^{q_p}=(s_1+\dots+s_k)(q_p/k)/2^{p-1}=t_q$,
of approximating piecewise geodesic curves
$\ga_p$, $\ga=\lim\ga_p$
(recall that the sequence
$v(q)=\set{v_p^{q_p}}{$p\ge r+1$}$
lies in a fiber
$F(q)\sub X_\om$
of the projection
$\pi_\om$,
see the proof of Lemma~\ref{lem:zigzag_preserved}).
That is,
$o'=\lim v_p$
for our limiting procedure. 

By Lemma~\ref{lem:shift_identity}, there is a shift
$\eta_p=\eta_{v_po'}:X_\om\to X_\om$
with
$\eta_p(v_p)=o'$
and
$\lim\eta_p=\id$.
Then
$\eta_p(\ga_p)=\ga_p'$,
where
$\ga_p'=\ga_p(o',\cL,\cS)$
is the piecewise geodesic curve with the base point 
$o'$
approximating the zigzag curve
$\ga'=\ga(o',\cL,\cS)$, $\ga'=\lim\ga_p'$.
Now for an arbitrary point
$x\in\ga$, $x=\ga(t)$,
we have
$x=\lim\ga_p(t)$.
We put
$x'=\ga'(t)=\lim\ga_p'(t)$.
Then for an arbitrary
$\ep>0$
we have
$|x\ga_p(t)|<\ep$, $|x'\ga_p'(t)|<\ep$,
and
$|x\eta_p(x)|<\ep$
for sufficiently large
$p$.
The last estimate holds since
$\lim\eta_p=\id$.
Using
$|\eta_p(x)\ga_p'(t)|=|\eta_p(x)\eta_p\circ\ga_p(t)|
 =|x\ga_p(t)|$,
we obtain 
$$|xx'|\le|x\eta_p(x)|+|\eta_p(x)\ga_p'(t)|+|\ga_p(t)x'|\le 3\ep,$$
thus
$x=x'$,
that is,
$\ga=\ga'$.

For a general case, the point
$o'=\ga(t)$
can be approximated by dyadic ones,
$t_q\to t$.
Then respective piecewise geodesic curves
$\ga_{p,q}'$
with dyadic base points
$\ga(t_q)$
approximate pointwise the curve
$\ga_p'$
with the base point 
$o'$
for every 
$p\ge 1$.
Thus
$\ga'=\ga$
also in that case.
\end{proof}

\begin{pro}\label{pro:zigzag_geodesic} Every zigzag curve
$\ga\sub X_\om$
is either a geodesic and hence a Ptolemy line, or it degenerates
to a point. More precisely, if
$\ga=\ga(o,\cL,\cS)$
for some base point 
$o\in X_\om$
and collections
$\cL=\{l_1,\dots,l_k\}$
of oriented Ptolemy lines in
$X_\om$, $\cS=\{s_1,\dots,s_k\}$
of nonnegative numbers with
$s_1+\dots+s_k>0$,
then
$\ga$
is degenerate if and only if
$\sum_i\al_is_i=0$
for every oriented Ptolemy line
$l\sub X_\om$,
where
$\al_i=\slope(l_i,l)$, $i=1,\dots,k$.
\end{pro}

\begin{proof} We first show that for each
$o$, $o'\in\ga$
there is a midpoint
$x\in\ga$.
By Lemma~\ref{lem:zigzag_change_basepoint} and
Lemma~\ref{lem:zigzag_preserved}, homotheties
$\phi\in\Ga_{\om,o}$, $\phi'\in\Ga_{\om,o'}$
with coefficient
$\la=1/2$
both preserve
$\ga$, $\phi(\ga)=\ga=\phi'(\ga)$.
Then for 
$x=\phi(o')\in\ga$
we have
$|ox|=|oo'|/2$,
and similarly for
$x'=\phi'(o)\in\ga$
we have
$|x'o'|=|oo'|/2$.
Furthermore, the length of the segment of
$\ga$
between
$o$, $x$
is half of the length of the segment between
$o$, $o'$, $L([ox]_\ga)=L([oo']_\ga)/2$.
Thus
$L([xo']_\ga)=L([oo']_\ga)/2$
by additivity of the length. Then
$L([x'o']_\ga)=L([oo']_\ga)/2=L([xo']_\ga)$
and hence
$x'=x$
by monotonicity of the length, and
$x$
is the required midpoint.

It follows that the segment of
$\ga$
between its any two points is geodesic. Since 
$\ga$
is invariant under the nontrivial homothety
$\phi\in\Ga_{\om,o}$,
we see that 
$\ga$
is a Ptolemy line unless it degenerates to a point.

If 
$\ga$
is degenerate, then any Busemann function
$b:X_\om\to\R$
is constant along 
$\ga$.
By Lemma~\ref{lem:busemann_affine_zigzag}, we have
$\sum_i\al_is_i=0$
for every oriented Ptolemy line 
$l\sub X_\om$, 
where
$\al_i=\slope(l_i,l)$, $i=1,\dots,k$.
Conversely, if
$\ga$
is nondegenerate, then 
$l=\ga(\R)$
is a Ptolemy line in
$X_\om$
by the first part of the proof,
and the associated Busemann function 
$b:X_\om\to\R$
is nonconstant along
$l$.
By Lemma~\ref{lem:busemann_affine_zigzag}, we have
$b\circ\ga(t)=\be t$
for the canonical parameterization of
$\ga$
with
$\be=\sum_i\al_is_i/\sum_is_i$,
$\al_i=\slope(l_i,l)$, $i=1,\dots,k$.
Thus
$\sum_i\al_is_i\neq 0$.
\end{proof}

\begin{rem}\label{rem:distinct_fiber_zigzag} Assume
$\ga_1$
is a piecewise geodesic curve (with finite number
of edges) between different fibers in
$X_\om$, $\ga_1(0)\in F$, $\ga_1(s_1+\dots+s_k)\in F'$,
$F\neq F'$,
where
$s_1,\dots,s_k$
are the lengths of its edges. Then the respective zigzag curve
$\ga=\lim\ga_p$
is not degenerate. This follows from
$\ga(0)\in F$, $\ga(s_1+\dots+s_k)\in F'$
by construction of the approximating sequence
$\ga_1,\dots,\ga_p,\dots\to\ga$. 
\end{rem}

Now, we compute a unit speed parameterization a zigzag
curve
$\ga=\ga(o,\cL,S)$
assuming for simplicity that the collection
$\cL$
consists of mutually orthogonal Ptolemy lines.

\begin{lem}\label{lem:uspeed_parameter_zigzag} Let
$\ga=\ga(o,\cL,S)$
be a zigzag curve in
$X_\om$,
where the collection
$\cL=\{l_1,\dots,l_k\}$
consists of mutually orthogonal oriented Ptolemy
lines,
$l_i\bot l_j$
for 
$i\neq j$.
Then
$$|o\ga(t)|=\frac{\sqrt{\sum_is_i^2}}{\sum_is_i}|t|$$
for all
$t\in\R$,
where, we recall,
$t\mapsto\ga(t)$
is the canonical parameterization, and
$S=\{s_1,\dots,s_k\}$. 
In particular,
$\ga$
is nondegenerate.
\end{lem}

\begin{proof} By Proposition~\ref{pro:zigzag_geodesic},
$l=\ga(\R)$
is a Ptolemy line or it degenerates to a point, in particular,
$l$
is invariant for every pure homothety
$\phi:X_\om\to X_\om$
with
$\phi(o)=o$.
From this one easily finds that there is
$\la\in[0,1]$
such that
$|o\ga(t)|=\la|t|$
for all 
$t\in\R$.
Let
$b$
be the Busemann function of
$l$
normalized by
$b(o)=0$
and
$b(t)<0$
for
$t>0$
(if
$l$
is degenerate, then
$b\equiv 0$
by definition). By Lemma~\ref{lem:busemann_affine_zigzag},
$-|o\ga(t)|=b\circ\ga(t)=\be t$
with
$\be=\sum_i\al_is_i/\sum_is_i$
for all 
$t\ge 0$,
where
$\al_i=\slope(l_i,l)$, $i=1,\dots,k$.
Let
$\be_i$, $i=1,\dots,k$,
be the Busemann function of
$l_i$
normalized by
$b_i(o)=0$.
Using symmetry of the slope (Lemma~\ref{lem:slope_symmetry}), 
$\al_i=\slope(l,l_i)$,
we obtain 
$\be_it=b_i\circ\ga(t)=\al_i\la t$
for all
$t\ge 0$,
where by Lemma~\ref{lem:busemann_affine_zigzag} again, 
$$\be_i=\sum_j\slope(l_j,l_i)s_j/\sum_js_j
  =-s_i/\sum_js_j$$
for
$i=1,\dots,k$.
Thus
$\la\be=-\frac{\sum_is_i^2}{(\sum_is_i)^2}$
and
$\la^2 t=\la|o\ga(t)|=t\frac{\sum_is_i^2}{(\sum_is_i)^2}$
for all
$t\ge 0$.
Therefore,
$\la^2=\sum_is_i^2/(\sum_is_i)^2>0$.
In particular,
$l$
is nondegenerate.
\end{proof}

\subsection{Orthogonalization procedure}
\label{subsect:existence}

As usual, we fix
$\om\in X$
and a metric 
$d$
of the M\"obius structure with infinitely
remote point 
$\om$.

\begin{pro}\label{pro:orthogonal_base} There is a finite
collection
$\cL_\bot$
of mutually orthogonal Ptolemy lines such that for every
$x\in X_\om$
the fiber
$F\sub X_\om$
through
$x$
of the fibration
$\pi_\om:X_\om\to B_\om$
is represented as
$F=\cap_{l\in\cL_\bot}H_l$,
where
$H_l$
is the horosphere of
$l$
through
$x$. 
\end{pro}

We first note that the cardinality of any collection
of mutually orthogonal Ptolemy lines in
$X_\om$
is uniformly bounded above.

\begin{lem}\label{lem:orthogonal_bound} There is
$N\in\N$
such that the cardinality of any collection
$\cL$
of mutually orthogonal Ptolemy lines in
$X_\om$
is bounded by
$N$, $|\cL|\le N$.
\end{lem}

\begin{proof} We fix 
$x\in X_\om$
and assume W.L.G. that all the lines of
$\cL$
pass through
$x$.
For every line
$l\in\cL$
we fix a point on
$l$
at the distance 1 from
$x$. 
Let
$A\sub X_\om$
be the set of obtained points. By compactness of
$X$
and homogeneity of
$X_\om$
it suffices to show that the distance
$|aa'|\ge 1$
for each distinct
$a$, $a'\in A$.
We have
$a\in l$, $a'\in l'$
for some distinct lines
$l$, $l'\in\cL$.
Since
$\slope(l',l)=0$,
the lines
$l$, $l'$
are also orthogonal at the infinite remote point
$\om$
according Lemma~\ref{lem:opposite_slopes} and 
Lemma~\ref{lem:slope_symmetry}. Thus in the space
$X_x$
with infinitely remote point 
$x$,
the Ptolemy line 
$l'\sm x\sub X_x$
lies in the horosphere 
$H$
of the line
$l\sm x\sub X_x$
through
$\om$.
By duality, see Lemma~\ref{lem:flat_duality}, the point 
$x\in l$
is closest on the line 
$l$
to any fixed point of 
$l'\sub X_\om$, 
in particular
$|a'a|\ge|a'x|=1$. 
\end{proof}

Next, we describe an orthogonalization procedure.

\begin{lem}\label{lem:orthogonalization_procedure} Let
$l_1,\dots,l_k$
be a collection of mutually orthogonal Ptolemy lines in
$X_\om$, $l_i\bot l_j$
for
$i\neq j$.
Given a Ptolemy line 
$l\sub X_\om$,
through any
$o\in X_\om$
there is a zigzag curve
$\ga=\ga(o,\cL,S)$,
where
$\cL=\{l_1,\dots,l_k,l\}$
is an ordered collection of oriented Ptolemy lines,
$S=\{s_1,\dots,s_{k+1}\}$
a collection of nonnegative numbers with
$s_1+\dots+s_{k+1}>0$,
which is orthogonal to
$l_1,\dots,l_k$, $\ga(\R)=l_{k+1}\bot l_i$
for
$i=1,\dots,k$.
Furthermore, if
$\sum_1^k\al_i^2\neq 1$,
where
$\al_i=\slope(l,l_i)$,
then
$\ga$
is nondegenerate, and
$l_{k+1}$
is a Ptolemy line. 
\end{lem}

\begin{proof} We fix an orientation of
$l$
and for every
$i=1,\dots,k$
we choose an orientation of
$l_i$
so that
$\al_i=\slope(l,l_i)\ge 0$,
and put
$\al:=\sum_i\al_i\ge 0$.
For any zigzag curve
$\ga=\ga(o,\cL,S)$
in
$X_\om$,
where
$\cL=\{l_1,\dots,l_k,l\}$, $S=\{s_1,\dots,s_{k+1}\}$,
for 
$i=1,\dots,k$
and for the Busemann function
$b_i$
of
$l_i$
with
$b_i(o)=0$,
by Lemma~\ref{lem:busemann_affine_zigzag} we have 
$b_i\circ\ga(t)=\be_it$
for all 
$t\in\R$,
where
$\be_i=(-s_i+\al_is_{k+1})/(s_1+\dots+s_{k+1})$.
Then putting
$s_i=\frac{\al_i}{1+\al}$, $i=1,\dots,k$,
$s_{k+1}=\frac{1}{1+\al}$,
we have
$s_1+\dots s_{k+1}=1$
and
$\be_i=0$
for every
$i=1,\dots,k$.
Thus 
$\ga$
is orthogonal to
$l_1,\dots,l_k$,
and it gives us a required Ptolemy line
$l_{k+1}=\ga(\R)$
unless
$\ga$
degenerates.
 
Let
$b$
be the Busemann function of
$l$
with 
$b(o)=0$. 
By Lemma~\ref{lem:busemann_affine_zigzag} we have
$b\circ\ga(t)=\be t$
for all 
$t\in\R$
with
$\be=\sum_1^k\slope(l_i,l)s_i-s_{k+1}$.
Using the symmetry of the slope, we see that
$\slope(l_i,l)=\slope(l,l_i)=\al_i$
and thus
$\be=(\sum_1^k\al_i^2-1)/(1+\al)$.
If
$\sum_1^k\al_i^2\neq 1$,
then this shows that
$\ga$
is nondegenerate.
\end{proof}

We say the a collection
$\{l_1,\dots,l_k\}$
mutually orthogonal Ptolemy lines in
$X_\om$
is {\em maximal} if there is no Ptolemy line in
$X_\om$
which is orthogonal to every
$l_1,\dots,l_k$.
By Lemma~\ref{lem:orthogonal_bound}, such a collection exists.

\begin{lem}\label{lem:max_orthogonal} Let
$\{l_1,\dots,l_k\}$
be a maximal collection of mutually orthogonal Ptolemy lines in
$X_\om$.
Then every Ptolemy line 
$l\sub X_\om$
can be represented as a zigzag curve
$\ga(o,\cL,S)$
with
$o\in l$
for a collection
$\cL=\{l_1,\dots,l_k\}$
of oriented Ptolemy lines and a collection 
$S=\{s_1,\dots,s_k\}$
of nonnegative numbers with
$s_1+\dots+s_k>0$.
Furthermore, we have
$\sum_i\al_i^2=1$,
where
$\al_i=\slope(l,l_i)$, $i=1,\dots,k$.
\end{lem}

\begin{proof} We apply the orthogonalization procedure
described in Lemma~\ref{lem:orthogonalization_procedure} to
the collection
$\{l_1,\dots,l_k,l\}$
and construct a zigzag curve
$\ga_l=\ga_l(o,\cL_l,S_l)$,
where
$\cL_l=\{l_1,\dots,l_k,l\}$
is the collection above of oriented Ptolemy lines,
$S=\{s_1,\dots,s_{k+1}\}$
for an appropriate choice of entries described there.
Since
$\ga_l$
of orthogonal to
$l_1,\dots,l_k$,
we conclude from maximality of
$\{l_1,\dots,l_k\}$
that 
$\ga_l$
degenerates and moreover
$\sum_i\al_i^2=1$
by Lemma~\ref{lem:orthogonalization_procedure}.
 
According to Remark~\ref{rem:distinct_fiber_zigzag},
the ends of the piecewise geodesic curve
$\ga_{l,1}$
with 
$k+1$
edges
$\si_1,\dots,\si_k,\si$
on Ptolemy lines Busemann parallel to
$l_1,\dots,l_k,l$
respectively with
$|\si_i|=s_i$, $i=1,\dots,k$, $|\si|=s_{k+1}$,
lie in one and the same fiber ($\K$-line)
$F\sub X_\om$, 
that is,
$o=\ga_{l,1}(0)$
and 
$x=\ga_{l,1}(s+s_{k+1})\in F$,
where
$s=s_1+\dots+s_k$.
Thus the reduced piecewise geodesic curve
$\ga_1=\si_1\cup\dots\cup\si_k$
and the last edge
$\si$
of
$\ga_{l,1}$
have the ends in the same fibers
$F$, $F'$,
where
$F'$
is the fiber through
$\ga_{l,1}(s)=\ga_1(s)$.

Then the zigzag curve
$\ga=\ga(o,\cL,S)$
is nondegenerate, where 
$\cL=\{l_1,\dots,l_k\}$, $S=\{s_1,\dots,s_k\}$,
and it gives a Ptolemy line 
$\ga(\R)\sub X_\om$
though
$o$
which hits the fiber
$F'$.
Since the Ptolemy line
$l'$
containing the segment
$\si$
is Busemann parallel to
$l$
and intersects the fibers
$F$, $F'$,
the line 
$\ga(\R)$
is Busemann parallel to
$l$,
see Lemma~\ref{lem:rfoliation_semik}. Hence
$\ga(\R)=l$
by uniqueness, see Lemma~\ref{lem:unique_line}. 
\end{proof}

\begin{lem}\label{lem:linear_combination_busemann} Let
$\cL=\{l_1,\dots,l_k\}$
be a maximal collection of mutually orthogonal oriented Ptolemy lines in
$X_\om$, $S=\{s_1,\dots,s_k\}$
a collection of nonnegative numbers with 
$s_1+\dots+s_k=1$, $b$, $b_i:X_\om\to\R$
the Busemann functions of the zigzag curve
$\ga=\ga(o,\cL,S)$,
Ptolemy line 
$l_i$
with
$b(o)=0=b_i(o)$
respectively,
$i=1,\dots,k$.
Then
$\la b=\sum_is_ib_i$,
where
$\la=\sqrt{\sum_is_i^2}$. 
\end{lem}

\begin{proof} We denote by
$h=\la b-\sum_is_ib_i$
the function
$X_\om\to\R$
with
$h(o)=0$,
which is an affine function on every Ptolemy line on
$X_\om$.
First, we check that 
$h$
vanishes along 
$l_1,\dots,l_k$
(assuming that these lines pass through
$o$).
Indeed,
$\sum_is_ib_i(z)=s_jb_j(z)$
for every
$z\in l_j$
because
$l_i\bot l_j$
for
$i\neq j$.
Since
$b$
is a Busemann function of
$l$,
it is affine on
$l_j$
with the coefficient
$\slope(l_j,l)$,
$b(z)=-\slope(l_j,l)b_j(z)$.
Using symmetry of the slope, we obtain
$$\slope(l_j,l)=\slope(l,l_j)=\frac{1}{\la}\sum_i\slope(l_i,l_j)s_i
  =-s_j/\la,$$
see the proof of Lemma~\ref{lem:uspeed_parameter_zigzag}.
Thus
$\la b(z)=s_jb_j(z)$,
and
$h(z)=0$.

Next, we show that if
$h$
is constant of a Ptolemy line
$l$, 
then it is constant on every Ptolemy line
$l'$
that is Busemann parallel to
$l$
(with maybe a different value). By Lemma~\ref{lem:busparallel_sublinear} 
we know that
$l$, $l'$
diverge at most sublinearly,  
and also that
$h$
is affine on
$l'$.
Thus 
$h|l'$
cannot be nonconstant because 
$h$
is a Lipschitz function on
$X_\om$.

It follows that 
$h$
vanishes on every piecewise geodesic curve with origin
$o$
and with edges Busemann parallel to the lines
$l_1,\dots,l_k$.
Hence,
$h$
vanishes along any zigzag curve of type
$\ga(o,\cL,S)$.
Using Lemma~\ref{lem:max_orthogonal}, we conclude that
$h$
is constant along any Ptolemy line in
$X_\om$.

By Proposition~\ref{pro:tangent_rcircle}, every Ptolemy
circle possesses a unique tangent line, which is certainly
a Ptolemy line, at every point. Using standard approximation
arguments, we see that
$h$
is constant along any Ptolemy circle in
$X_\om$.
By the existence property (E), every
$x\in X_\om$
is connected with 
$o$
by a Ptolemy circle. Thus
$h(x)=0$
and
$\la b=\sum_ib_i$. 
\end{proof}

\begin{proof}[Proof of Proposition~\ref{pro:orthogonal_base}]
Let
$\cL_\bot=\{l_1,\dots,l_k\}$
be a maximal collection of mutually orthogonal Ptolemy lines in
$X_\om$.
This means that we actually consider respective foliations of
$X_\om$
by Busemann parallel Ptolemy lines.
For any
$x\in X_\om$,
for the fiber
$F$
and for the respective lines from
$\cL_\bot$
through 
$x$,
we have by definition
$F\sub\cap_jH_j$,
where
$H_j\sub X_\om$
is the horosphere of
$l_j$
through
$x$.
It follows from Lemma~\ref{lem:max_orthogonal} and
Lemma~\ref{lem:linear_combination_busemann} that any
Busemann function
$b:X_\om\to\R$
with 
$b(x)=0$
is a linear combination of the Busemann functions
$b_1,\dots,b_k$
of the lines
$l_1,\dots,l_k$
which vanish at
$x$.
Thus 
$b$
vanishes on
$\cap_jH_j$,
and therefore
$F=\cap_jH_j$.
\end{proof}

\begin{proof}[Proof the property ($1_\K$)] Given a fiber 
($\K$-line) $F\sub X_\om$
and
$x\in X_\om\sm F$,
we show that there is a Ptolemy line 
$l\sub X_\om$
through
$x$
that hits
$F$.
Uniqueness of
$l$
is proved at the end of sect.~\ref{sect:fibration}.

Using Proposition~\ref{pro:orthogonal_base}, we represent
$F=\cap_{l\in\cL_\bot}H_l$,
where
$\cL_\bot$
is a finite collection of mutually orthogonal Ptolemy lines,
$H_l$
a horosphere of
$l$.
Choosing appropriate orientations of the members of
$\cL_\bot$,
we can assume that
$H_l=b_l^{-1}(0)$
and
$b_l(x)\ge 0$
for every
$l\in\cL_\bot$,
where
$b_l:X_\om\to\R$
is a Busemann function of
$l$.
Moving from
$x$
in an appropriate direction along a Ptolemy line,
which is Busemann parallel to 
$l\in\cL_\bot$
with
$b_l(x)>0$,
we reduce the value of
$b_l$
to zero keeping up every other Busemann function
$b_{l'}$, $l'\in\cL_\bot$,
constant. Repeating this procedure, we connect
$x$
with 
$F$
by a piecewise geodesic curve with at most
$|\cL_\bot|$
edges. Now, the zigzag construction produces a required
Ptolemy line through
$x$
that hits
$F$.
\end{proof}

\subsection{Properties of the base $B_\om$}
\label{subsect:base_properties}

We fix
$\om\in X$
and a metric 
$d$
from the M\"obius structure for which
$\om$
is infinitely remote. We also use notation
$|xy|=d(x,y)$
for the distance between
$x$, $y\in X_\om$.

\begin{lem}\label{lem:homothety_vert_shift} Let
$\phi:X_\om\to X_\om$
be a pure homothety with
$\phi(o)=o$, $o\in X_\om$.
Then
$\phi$
preserves the fiber ($\K$-line)
$F$
through
$o$, $\phi(F)=F$.
In particular, every shift
$\eta_{xx'}:X_\om\to X_\om$
with
$x$, $x'\in F$
preserves
$F$.
\end{lem}

\begin{proof} As in Lemma~\ref{lem:pure_homothety}, we have
$\la b\circ\phi=b$
for any Busemann function
$b:X_\om\to\R$
with 
$b(o)=0$,
where
$\la$
is the coefficient of the homothety
$\phi$.
Since
$b(x)=0$
for every
$x\in F$,
we see that
$b\circ\phi(x)=0$,
and thus
$\phi(x)\in F$,
that is,
$\phi(F)=F$.
The assertion about a shift
$\eta_{xx'}$
follows now from the definition of
$\eta_{xx'}$.
\end{proof}

\begin{lem}\label{lem:vert_shift} Let
$\eta:X_\om\to X_\om$
be a shift that preserves a $\K$-line
$F\sub X_\om$.
Then 
$\eta$
preserves any other $\K$-line 
$F'\sub X_\om$.
\end{lem}

\begin{proof} Let
$b:X_\om\to\R$
be a Busemann function associated with an (oriented)
Ptolemy line
$l\sub X_\om$.
Then for any isometry
$\eta:X_\om\to X_\om$
the function
$b\circ\eta$
is a Busemann function associated with 
Ptolemy line 
$\eta^{-1}(l)$.
Thus for an arbitrary shift
$\eta:X_\om\to X_\om$,
we have
$b\circ\eta=b+c_b$,
where
$c_b\in\R$
is a constant depending on
$b$,
because the line 
$\eta^{-1}(l)$
is Busemann parallel to
$l$,
hence the function
$b\circ\eta$
is also a Busemann function of
$l$,
and the functions
$b$, $b\circ\eta$
differ by a constant. 

In our case, when
$\eta$
preserves a $\K$-line, this constant is zero,
$c_b=0$,
thus
$b\circ\eta=b$
for any Busemann function
$b:X_\om\to X_\om$.
Therefore,
$\eta$
preserves any $\K$-line.
\end{proof}

We define
$$d(F,F')=\inf\set{|xx'|}{$x\in F,\ x'\in F'$}$$
for $\K$-lines
$F$, $F'\sub X_\om$.

\begin{lem}\label{lem:klines_distance} Given $\K$-lines
$F$, $F'\sub X_\om$,
and
$x\in F$,
there is
$x'\in F'$
such that
$d(F,F')=|xx'|$.
\end{lem}

\begin{proof} Let
$x_i\in F$, $x_i'\in F'$
be sequences with 
$|x_ix_i'|\to d(F,F')$.
Using Lemma~\ref{lem:vert_shift}, we can assume that
$x_i=x$
for all
$i$.
Then the sequence
$x_i'$
is bounded, and by compactness of
$X$
it subconverges to
$x'\in F'$
with
$|xx'|=d(F,F')$.
\end{proof}

\begin{lem}\label{lem:line_dist_klines} For any $\K$-lines
$F$, $F'\sub X_\om$,
we have
$d(F,F')=|xx'|$,
where
$x=l\cap F$, $x'=l\cap F'$,
and 
$l$
is any Ptolemy line in
$X_\om$
that meets both
$F$, $F'$.
\end{lem}

\begin{proof} By Lemma~\ref{lem:rfoliation_semik}, the distance
$|xx'|$
is independent of the choice of
$l$.
By definition
$|xx'|\ge d(F,F')$.
By Lemma~\ref{lem:klines_distance}, there is
$x''\in F'$
with
$|xx''|=d(F,F')$.
The horosphere 
$H$
of (a Busemann function associated with)
$l$
through
$x'$
contains
$F'$,
in particular,
$x''\in H$.
Then
$|xx''|\ge|xx'|$
and hence,
$d(F,F')=|xx'|$.
\end{proof}

Let
$\pi_\om:X_\om\to B_\om$
be the canonical fibration, see sect.~\ref{sect:fibration}.
For
$b\in B_\om$,
we denote with
$F_b=\pi_\om^{-1}(b)$
the $\K$-line over 
$b$.
For
$b$, $b'\in B_\om$
we put
$|bb'|:=|xx'|$,
where
$x=l\cap F_b$, $x'=l\cap F_{b'}$,
and
$l\sub X_\om$
is any Ptolemy line that meets both
$F_b$
and
$F_{b'}$.
By property ($1_\K$),
such a line 
$l$
exists, by Lemma~\ref{lem:line_dist_klines}, the number
$|bb'|$
is well defined, and the function
$(b,b')\mapsto|bb'|$
is a metric on
$B_\om$.
This metric is said to be {\em canonical}.

\begin{pro}\label{pro:base_metric} The canonical projection
$\pi_\om:X_\om\to B_\om$
is a 1-Lipschitz submetry with respect to the canonical
metric on
$B_\om$.
Furthermore,
$B_\om$
is a geodesic metric space with the property
that through any two distinct points
$b$, $b'\in B_\om$
there is a unique geodesic line in
$B_\om$.
\end{pro}

\begin{proof}
It follows from Lemma~\ref{lem:line_dist_klines} that the map
$\pi_\om$
is 1-Lipschitz. Let
$D=D_r(o)$
be the metric ball in
$X_\om$
of radius
$r>0$
centered at a point
$o\in X_\om$, $D'\sub B_\om$
the metric ball of the same radius
$r$
centered at
$\pi_\om(o)$.
The inclusion
$D'\sub\pi_\om(D)$
follows from the definition of the metric of
$B_\om$.
The opposite inclusion
$D'\sups\pi_\om(D)$
holds because
$\pi_\om$
is 1-Lipschitz. Thus
$\pi_\om:X_\om\to B_\om$
is a 1-Lipschitz submetry.

Furthermore, by Lemma~\ref{lem:line_dist_klines},
the projection
$\pi_\om$
restricted to every Ptolemy line in
$X_\om$
is isometric, and thus by property~($1_\K$), the base
$B_\om$
is a geodesic metric space. Moreover, it follows
from ($1_\K$) that through any two distinct
points
$b$, $b'\in B_\om$
there is a unique geodesic line in
$B_\om$.
\end{proof}

\begin{cor}\label{cor:pi_submetry} For any homothety
$\phi:X_\om\to X_\om$,
the induced map
$\pi_\ast(\phi):B_\om\to B_\om$
is a homothety with the same dilatation coefficient.
\qed
\end{cor}

\begin{pro}\label{pro:base_euclidean} The base
$B_\om$
is isometric to an Euclidean
$\R^k$
for some 
$k\le\dim X$. 
\end{pro}

\begin{proof} Any Busemann function
$b:X_\om\to\R$
is affine on Ptolemy lines by Proposition~\ref{pro:busemann_affine}.
By definition,
$b$
is constant on the fibers of
$\pi_\om$,
thus it determines a function
$\ov b:B_\om\to\R$
such that
$\ov b\circ\pi_\om=b$.
This function is affine on geodesic lines in
$B_\om$
because every geodesic line
$\ov l\sub B_\om$
is of the form
$\ov l=\pi_\om(l)$
for some Ptolemy line
$l\sub X_\om$,
and each unit speed parameterization
$c:\R\to X_\om$
of
$l$
induces the unit speed parameterization
$\ov c=\pi_\om\circ c$
of
$\ov l$.
Then
$\ov b\circ\ov c=\ov b\circ\pi_\om\circ c=b\circ c$
is an affine function on
$\R$.

We fix a base point 
$o\in X_\om$
and a maximal collection
$\cL=\{l_1,\dots,l_k\}$
of mutually orthogonal Ptolemy lines of
$X_\om$
through
$o$.
Let
$b_1,\dots,b_k$
be Busemann functions of the lines
$l_1,\dots,l_k$
respectively which vanish at
$o$.
We denote with 
$\ov l_i$
the projection of
$l_i$
to
$B_\om$,
and with
$\ov b_i:B_\om\to\R$
the function corresponding to
$b_i$, $i=1,\dots,k$.
By Proposition~\ref{pro:orthogonal_base}, the functions
$b_1,\dots,b_k$
separates fibers in
$X_\om$.
Thus the functions
$\ov b_1,\dots,\ov b_k$
separates points of 
$B_\om$,
that is, for each
$z$, $z'\in B_\om$
there is 
$i$
with
$\ov b_i(z)\neq\ov b_i(z')$.
Therefore, the continuous map
$h:B_\om\to\R^k$, $h(z)=(\ov b_1(z),\dots,\ov b_k(z))$
is injective. This map is surjective by the same argument
as in the proof of the property ($1_\K$), and it introduces
coordinates on
$B_\om$.
We compute the distance on
$B_\om$
in these coordinates. Applying a shift if necessary,
see Corollary~\ref{cor:pi_submetry}, we consider W.L.G. the distance
$|\ov o\,\ov z|$
for every
$z\in X_\om$,
where
$\ov z=\pi_\om(z)$.
By ($1_\K$) there is a unique Ptolemy line 
$l\sub X_\om$
through
$o$
that hits the fiber
$F_z$
through
$z$.
It follows from our definitions that for
$\al_i=\slope(l,l_i)$
we have
$\ov b_i(\ov z)=\al_i|\ov o\,\ov z|$, $i=1,\dots,k$.
By Lemma~\ref{lem:max_orthogonal},
$\sum_i\al_i^2=1$,
thus
$|\ov o\,\ov z|^2=\sum_i\ov b_i^2(\ov z)$.
This shows that 
$B_\om$
is isometric to an Euclidean
$\R^k$.
We have
$k\le\dim X$,
because
$\pi_\om:X_\om\to B_\om$
is a 1-Lipschitz submetry. 
\end{proof}

\section{Extension of M\"obius automorphisms of circles}
\label{sect:extension_circles}

\subsection{Distance and arclength parameterizations of a circle}
\label{subsect:dist_unit_speed_param}

In this section, we establish existence of a distance
parameterization in a Ptolemy circle and study its 
relationship with an arclength parameterization. A distance
parameterization is convenient to obtain an important estimate
(\ref{eq:quadratic_reduce}) below. On the other hand, in 
an application of this estimate we use computation of slops,
which are most convenient to do in an arclength parameterization.

In what follows, we consider a (bounded) Ptolemy circle
$\si\sub X_\om$
and points 
$x$, $y\in\si$
with
$a:=|xy|>0$.

\begin{lem}\label{lem:dist_parameter} Let 
$\si_+$, $\si_-$
be the two components of
$\si\sm\{x,y\}$.
Then for all
$0<t<a$ 
there exists exactly one point
$x_t^+\in\si_+$
(resp. 
$x_t^-\in\si_-$) 
with
$|xx_t^+|=|xx_t^-|=t$.
Therefore
$\ga:(-a,a)\to\si$
with
$\ga(0)=x$, $\ga(t)=x_t^+$
for 
$t>0$,
and
$\ga(t)=x_{-t}^-$
for 
$t<0$
parameterizes a neighborhood of
$x$ 
in 
$\si$.
\end{lem}

\begin{proof} The existence of a point
$x_t^+ \in\si_+$
with
$|xx_t^+|=t$
is clear by continuity. To prove uniqueness consider points
$x<p<q<y$ 
in this order on
$\si_+$
and assume
$b:=|xp|<a=|xy|$.
Let
$c:=|xq|$, $\la_a:=|pq|$, $\la_b:=|qy|$, $\la_c:=|py|$.

The Ptolemy equality and the triangle inequality give
$$a\la_a+b\la_b=c\la_c\le c(\la_a+\la_b).$$
Therefore
$$c \ge \frac{\la_a}{\la_a+\la_b}a+\frac{\la_b}{\la_a+\la_b}b >b$$
where the last equality holds, since 
$a>b$.
In particular
$|xp|\neq |xq|$.
\end{proof}

In what follows, we use the parameterization
$\ga:(-a,a)\to\si$
of a neighborhood of
$x\in\si$,
and call it a {\em distance} parameterization.

\begin{lem}\label{lem:smooth_concave} The function
$g(t):=|\ga(t)y|$
is concave and $C^1$-smooth on
$(-a,a)$.
\end{lem}

\begin{proof} For
$-a\le t_1<t_2\le a$
the Ptolemy equality for the points
$x,\ga(t_1),\ga(t_2),y$
implies
$$t_2g(t_1)-t_1g(t_2)=a|\ga(t_1)\ga(t_2)|.$$
Thus for 
$-a\le t_1<t_2<t_3\le a$
the triangle inequality 
$|\ga(t_1)\ga(t_3)|\le|\ga(t_1)\ga(t_2)|+|\ga(t_2)\ga(t_3)|$
implies
$$t_3g(t_2)-t_2g(t_3)+t_2g(t_1)-t_1g(t_2)\,\ge\,t_3g(t_1)-t_1g(t_3)$$
which is equivalent to
$$\frac{g(t_2)-g(t_1)}{t_2-t_1}\,\ge\, \frac{g(t_3)-g(t_2)}{t_3-t_2}.$$
Therefore,
$g$
is concave. It follows, in particular, that
$g$
has the left 
$g_-'(t)$
and the right derivative 
$g_+'(t)$
at every
$t\in(-a,a)$, $g_-'(t)\ge g_+'(t)$
and these derivatives are nonincreasing,
$g_+'(t)\ge g_-'(t')$
for
$t<t'$.
Furthermore,
$g_-'(t)\to g_-'(t')$
as
$t\nearrow t'$,
$g_+'(t')\to g_+'(t)$
as
$t'\searrow t$.
These are standard well known facts about concave functions, 
see e.g. \cite{H-UL}.

We fix
$t_0\in(-a,a)$
and consider the Ptolemy line 
$l\sub X_\om$
tangent to
$\si$
at
$x_0=\ga(t_0)$.
We assume that
$l$
is oriented and that its orientation is compatible with
the orientation of
$\si$
given by the distance parameterization
$\ga$.
Let
$c:\R\to X_\om$
be the unit speed parameterization of
$l$
compatible with the orientation,
$c(0)=x_0$.
By Corollary~\ref{cor:tangent_unique},
$y\not\in l$. 
By Lemma~\ref{lem:smooth_convex}, the function
$\wt g(s)=|c(s)y|$, $s\in\R$,
is
$C^1$-smooth.
If
$t_0=0$,
then
$g_-'(t_0)=\frac{d\wt g}{ds}(0)=g_+'(t_0)$,
because
$l$
is tangent to 
$\si$
at
$x_0=x$,
and thus
$g$
is differentiable at
$t_0=0$.

Consider now the case
$t_0\neq 0$.
Then again by Corollary~\ref{cor:tangent_unique},
$x\not\in l$,
thus the function
$\wt f(s)=|c(s)x|$, $s\in\R$,
is $C^1$-smooth.
We show that
$\frac{d\wt f}{ds}(0)\neq 0$.
We suppose W.L.G. that
$t_0>0$.
Then for all
$t_1\in(0,t_0)$
sufficiently close to
$t_0$
we have
$\frac{d\wt h}{ds}(0)\neq 0$,
where
$\wt h(s)=|x_1c(s)|$, $x_1=\ga(t_1)$.
We fix such a point 
$t_1$, 
and using Lemma~\ref{lem:dist_parameter}
consider the distance parameterization of a neighborhood of
$x_0=\ga(t_0)$
in
$\si$, $|z(\tau)x_1|=\tau$
for all
$z\in\si$
sufficiently close to
$x_0$.
Then 
$t=t(\tau)$
and the function
$f(\tau)=|x\ga\circ t(\tau)|$
is concave by the first part of the proof. 
Since the functions
$\wt f(s)=|xc(s)|$, $\wt h(s)=|x_1c(s)|$
are $C^1$-smooth, and
$\frac{d\wt h}{ds}(0)\neq 0$,
the function 
$\wt f(\tau)=\wt f\circ\wt h^{-1}(\tau)$
is $C^1$-smooth
in a neighborhood of
$\tau_0=|x_1x_0|$
by the inverse function theorem. Therefore,
$f_-'(\tau_0)=\frac{d\wt f}{ds}(0)=f_+'(\tau_0)$
because
$l$
is tangent to
$\si$
at
$x_0$.
The assumption
$\frac{d\wt f}{ds}(0)=0$
implies
$\frac{df}{d\tau}(\tau_0)=0$.
By concavity, 
$\tau_0$
is a maximum point of the function
$f(\tau)$, 
and there are different
$\tau$, $\tau'$
arbitrarily close to
$\tau_0$
with
$f(\tau)=f(\tau')$.
This contradicts properties of the parameterization
$\ga(\tau)=\ga\circ t(\tau)$.
Hence,
$\frac{d\wt f}{ds}(0)\neq 0$.

Again, by the inverse function theorem, the function
$\wt g(t)=\wt g\circ\wt f^{-1}(t)$
is
$C^1$-smooth 
in a neighborhood of
$t_0$.
However,
$\frac{d\wt g}{dt}(t_0)$
coincides with the left as well as with the right
derivative of the function
$g$
at
$t_0$
because
$l$
is tangent to
$\si$
at
$x_0$.
Therefore,
$g$
is differentiable at
$t_0$.
It follows from continuity properties of
one-sided derivatives of concave functions that the derivative
$g'$
is continuous, i.e.,
$g$
is 
$C^1$-smooth.
\end{proof}

\begin{lem}\label{lem:circle_rectifiable} Every Ptolemy circle
$\si\sub X_\om$
is rectifiable and
$$L(xx')=|xx'|+o(|xx'|^2)$$
as
$x'\to x$
in
$\si$,
where
$L(xx')$
is the length of the (smallest) arc
$xx'\sub\si$.
\end{lem}

\begin{proof} We fix 
$y\in\si$, $y\neq x$,
and introduce a distance parameterization
$\ga:(-a,a)\to\si$
of a neighborhood of
$x=\ga(0)$
in
$\si$.
Rescaling the metric of
$X_\om$
we assume that
$a=|xy|=1$
for simplicity of computations. We use notation
$d(z,z')=|zz'|$
for the distance in
$X_\om$,
and
$|zz'|_y$
for the distance in
$X_y$,
assuming that
$$|zz'|_y=\frac{|zz'|}{|zy||z'y|}$$
is the metric inversion of the metric 
$d$.

Recall that 
$\si\sm y$
is a Ptolemy line in
$X_y$.
Thus for a given
$r\in(0,1)$,
and for every partition
$0=t_0\le\dots\le t_n=r$
we have
$$\sum_i|\ga(t_i)\ga(t_{i+1})|\le\La|xx_r|_y,$$
where
$\La=\max\set{g^2(t)}{$0\le t\le r$}$, $g(t)=|\ga(t)y|$,
$x_r=\ga(r)$.
Hence
$\si$
is rectifiable and
$L(xx_r)\le\La|xx_r|_y$.
Moreover, using
$|\ga(t_i)\ga(t_{i+1})|=g(t_i)g(t_{i+1})|\ga(t_i)\ga(t_{i+1})|_y$,
we actually have
$$L(r)=L(xx_r)=\int_0^{r/g(r)}g^2(s)ds,$$
where
$g(s)=g\circ t(s)$
with
$s=\frac{t}{g(0)g(t)}=t/g(t)$.
Recall that the function
$g(t)$
is $C^1$-smooth by Lemma~\ref{lem:smooth_concave}. Then
$ds=\frac{g(t)-tg'(t)}{g^2(t)}dt$
and
$\frac{dt}{ds}=\frac{g^2(t)}{g(t)-tg'(t)}$,
in particular,
$\frac{dt}{ds}(0)=1$.

Using developments
$g(s)=1+\frac{dg}{ds}(0)s+o(s)$,
$g^2(s)=1+2\frac{dg}{ds}(0)s+o(s)$,
we obtain
$$L(r)=\frac{r}{g(r)}+\frac{dg}{ds}(0)\frac{r^2}{g^2(r)}+o(r^2),$$
where
$\frac{dg}{ds}(0)=g'(0)\frac{dt}{ds}(0)=g'(0)$.
Since
$g(r)=1+g'(0)r+o(r)$, $g^2(r)=1+2g'(0)r+o(r)$,
we finally have
$$L(r)=r(1-g'(0)r)+g'(0)r^2(1-2g'(0)r)+o(r^2)=r+o(r^2).$$
Hence
$L(xx')=|xx'|+o(|xx'|^2)$
as
$x'\to x$
in
$\si$. 
\end{proof}

\subsection{Slope of Ptolemy circles}
\label{subsect:slope_circles}

If oriented Ptolemy circles
$\si$, $\si'\sub X$
are disjoint, then their slope is not determined. Assume now that
$\om\in\si\cap\si'$.
Then 
$\slope_\om(\si,\si)\in[-1,1]$
is defined as the slope of oriented Ptolemy lines
$\si\sm x$, $\si'\sm x\sub X_\om$.
This is well defined and symmetric,
$\slope_\om(\si,\si')=\slope_\om(\si',\si)$,
by Lemma~\ref{lem:slope_symmetry}. In the case
$\slope_\om(\si,\si')=\pm 1$,
the Ptolemy circles are tangent to each other at
$\om$,
having compatible ($-1$) or opposite ($+1$) orientations.
This means that in any space 
$X_{\om'}$
with 
$\om'\neq\om$,
the Ptolemy circles
$\si\sm\om'$, $\si'\sm\om'\sub X_{\om}$
have a common tangent line at
$\om$.

More generally, let
$l$, $l'\sub X_{\om'}$
be the tangent lines to
$\si$, $\si'$
respectively at
$\om$, 
oriented according to the orientations of
$\si$, $\si'$.
Then
$\slope_\om(\si,\si')=\slope(l,l')$
because in the space 
$X_\om$
the Ptolemy line
$\si\sm\om$
is Busemann parallel to
$l\sm\om$,
and
$\si'\sm\om$
is Busemann parallel to
$l'\sm\om$,
and we can apply Lemma~\ref{lem:paraline_busemann}.

Furthermore, if
$\si$, $\si'$
have two distinct common points,
$\om$, $\om'\in\si\cap\si'$,
then
$\slope_\om(\si,\si')=\slope_{\om'}(\si,\si')$.
This follows from Lemma~\ref{lem:opposite_slopes}.

\begin{lem}\label{lem:quaratic_excess} Assume that a (bounded) oriented
Ptolemy circle
$\si\sub X_\om$
has two different points in common with a Ptolemy line 
$l\sub X_\om$, $x$, $y\in\si\cap l$,
and the line 
$l$
is oriented from
$y$
to
$x$.
Let
$\si_+\sub\si$
be the arc of
$\si$
from
$x$
to
$y$
chosen according to the orientation of
$\si$.
Let
$x_t$, $y_t$
be the distance parameterizations of neighborhoods of
$x$, $y$
respectively
such that
$x_t$, $y_t\in\si_+$
for 
$t>0$, $|x_tx|=t=|y_ty|$.
Furthermore, let
$b^\pm:X_\om\to\R$
be the opposite Busemann functions of
$l$
normalized by
$b^+(x)=0$, $b^+(y)=-a$, $b^-(x)=-a$, $b^-(y)=0$,
where
$a=|xy|$.
Then
\begin{equation}\label{eq:quadratic_reduce}
b^+(x_t)+b^-(y_t)\le 2\al t-\frac{1}{a}(1-\al^2)t^2
\end{equation}
for all 
$t>0$
in the domain of the parameterizations, where
$\al=\slope_x(\si,l)=\slope_y(\si,l)$. 
\end{lem}

\begin{proof} By Lemma~\ref{lem:smooth_concave} the functions
$g(t)=|x_ty|$, $f(t)=|y_tx|$
are $C^1$-smooth and concave. Furthermore, their first derivatives
at 0,
$g'(0)$
and
$f'(0)$,
coincide with first derivatives of the distance functions to the 
respective tangent lines,
$g'(0)=\wt g'(0)$
and
$f'(0)=\wt f'(0)$,
where
$\wt g(s)=|c_x(s)y|$, $\wt f(s)=|c_y(s)x|$,
and the unit speed parameterizations of the tangent lines
$l_x$
to
$\si$
at
$x$
and
$l_y$
to
$\si$
at 
$y$
are chosen compatible with the distance parameterizations
$x_t$, $y_t$
so that
$c_x(0)=x$, $c_y(0)=y$.

Using that
$l$
is oriented from
$y$
to
$x$,
and
$\si_+$
from
$x$
to
$y$
and applying equation~(\ref{eq:re_first_variation}), we find
$\wt g'(0)=\slope(l_x,l)$.
By the same equation~(\ref{eq:re_first_variation}) we have
$\wt f'(0)=\slope(l_y,-l)=-\slope(l_y,l)$.
The sign
$-1$
appears because the orientation of
$l$
from
$x$
to
$y$
is opposite to the chosen orientation.
Note that the orientation of
$l_x$
is compatible with that of
$\si$,
while the orientation of
$l_y$
is opposite to that of 
$\si$.
Therefore,
$g'(0)=\al=\slope(\si,l)$
and
$f'(0)=-\slope(l_y,l)=\slope(\si,l)=\al$.

Using concavity we obtain
$g(t)\le g(0)+g'(0)t=a+\al t$
and similarly
$f(t)\le a+\al t$
for all 
$0\le t<a$.
The Ptolemy equality applied to the ordered quadruple
$(x,x_t,y_t,y)\sub\si$
gives
$g(t)f(t)=t^2+a|x_ty_t|$,
hence
$$|x_ty_t|\le a+2\al t-\frac{1}{a}(1-\al^2)t^2.$$

Let
$H_t^+$, $H_t^-$
be the horospheres of
$b^+$, $b^-$
through
$x_t$, $y_t$
respectively,
$x_t\in H_t^+$, $y_t\in H_t^-$.
Since
$X$
is Busemann flat, see Lemma~\ref{lem:smooth_convex},
$H_t^+$
is also a horosphere of
$b^-$
and
$H_t^-$
is a horosphere of
$b^+$.
Thus
$|x_ty_t|\ge\xi$,
where
$\xi$
is the distance between the points
$l\cap H_t^+$, $l\cap H_t^-$,
$\xi=a+b^+(x_t)+b^-(y_t)$.
Therefore,
$$b^+(x_t)+b^-(y_t)\le 2\al t-\frac{1}{a}(1-\al^2)t^2.$$
\end{proof}

\subsection{Extension property}
\label{subsect:extension_property}

Surprisingly, the proof of the extension property
(${\rm E}_2$) is based on study of second order properties 
of Ptolemy circles like Lemma~\ref{lem:quaratic_excess}.

\begin{pro}\label{pro:extesion_property} Any compact Ptolemy space 
with properties (E) and (I) possesses the extension property
(${\rm E}_2$), see sect.~\ref{subsect:space_inversions}.
\end{pro}

\begin{lem}\label{lem:extend_cirle_map} Any M\"obius
automorphism of any Ptolemy circle
$\si\sub X$
preserving orientations extends to a M\"obius automorphism of
$X$. 
\end{lem}

\begin{proof} We represent
$\si$
as the boundary at infinity of the real hyperbolic plane,
$\si=\di\hyp^2$.
Then the group 
$G$
of preserving orientations M\"obius automorphisms of
$\si$
is identified with the group of preserving orientations
isometries of
$\hyp^2$.
The last is generated by central symmetries, and
any central symmetry of
$\hyp^2$
induces an s-inversion of
$\si$.
Thus
$G$
is generated by s-inversions of
$\si$.

Now, any s-inversion of
$\si$
can be obtained as follows. Take distinct
$\om$, $\om'\in\si$
and a metric sphere
$S\sub X$
between
$\om$, $\om'$.
Then an s-inversion
$\phi=\phi_{\om,\om',S}:X\to X$
restricts to an s-inversion of
$\si$.
Thus any M\"obius automorphism of
$\si$
from
$G$
extends to a M\"obius automorphism of
$X$.
\end{proof}

The group
$\aut X$
of M\"obius automorphisms of
$X$
is noncompact: a sequence of homotheties of
$X_\om$
with coefficients
$\la_i\to\infty$
and with the same fixed point has no converging subsequences.
However, we have the following standard compactness result.

\begin{lem}\label{lem:nondegenerate_morphism} Assume that
for a nondegenerate triple
$T=(x,y,z)\sub X$
and for a sequence 
$\phi_i\in\aut X$
the sequence
$T_i=\phi_i(T)$
converges to a nondegenerate triple
$T'=(x',y',z')\sub X$.
Then there exists
$\phi\in\aut X$
with
$\phi(T)=T'$.
\end{lem}

\begin{proof} For every
$u\in X\sm T$
the quadruple
$Q=(T,u)$
is nondegenerate in the sense that its cross-ratio triple
$\crt(Q)=(a:b:c)$
has no zero entry. Since
$\crt(\phi_i(Q))=\crt(Q)$,
any accumulation point 
$u'$
of the sequence
$u_i=\phi_i(u)$
is not in
$T'$.
Thus for the nondegenerate triple
$S=(x,y,u)$
any sublimit 
$S'=(x',y',u')$
of the sequence
$S_i=\phi(S)$
is nondegenerate. Applying the same argument to any
$v\in X\sm S$,
we observe that the sequences
$u_i$, $v_i=\phi_i(v)$
have no common accumulation point. This shows that
any limiting map
$\phi$
of the sequence
$\phi_i$,
obtained e.g. by taking a nonprincipal ultra-filter limit,
is injective, and hence it is a M\"obius automorphism of
$X$
with
$\phi(T)=T'$. 
\end{proof}

\begin{pro}\label{pro:transitive_circle} The group of
M\"obius automorphisms of
$X$
acts transitively on the set of the oriented Ptolemy circles in
$X$. 
In particular, for any oriented circle
$\si\sub X$
there is a M\"obius automorphism
$\phi:X\to X$
such that
$\phi(\si)=\si$
and 
$\phi$
reverses the orientation of
$\si$.
\end{pro}

\begin{proof} We fix an oriented Ptolemy circle
$\si\sub X$
and distinct points
$x$, $y\in\si$.
For an oriented circle
$\si_0\sub X$
we denote with 
$A$
the set of all the circles
$\phi(\si_0)$, $\phi\in\aut X$,
with the induced orientation which pass also through
$x$
and
$y$.
Let
$$\al=\inf\set{\slope(\si,\si')}{$\si'\in A$}.$$
By two-point homogeneity property, see Proposition~\ref{pro:two_point_homogeneous},
$A\neq\es$.
Applying Lemma~\ref{lem:nondegenerate_morphism} we find
$\si'\in A$
with
$\slope(\si,\si')=\al$.
Next, we show that
$\al<1$.
Applying a shift we first make
$\si_0$
disjoint with 
$\si$.
Taking a point 
$\om\in\si_0$
as infinitely remote, we consider all Ptolemy lines in
$X_\om$
which are Busemann parallel to
$\si_0\sm\om$
and intersect
$\si$.
Since 
$\si$
is bounded in
$X_\om$,
at least one of them,
$l$,
is not tangent to
$\si$.
Then
$\slope(\si,l)<1$.
This 
$l$
can be obtained from
$\si_0\sm\om$
by a shift. Applying another shift to
$l$
in the space 
$X_{\om'}$
with
$\om'\in\si\cap l$
(this does not change the slope), we can assume that
$x\in\si\cap l$.
Repeating this in the space 
$X_x$,
we find
$\wt\si\in A$
with 
$\slope(\si,\wt\si)<1$.
Thus
$\al<1$.

We show that
$\al=-1$.
Then 
$\si=\phi(\si_0)$
as oriented Ptolemy circles for some
$\phi\in\aut X$,
which would complete the proof.

Assume that
$\al>-1$.
The points
$x$, $y$
subdivide each of the circles
$\si$, $\si'$
into two arcs. We choose an arc
$\si_+\sub\si$
leading from
$x$
to
$y$
according to the orientation of
$\si$, 
and an arc
$\si_+'\sub\si'$
leading from
$y$
to
$x$
according to the orientation of
$\si'$.
Taking a point 
$\om\in\si'$
inside of the opposite to
$\si_+'$
arc, we see that
$l=\si'\sm\om$
is a Ptolemy line in the space
$X_\om$
oriented from
$y$
to
$x$.

Given
$x'\in\si_+$,
for every Ptolemy line 
$l_{x'}\sub X_\om$
through
$x'$,
which is Busemann parallel to
$l$
and is oriented as
$l$,
we have
$\slope(\si,l_{x'})\ge\al$
by the definition of
$\al$,
because by the same argument as above
$l_{x'}$
can be put in the set 
$A$
without changing the slope.

Let
$b^\pm:X_\om\to\R$
be the opposite Busemann functions of
$l$
normalized by
$b^+(x)=0$, $b^+(y)=-a$, $b^-(x)=-a$, $b^-(y)=0$,
where
$a=|xy|$.
Using Lemma~\ref{lem:circle_rectifiable} we consider
for a sufficiently small
$\ep>0$
arclength parameterizations
$c_x$, $c_y:(-\ep,\ep)\to\si$
with
$c_x(0)=x$, $c_y(0)=y$, $c_x(s)$, $c_y(s)\in\si_+$
for
$s>0$,
of neighborhoods of
$x$, $y$
respectively in
$\si$.
Since Busemann functions on
$X_\om$
are affine and hence differentiable along Ptolemy lines,
and since the derivative
$\frac{db^+\circ c_x}{ds}(s)$
coincides with the derivative of
$b^+$
along the tangent line to
$\si$
at
$c_x(s)$,
we have
$$\frac{db^+\circ c_x}{ds}(s)=\slope(\si,l_{c(s)})\ge\al.$$
For a sufficiently small
$t>0$
let
$x_t\in\si_+$
be a point at the distance
$t$
from
$x$, $|x_tx|=t$, $x_t=c_x(\tau)$
for some
$\tau=\tau(t)$.
By integrating we obtain 
$b^+(x_t)\ge\al L(xx_t)\ge\al t$.
A similar argument shows that
$b^-(y_t)\ge\al t$,
where
$y_t=c_y(\tau')$
for some
$\tau'=\tau'(t)$, $|yy_t|=t$.
Therefore,
$b^+(x_t)+b^-(y_t)\ge 2\al t$.
This contradicts the estimate~(\ref{eq:quadratic_reduce})
of Lemma~\ref{lem:quaratic_excess}. Thus
$\al=-1$.
\end{proof}

\begin{proof}[Proof of Proposition~\ref{pro:extesion_property}]
Given a M\"obius map
$\psi:\si\to\si'$
between Ptolemy circles
$\si$, $\si'\sub X$,
we choose orientations of
$\si$, $\si'$
so that
$\psi$
preserves the orientations. By Proposition~\ref{pro:transitive_circle}
there is
$\phi\in\aut X$
with 
$\phi(\si)=\si'$
preserving the orientations. Then
$\phi^{-1}\circ\psi:\si\to\si$
preserves the orientation of
$\si$,
and hence it extends by Lemma~\ref{lem:extend_cirle_map} to
$\phi'\in\aut X$,
$\phi'|\si=\phi^{-1}\circ\psi$.
Then
$\phi\circ\phi'\in\aut X$
is a required M\"obius automorphism.
\end{proof}

\section{Topology of the space $X$ and of $\K$-lines}
\label{sect:topology_space}

\subsection{Groups of shifts}
\label{subsect:group_shifts}

Recall that by Lemma~\ref{lem:pure_homothety} a shift
$\eta:X_\om\to X_\om$
is an isometry that preserves every foliation of 
$X_\om$
by (oriented) Busemann parallel Ptolemy lines. Clearly, the shifts of
$X_\om$
form a group which we denote with 
$N_\om$.
Then
$N_\om$
is a subgroup of the group
$\aut X$
of the M\"obius automorphisms of
$X$.

\begin{lem}\label{lem:simply_transitivity_shifts} The group
$N_\om$
acts simply transitively on
$X_\om$. 
\end{lem}

\begin{proof} Given
$x$, $x'\in X_\om$,
the shift
$\eta_{xx'}$
moves
$x$
to
$x'$, $\eta_{xx'}(x)=x'$,
by construction, see sect.~\ref{subsect:parallel_lines_pure_homothethies_shifts}. Thus
$N_\om$
acts transitively on
$X_\om$.

Assume that
$\eta(x)=x$
for some shift
$\eta:X_\om\to X_\om$
and some
$x\in X_\om$.
We denote with 
$V$
the fixed point set of
$\eta$, $\eta(y)=y$
for every
$y\in V$.
We show that
$V=X$.
Note that every Ptolemy line 
$l\sub X_\om$,
which meets
$V$,
is contained in
$V$
because the isometry
$\eta$
preserves every foliation of
$X_\om$
by oriented Busemann parallel Ptolemy lines.
Next, we note that every Ptolemy circle
$\si\sub X_\om$,
which meets
$V$
at two different points, is contained in
$V$.
Indeed, the tangent to 
$\si$
lines at these points are contained in
$V$.
However,
$\si$
is uniquely determined by its tangent line and any other
its point, see Corollary~\ref{cor:tangent_unique}.

Assume that
$V\neq X$, and let
$\om'\in X\sm V$.
Since 
$V$
is closed,
$\om'$
is contained in
$X\sm V$
together with a neighborhood
$U$
of
$\om'$.
Let
$\phi:X\to X$
is a space inversion with 
$\phi(\om')=\om$, $W=\phi(V)$.
Then
$W$
misses the neighborhood
$\phi(U)$
of
$\om$,
and thus
$W$
is compact in
$X_\om$.
Furthermore,
$W$
contains together with any two different points every
Ptolemy circle through these points. 
The image
$\ov W=\pi_\om(W)\sub B_\om$
under the canonical projection
$\pi_\om:X_\om\to B_\om$
is compact since
$W$
is compact in
$X_\om$.
On the other hand, given
$z\in W$
and a Ptolemy line 
$l\sub X_\om$
through
$z$,
there is a Ptolemy circle
$\si\sub W$
through
$z$
with the tangent line 
$l$.
Indeed, for every
$z'\in W$, $z'\neq z$,
by Corollary~\ref{cor:tangent_unique}
there is a unique circle through
$z$, $z'$
that is tangent to
$l$.
This circle is contained in
$W$
by properties of
$W$.
It follows that 
$\ov W$
is open in
$B_\om$
and thus
$\ov W=B_\om$.
This contradicts the fact that 
$\ov W$
is compact.

Therefore,
$V=X$, 
and thus
$\eta=\id$,
i.e. the group
$N_\om$
acts simply transitively on
$X_\om$.
\end{proof}

We fix
$o\in X_\om$
and using Lemma~\ref{lem:simply_transitivity_shifts} identify
$N_\om$
with
$X_\om$
by
$\eta\mapsto\eta(o)$.
Then
$N_\om$
is a locally compact topological group.

An automorphism
$\tau:N_\om\to N_\om$
is said to be {\em contractible} if for every
$\eta\in N_\om$
we have
$\lim_{n\to\infty}\tau^n(\eta)=\id$.
If
$N_\om$
admits a contractible automorphism, then 
$N_\om$
is also said to be {\em contractible}.

\begin{lem}\label{lem:contract_auto} There is a contractible
automorphism
$\tau:N_\om\to N_\om$. 
\end{lem}

\begin{proof} We take any pure homothety
$\phi:X_\om\to X_\om$
with
$\phi(o)=o$
and with the coefficient
$\la\in(0,1)$.
Then we define
$\tau(\eta)=\phi\circ\eta\circ\phi^{-1}$.
The map 
$\eta'=\tau(\eta):X_\om\to X_\om$
is an isometry preserving every foliation of
$X_\om$
by Busemann parallel Ptolemy lines, i.e.
$\eta'$
is a shift, and it is clear that
$\tau$
is an automorphism of
$N_\om$.

For the sequence of shifts
$\eta_n=\tau^n(\eta)$
we have
$\eta_n(o)=\phi^n\circ\eta(o)\to o$
as
$n\to\infty$.
Thus
$\eta_n$
converges to a shift
$\eta_\infty$
with
$\eta_\infty(o)=o$,
hence,
$\eta_\infty=\id$. 
\end{proof}

\begin{cor}\label{cor:nilponent_lie_group} The group
$N_\om$
is a simply connected nilpotent Lie group. In particular, the space 
$X_\om$
is homeomorphic to
$\R^n$,
and the space 
$X$
is homeomorphic to the sphere
$S^n$
with 
$n=\dim X$. 
\end{cor}

\begin{proof} The group
$N_\om$
is connected and locally compact because the space 
$X_\om$
is. By Lemma~\ref{lem:contract_auto},
$N_\om$
is contractible. Then by \cite[Corollary~2.4]{Sieb}
$N_\om$
is a simply connected nilpotent Lie group. 
\end{proof}

We denote with 
$Z_\om$
a subgroup in
$N_\om$
which consists of all shifts
$\eta\in N_\om$
acting identically on the base 
$B_\om$, $\pi_\ast(\eta)=\id$,
where
$\pi_\ast(\eta):B_\om\to B_\om$
is the shift induced by the projection
$\pi_\om:X_\om\to B_\om$,
see Corollary~\ref{cor:pi_submetry}. Every
$\eta\in Z_\om$
preserves every fiber ($\K$-line) of 
$\pi_\om$,
see Lemma~\ref{lem:homothety_vert_shift} and 
Lemma~\ref{lem:vert_shift}.

\begin{pro}\label{pro:fiber_nilpotent_lie_group} The group
$Z_\om$
acts simply transitively on every $\K$-line
$F\sub X_\om$,
and thus it is a contractible, connected, locally compact 
topological group. Therefore,
$Z_\om$
is a simply connected nilpotent Lie group, and
$F$
is homeomorphic to
$\R^p$
for some
$0\le p<n$. 
\end{pro}

\begin{proof} The group
$Z_\om$
acts transitively on
$F$
by Lemma~\ref{lem:homothety_vert_shift}. The action is
simply transitive by Lemma~\ref{lem:simply_transitivity_shifts}.
We fix 
$o\in F$
and identify
$Z_\om$
with
$F$
by
$\eta\mapsto\eta(o)$.
By the same argument as in Lemma~\ref{lem:contract_auto} we see
that the group
$Z_\om$
is contractible. Furthermore,
$F$
is locally compact. Given
$x$, $x'\in F$,
there is a Ptolemy circle
$\si\sub X_\om$
through
$x$, $x'$.
By ($1_\K$), through any point 
$z\in\si$
there is a uniquely determined Ptolemy line that hits
$F$.
This defines a continuous map 
$\si\to F$.
Thus
$F$
is linearly connected. Hence,
$Z_\om$
is a contractible, locally compact, connected topological group.
By \cite[Corollary~2.4]{Sieb},
$Z_\om$
is a simply connected nilpotent Lie group, and thus
$F$
is homeomorphic to
$\R^p$
for some
$0\le p\le n$.
In fact 
$p<n$
because
$X$
contains Ptolemy circles and thus
$k=\dim B_\om>0$,
while
$n=k+p$.
\end{proof}

To complete the proof of Theorem~\ref{thm:basic_ptolemy} it 
remains to show that if
$k=1$,
then
$X=\wh\R$.
In the following section, we establish a more general fact from which 
this property follows immediately.

\subsection{Non-integrability of the canonical distribution}
\label{subsect:nonintegrability}

Given
$o\in X_\om$, $\ov x\in B_\om$,
by the property~($1_\K$) there is a unique
$x\in F_{\ov x}=\pi_\om^{-1}(\ov x)$
that is connected with
$o$
by a geodesic segment
$ox$.
The point
$x$
is called the {\em lift} of
$\ov x$
with respect to
$o$,
and we use notation
$x=\lift_o(\ov x)$.
This defines an embedding
$\lift_o:B_\om\to X_\om$
with
$\pi_\om\circ\lift_o=\id$
for every
$o\in X_\om$.
We denote
$D_o=\lift_o(B_\om)$.
The embedding
$\lift_o$
is radially isometric,
$|o\lift_o(\ov x)|=|\pi(o)\ov x|$
for every
$\ov x\in B_\om$.
Though there is no reason for
$\lift_o$
as well as for the projection
$\pi_\om|D_o$
to be isometric, the map
$\lift_o$
is continuous which follows the uniqueness property of
($1_\K$) and compactness of
$X$.

The family of subspaces
$D_o$, $o\in X_\om$,
is called the (canonical) {\em distribution} on
$X_\om$.
We say that the canonical distribution
$\cD=\set{D_o}{$o\in X_\om$}$
on
$X_\om$
is {\em integrable} if for any
$o\in X_\om$
and any
$o'\in D_o$,
the subspaces
$D_o$
and
$D_{o'}$
of
$X_\om$
coincide,
$D_o=D_{o'}$.
For example, if the base 
$B_\om$
is one-dimensional, then the canonical distribution
$\cD$
is obviously integrable.

\begin{pro}\label{pro:nonintegrable_canonical_distribution} Assume
that the canonical distribution on
$X_\om$
is integrable for every
$\om\in X$.
Then
$p=0$, 
i.e. every fiber of every projection
$\pi_\om$
is a point, and the space
$X$
is M\"obius equivalent to
$\wh\R^n$
with
$n=\dim X$.
\end{pro}

\begin{proof} We first show that 
$\pi_\om:D_o\to B_\om$
is an isometry for every
$o\in X_\om$.
Recall that the map
$\pi_\om:D_o\to B_\om$
is radially isometric. For any
$o'$, $x\in D_o$
we have
$x\in D_{o'}$
because
$D_o=D_{o'}$.
Thus
$|\pi_\om(o')\pi_\om(x)|=|o'x|$,
and the map
$\pi_\om:D_o\to B_\om$
is an isometry.
 
It follows that through any two distinct points
$\om$, $o\in X$
there is a uniquely determined subspace
$B\sub X$,
the induced M\"obius structure of which is
the canonical M\"obius structure of the sphere
$S^k$,
where
$k$
is the dimension of any base
$B_\om$, $\om\in X$.
Any such a sphere is called a {\em foliating sphere}.
 
Next we show that
two different foliating spheres
$B$, $B'\sub X$
have at most one point in common. For
$\om\in B\cap B'$
consider a metric of the M\"obius structure such that
$\om$
is the infinitely remote point. Then
$B\sm\{\om\}$, $B'\sm\{\om\}$
are disjoint being the members of the foliation of
$X_\om$
by foliating spheres.

Now, we exploit the same idea as in the proof of
Lemma~\ref{lem:simply_transitivity_shifts}. Assume
$p>0$.
Then there are different
$x$, $x'\in F$,
and let
$B\sub X$
the foliating sphere through
$x$, $x'$.
We have
$\om\not\in B$
since otherwise
$B\sm\om$
is a member of the foliation of
$X_\om$,
and
$B\sm\om$
is covered by Ptolemy lines
$l$
through
$x$, $x'$.
Then by Lemma~\ref{lem:geoconvex_horosphere} 
$B\sm\om$
must lie in
$F$.
However
$F$
contains no Ptolemy line by construction. 

Thus
$B\sub X_\om$
is compact, and its projection
$\ov B=\pi_\om(B)\sub B_\om$
is compact. On the other hand, given
$z\in B$
and a Ptolemy line 
$l\sub X_\om$
through
$z$,
there is a Ptolemy circle
$\si\sub B$
through
$z$
with the tangent line 
$l$.
Indeed, for every
$z'\in B$, $z'\neq z$,
by Corollary~\ref{cor:tangent_unique}
there is a unique circle through
$z$, $z'$
that is tangent to
$l$.
This circle is contained in
$B$
because
$B$
is M\"obius equivalent to
$\wh\R^k$
and it contains with any two points every Ptolemy
circle in
$X$
through these points. It follows that 
$\ov B$
is open in
$B_\om$
and thus
$\ov B=B_\om$.
This contradicts the fact that 
$\ov B$
is compact in
$B_\om$.

Hence
$p=0$, $n=k$,
and
$X$
is M\"obius equivalent to
$\wh\R^n$
with 
$n=\dim X$.
\end{proof}

\begin{cor}\label{cor:nonintegrability} Assume
$p>0$,
that is, fibers of the canonical projections
$\pi_\om$, $\om\in X$,
are nondegenerate. Then the canonical distribution
on
$X_\om$
is non-integrable for every
$\om\in X$.
\end{cor}

\begin{proof} By Proposition~\ref{pro:two_point_homogeneous}
the space
$X$
is 2-point homogeneous. It follows
that if the canonical distribution on
$X_\om$
is integrable for some
$\om$,
then this is true for every
$\om\in X$.
Then 
$p=0$
by Proposition~\ref{pro:nonintegrable_canonical_distribution}.
This contradicts our assumption.
\end{proof}

The next corollary follows immediately from 
Proposition~\ref{pro:nonintegrable_canonical_distribution},
and it completes the proof of Theorem~\ref{thm:basic_ptolemy}.

\begin{cor}\label{cor:onedimensional_base} If the base
$B_\om$
of
$X_\om$
is one-dimensional (this is independent of
$\om\in X$), 
then
$p=0$
and
$X=\wh\R$.
\qed
\end{cor}

\section{Semi$\text{-}\C$-planes}
\label{sect:semi_c_planes}

Starting from this place, we consider in what follows
a compact Ptolemy space 
$X$
that satisfies the assumption of 
Theorem~\ref{thm:basic_ptolemy} and its conclusion for the case
$p=1$,
i.e. when fibers of the canonical projection
$\pi_\om:X_\om\to B_\om$
are one-dimensional for every
$\om\in X$,
and therefore they are homeomorphic to
$\R$.
This also means that the notation
$\K$
is replaced for 
$X$
by
$\C$,
e.g. fibers are $\C$-lines, the properties
($1_\K$), ($2_\K$)
become
($1_\C$), ($2_\C$)
respectively, \semik-planes are called \semi-planes, 
etc. It follows from Corollary~\ref{cor:onedimensional_base} that
$\dim B_\om=k\ge 2$, $\om\in X$.

\subsection{Foliations of \semi-planes}
\label{subsect:foliations_semi_planes}

Let
$l\sub X_\om$
be a Ptolemy line,
$M=M_l$
the respective \semi-plane, see sect.~\ref{subsect:semi_kplanes}.
Any fiber ($\C$-line)
$F\sub X_\om$
of
$\pi_\om$,
that meets
$l$,
is contained in
$M$
by definition of
$M$.
We fix such an
$F\sub M$.
By Lemma~\ref{lem:rfoliation_semik}, every Ptolemy line in
$X_\om$
through any point of
$F$,
which is Busemann parallel to
$l$,
is contained in
$M$
and hits every other $\C$-line 
$F'\sub M$.
Furthermore, any two $\C$-lines
$F$, $F'\sub M$
are equidistant in the sense that the segments of any two
Ptolemy lines
$l$, $l'\sub M$
between
$F$, $F'$
have equal lengths.

\begin{pro}\label{pro:semiplane_homeo} The map
$\psi:l\times F\to M$, $\psi(x,y)=F_x\cap l_y$,
where
$F_x\sub M$
is the
$\C$-line
through
$x\in l$, $l_y\sub M$
is the Ptolemy line through
$y\in F$,
is a homeomorphism, in particular
$M$
is homeomorphic to
$\R^2$.
\end{pro}

\begin{proof} The map
$\psi$
is well defined and bijective by properties of the projection
$\pi_\om:X_\om\to B_\om$,
of the \semi-plane
$M$,
and by ($1_\C$). Since
$X$
is Hausdorff, to show that
$\psi$
is a homeomorphism, it suffices to show
that
$\psi$
is continuous.

Assume that
$F\ni y_i\to y\in F$, $z=l_y\cap F_x$,
$z_i=l_i\cap F_x$,
where
$l_y$, $l_i=l_{y_i}\sub M$
are Ptolemy lines through
$y$, $y_i$
respectively,
$F_x\sub M$
is the
$\C$-line
through
$x\in l$.
Since
$|y_iz_i|=|yz|$
by the equidistant property, the sequence
$|zz_i|$
is bounded, thus
$z_i\in F_x$
subconverges to some
$z'\in F_x$.
Since a pointwise limit of geodesics in any metric
space is a geodesic, we see that
$l_i\to l'$
pointwise, where
$l'\sub X_\om$
is a Ptolemy line through
$y$, $z'$.
By ($1_\C$),
$l'=l_y\sub M$,
hence
$z'=z=\lim z_i$.
This shows that the map
$\psi_x:F\to F_x$, $\psi_x(y)=\psi(x,y)$,
is continuous for every
$x\in l$.

The map
$\psi_y:l\to l_y$ 
given by
$\psi_y(x)=\psi(x,y)$
coincides with the isometry
$\lift_y\circ\pi_\om|l:l\to l_y$.
Now the map
$\psi$
is continuous at every point 
$(x,y)\in l\times F$
because

\begin{align*}
|\psi(x',y')\psi(x,y)|&\le|\psi(x',y')\psi(x,y')|
+|\psi(x,y')\psi(x,y)|\\
 &=|x'x|+|\psi_x(y')\psi_x(y)|\to 0
\end{align*}
as
$(x',y')\to(x,y)$.
\end{proof}

It follows from Proposition~\ref{pro:semiplane_homeo}
that every \semi-plane
$M\sub X_\om$
carries two foliations, one by
$\C$-lines
and the other by Ptolemy lines. Thus fixing an order on a
$\C$-line
$F\sub M$,
we have a well defined order on any other
$\C$-line
$F'\sub M$
compatible with the Ptolemy line foliation of
$M$.
Therefore, given a Ptolemy line
$l\sub M$,
the notion of a half-plane of
$M$
bounded by
$l$
is well defined, and there are two half-planes
bounded by
$l$
whose union is
$M$.
Any homothety
$\phi:M\to M$
with
$\phi(l)=l$
either preserves each of the two half-plane bounded by
$l$
or it permutes them.

\subsection{Flips of \semi-planes and self-duality}
\label{subsect:flips_semi_planes}

By the extension property ($E_2$), every flip
$\phi:l\to l$
of a Ptolemy line
$l\sub X_\om$
extends to an isometry of
$X_\om$
for which we use the same notation
$\phi$.
Since
$\phi$
preserves the line
$l$,
the \semi-plane
$M$
that contains
$l$
is also preserved by
$\phi$
and
$\phi:M\to M$
is an isometry with a fixed point
$o\in l\sub M$.
If
$\phi$
preserves each of the half-planes bounded by
$l$,
then we say that
$\phi:M\to M$
is a flip.

\begin{pro}\label{pro:flip_extension} For every \semi-plane
$M\sub X_\om$
every flip
$\phi:l\to l$
of any Ptolemy line
$l\sub M$
extends to an isometry of
$X_\om$,
which is a flip on
$M$.
\end{pro}

We begin the proof with following

\begin{lem}\label{lem:path_outside} Let
$M\sub X_\om$
be a \semi-plane,
$l\sub M$
a Ptolemy line. Given distinct
$p$, $p'\in l$,
there exists a continuous path
$\si:[0,1]\to X_\om$
between
$p$, $p'$, $\si(0)=p$, $\si(1)=p'$,
such that
$\si((0,1))\sub X_\om\sm M$,
and
$o\si(t)$
is a geodesic segment in
$X_\om$
of length
$|pp'|/2$
for the midpoint
$o\in pp'$
and for all
$t\in[0,1]$.
\end{lem}

\begin{proof} It follows from the property~($1_\C$)
that for each distinct $\C$-lines 
$F$, $F'\sub X_\om$
there is a unique \semi-plane
$M\sub X_\om$
that contains
$F$, $F'$,
and if distinct \semi-planes
$M$, $M'\sub X_\om$
intersect, then their intersection
$M\cap M'$
is a unique common $\C$-line.

Let
$F$, $F'\sub X_\om$
be the $\C$-lines
through
$p$, $p'$
respectively. Since the dimension of the base
$B_\om$
is at least 2, and every \semi-plane in
$X_\om$
projects down to a line in the base, there is a point
$x\in X_\om\sm M$.
Let
$F''$
be the $\C$-line
through
$x$.
We denote with
$M'$, $M''$
the \semi-planes
that contain the pairs
$F$, $F''$
and
$F'$, $F''$
respectively. Then
$F$, $F'\sub M$
bound a strip
$FF'\sub M$
foliated by $\C$-lines.
Similarly,
$FF''\sub M'$, $F'F''\sub M''$
are respective strips foliated by $\C$-lines.

We parameterize the set of the $\C$-lines
of
$FF''\cup F'F''$
by the segment
$[0,1]$
in the natural way,
$t\mapsto F_t$,
so that
$F_0=F$, $F_{1/2}=F''$, $F_1=F'$.
Then the $\C$-line
$F_t$
is disjoint with the \semi-plane
$M$
for all
$0<t<1$
by our construction. By the property~($1_\C$),
for each
$F_t$, $0\le t\le 1$,
there is a unique Ptolemy line
$l_t\sub X_\om$
connecting
$o$
with
$F_t$
(note that
$l_0=l=l_1$).

Now, the point
$\si(t)\in l_t$, $0\le t\le 1$,
is uniquely determined by conditions
$|o\si(t)|=|pp'|/2$,
and
$\si(t)$
lies on the subray of
$l_t$
with the vertex
$o$
that intersects
$F_t$.
In particular,
$\si(0)=p$, $\si(1)=p'$
and
$\si((0,1))\sub X_\om\sm M$.
Furthermore,
$o\si(t)$
is a subsegment of the Ptolemy line
$l_t$
in
$X_\om$.

Note that by Proposition~\ref{pro:semiplane_homeo} the map
$t\mapsto F_t$
is continuous in the sense that if
$t_i\to t$,
then
$F_{t_i} \to F_t$
pointwise. Thus continuity of the map
$\si:[0,1]\to X_\om$
follows from compactness of
$X$
and the fact that
$o\si(t)$
is the unique geodesic segment in
$X_\om$
between
$o$
and
$F_t$
for every
$t\in[0,1]$.
\end{proof}

\begin{proof}[Proof of Proposition~\ref{pro:flip_extension}] Let
$o\in l$
be the fixed point of the flip
$\phi:l\to l$,
$F\sub M$
the $\C$-line
through
$o$.
Since
$\phi:l\to l$
is an isometry, every its extension
$\ov\phi:X_\om\to X_\om$
is an isometry, the \semi-plane
$M$
is invariant under
$\ov\phi$, $\ov\phi(M)=M$,
by definition of
$M$,
and
$\ov\phi(F)=F$
because there is only one $\C$-line through
$o$.
If the restriction
$\ov\phi|F$
is the identity of
$F$
for some extension
$\ov\phi$
of
$\phi$,
then
$\ov\phi:M\to M$
is a flip.

We fix
$p\in l$, $p\neq o$,
and put
$p'=\phi(p)$.
Let
$\si:[0,1]\to X_\om$
be the path between
$p=\si(0)$
and
$p'=\si(1)$
constructed in Lemma~\ref{lem:path_outside}. We denote by
$l_t$
the Ptolemy line in
$X_\om$
that contains the segment
$o\si(t)$,
and by
$\phi_t:l\to l_t$
the isometry with
$\phi_t(o)=o$, $\phi_t(p)=\si(t)$, $t\in[0,1]$.
Then
$\phi_0=\id_l$
and
$\phi_1=\phi$.

Assume that an isometry
$\al_t:X_\om\to X_\om$
fixes the Ptolemy line
$l_t$
pointwise for some
$t\in[0,1]$
and the restriction
$\al_t|F:F\to F$
is a flip. Using the same notation
$\phi_t:X_\om\to X_\om$
for an extension of
$\phi_t:l\to l_t$
that exists by ($E_2$), we note that
$\be_t=\phi_t^{-1}\circ\al_t\circ\phi_t:X_\om\to X_\om$
preserves
$M$,
fixes the Ptolemy line
$l$
pointwise and
$\be_t:F\to F$
is a flip. Then for every extension
$\ov\phi$
of
$\phi$
that is not a flip on
$M$,
the isometry
$\ov\phi\circ\be_t:M\to M$
is a flip that extends
$\phi$.

Thus we assume that for every
$t\in[0,1]$,
every isometry
$\al_t:X_\om\to X_\om$,
that fixes the Ptolemy line
$l_t$
pointwise, restricts to the identity of
$F$, $\al_t|F=\id_F$.
The set
$$A=\set{t\in[0,1]}{$\ov\phi|F=\id_F\ \text{for every extension}\
\ov\phi_t:X_\om\to X_\om\ \text{of}\ \phi_t$}$$
is closed in
$[0,1]$
by continuity and
$0\in A$
by our assumption. We show that
$A$
is open in
$[0,1]$.
If not, then there is a sequence
$t_i\in[0,1]\sm A$
converging to some
$t\in A$,
and for every
$i$
there is an extension
$\ov\phi_i:X_\om\to X_\om$
of
$\phi_{t_i}$
such that
$\ov\phi_i|F:F\to F$
is a flip. The sequence
$\{\ov\phi_i\}$
subconverges to an isometry
$\psi:X_\om\to X_\om$
such that
$\psi|l=\phi_t$
and
$\psi|F:F\to F$
is a flip. This contradicts the condition
$t\in A$.
Thus
$A=[0,1]$
and
$\phi=\phi_1$
extends to the required flip
$\ov\phi:M\to M$.
\end{proof}

Given a Ptolemy circle
$\si\sub X$
and distinct
$\om$, $\om'\in\si$,
recall that the set 
$B_{\si,\om'}^\om$
consists of all
$x\in X_\om$
with
$b^\pm(x)\ge 0$,
where
$b^\pm:X_\om\to\R$
are the opposite Busemann functions of the Ptolemy line
$\si\sm\om\sub X_\om$
with
$b^\pm(\om')=0$,
see sect.~\ref{subsect:duality_dist_busemann}. In a Busemann 
flat Ptolemy space,
$B_{\si,\om'}^\om=H_{\si,\om'}^\om$
is the horosphere of
$\si\sm\om$
through
$\om'$.
Furthermore,  
$D_{\si,\om}^{\om'}$
is the subset in
$X_{\om'}$
which consists of all
$x$
such that
$\om$
is a closest to 
$x$
point in the geodesic line
$\si\sm\om'$
(w.r.t. the metric of
$X_{\om'}$).
By duality, see Lemma~\ref{lem:omega_closest_subset}, we have
$$B_{\si,\om'}^\om\cup\om=D_{\si,\om}^{\om'}\cup\om',$$
or in the case of a flat space by Lemma~\ref{lem:flat_duality}
$$H_{\si,\om'}^\om\cup\om=D_{\si,\om}^{\om'}\cup\om'.$$
It follows that
$$B_{\om'}^\om\cup\om=D_\om^{\om'}\cup\om',$$
where
$B_{\om'}^\om$
is the intersection of
$B_{\si,\om'}^\om$
taken over all the Ptolemy circles 
$\si$
in
$X$
containing both
$\om$, $\om'$, $D_\om^{\om'}$
is the intersection of
$D_{\si,\om}^{\om'}$
taken over all the Ptolemy circles 
$\si\sub X$
containing both
$\om$, $\om'$.
In the case of a flat space this equality takes the form
$$H_{\om'}^\om\cup\om=D_\om^{\om'}\cup\om',$$
where 
$H_{\om'}^\om=F$
is the fiber through
$\om'$
of the canonical projection
$\pi_\om:X_\om\to B_\om$.
In other words, for the fiber ($\C$-line)
$F\sub X_\om$
through
$\om'$,
every point 
$x\in F\sm\om'$
has the property that 
$\om$
is a closest point to
$x$
on the Ptolemy line 
$\si\sm\om'\sub X_{\om'}$
in the space
$X_{\om'}$
for every Ptolemy circle
$\si\sub X$
through
$\om$, $\om'$.

\begin{cor}\label{cor:selfduality} The space 
$X$
is self-dual in the sense that
$H_{\om'}^\om=D_{\om'}^\om$
for each distinct
$\om$, $\om'\in X$.
That is, every point
$x$
of the $\C$-line
$F\sub X_\om$
through
$\om'$
has the property that 
$\om'$
is a closest to
$x$
point on every Ptolemy line in
$X_\om$
through
$x$,
and vice versa.
\end{cor}

\begin{proof} We show that 
$H_{\om'}^\om\sub D_{\om'}^\om$.
We fix
$x\in H_{\om'}^\om$
and let 
$l\sub X_\om$
be a Ptolemy line through
$\om'$.
There is
$z\in l$
closest to
$x$
among all the points of
$l$.
By Proposition~\ref{pro:flip_extension} there is an isometry
$\phi:X_\om\to X_\om$
with 
$\phi(x)=x$
that flips
$l$
at
$\om'$.
Then 
$z'=\phi(z)\in l$
is also closest to
$x$.
By convexity of the distance function from
$x$
along
$l$, 
the segment
$zz'\sub l$
consists of points closest to
$x$.
However,
$\om'\in zz'$.
Hence
$H_{\om'}^\om\sub D_{\om'}^\om$.
By duality,
$D_{\om'}^\om\cup\om=H_\om^{\om'}\cup\om'$.
Thus
$H_{\om'}^\om\cup\om\sub H_\om^{\om'}\cup\om'$.
Interchanging
$\om$
with
$\om'$,
we obtain 
$H_{\om'}^\om\cup\om=H_\om^{\om'}\cup\om'=D_{\om'}^\om\cup\om$,
i.e.
$H_{\om'}^\om=D_{\om'}^\om$,
and
$X$
is self-dual. 
\end{proof}

We put
$\wh F=F\cup\om$
for every fiber
$F$
of
$\pi_\om$.
Then
$\wh F$
is a compact subset of
$X$
called a {\em $\C$-circle}. Since
$F$
is homeomorphic to
$\R$, $\wh F$
is homeomorphic to
$S^1$.
It follows from Lemma~\ref{lem:induced_base_map} 
that the collection of all the $\C$-circles in
$X$
is invariant under any M\"obius automorphism of
$X$.

\begin{lem}\label{lem:C-circles}
For every $\C$-circle
$\wh F\sub X$
and every point
$\om\in\wh F$,
the set
$F=\wh F\sm\om$
is a fiber of the fibration 
$\pi_\om:X_\om\to B_\om$.
\end{lem}

\begin{proof} We have
$\wh F=F'\cup\om'$
for some 
$\om'\in X$,
where
$F'$
is a fiber of the fibration
$\pi_{\om'}:X_{\om'}\to B_{\om'}$.
We suppose that
$\om'\neq\om$
since otherwise the assertion is trivial. Then
$\om\in F'$
and
$F'=H_\om^{\om'}$.
By self-duality, see Corollary~\ref{cor:selfduality},
$H_\om^{\om'}=D_\om^{\om'}$.
Using duality we obtain 
$\wh F=F'\cup\om'=H_{\om'}^\om\cup\om$.
Thus
$F=\wh F\sm\om=H_{\om'}^\om$
is the fiber of
$\pi_\om$
through 
$\om'$.
\end{proof}

\begin{cor}\label{cor:3c_4c} In addition to ($1_\C$) and ($2_\C$),
see Theorem~\ref{thm:basic_ptolemy}, the space 
$X$
has the following basic properties
\begin{itemize}
 \item[($3_\C$)] through any two distinct points in
$X$
there is a unique $\C$-circle;
 \item[($4_\C$)] any $\C$-circle and any  Ptolemy circle in
$X$
have at most two points in common.
\end{itemize}
\end{cor}

\begin{proof} ($3_\C$). Given distinct
$\om$, $\om'\in X$,
let
$F\sub X_\om$
be the fiber of
$\pi_\om:X_\om\to B_\om$
through
$\om'$.
Then
$\wh F=F\cup\om$
is a $\C$-circle through
$\om$, $\om'$.
By Lemma~\ref{lem:C-circles},
$F'=\wh F\sm\om'$
is the fiber of
$\pi_{\om'}:X_{\om'}\to B_{\om'}$
through
$\om$.
Since the fibers of
$\pi_\om$, $\pi_{\om'}$
through given points are uniquely determined, it follows that
$\wh F$
is a unique $\C$-circle through
$\om$, $\om'$.

($4_\C$). Let
$\om\in\wh F\cap\si$
be a common point of a $\C$-circle
$\wh F\sub X$
and a Ptolemy circle
$\si\sub X$.
Then 
$l=\si\sm\om\sub X_\om$
is a Ptolemy line,
and by Lemma~\ref{lem:C-circles},
$F=\wh F\sm\om$
is a fiber of 
$\pi_\om:X_\om\to B_\om$.
It follows from the definition of fibers of
$\pi_\om$
and Lemma~\ref{lem:geoconvex_horosphere} that
$F$
and
$l$
have at most one point in common. Hence, the claim. 
\end{proof}

\subsection{Both foliations of \semi-planes are equidistant}
\label{subsect:both_foliation_equidistant}

We already know that the foliation of a \semi-plane
$M\sub X_\om$
by $\C$-lines is equidistant. Now, Ptolemy lines
$l$, $l'$
in 
$M$
are called {\em equidistant}
if the distance between
$x=l\cap F$
and
$x'=l'\cap F$
is independent of a
$\C$-line
$F\sub M$.

\begin{lem}\label{lem:equidist_lines} All the Ptolemy lines in every
\semi-plane
$M\sub X$
are pairwise equidistant.
\end{lem}

\begin{proof} Let
$l$, $l'\sub M$
be Ptolemy lines,
$c$, $c':\R\to M$
unit speed parameterizations of
$l$, $l'$
respectively with 
$c(0)$, $c'(0)\in F$,
where 
$F\sub M$
is a $\C$-line. We assume using equidistant property of
$\C$-lines that 
$c$, $c'$
are compatible in the sense that the points
$c(t)$, $c'(t)$
lie in one and the same $\C$-line 
$F_t\sub M$
for every
$t\in\R$.
We put
$\mu(t)=|c(t)c'(t)|$.
By self-duality,
$c'(t)$
is a closest to
$c(t)$
point on
$l'$,
and vice versa,
$c(t)$
is a closest to
$c'(t)$
point on
$l$
for every
$t\in\R$.
Thus
$|c(t)c'(t')|$, $|c'(t)c(t')|\ge\max\{\mu(t),\mu(t')\}$
for each
$t$, $t'\in\R$.
Applying the Ptolemy inequality to the quadruple
$(c(t),c(t'),c'(t'),c'(t))$,
we obtain 
$$\max\{\mu(t),\mu(t')\}^2\le|c(t)c'(t')||c'(t)c(t')|
  \le\mu(t)\mu(t')+(t-t')^2.$$
We show that
$\mu(a)=\mu(0)$
for every
$a\in\R$.
Assume W.L.G. that
$a>0$
and put
$m=1/\min_{0\le s\le a}\mu(s)$.
Then
$|\mu(t)-\mu(t')|\le m(t-t')^2$
for each
$0\le t,t'\le a$.
Now
$$\mu(a)-\mu(0)=\mu(s)-\mu(0)+\mu(2s)-\mu(s)+\dots+\mu(a)-\mu((k-1)s),$$
where
$s=a/k$
for 
$k\in\N$.
It follows
$|\mu(a)-\mu(0)|\le mks^2=ma^2/k\to 0$
as
$k\to\infty$.
Hence,
$\mu(a)=\mu(0)$.
\end{proof}

\begin{lem}\label{lem:rectangle_side_length} For any
$x,y\in F$, $x',y'\in F'$,
where
$F$, $F'$
are
$\C$-lines
in a \semi-plane
$M$,
such that
$x,x'\in l$, $y,y'\in l'$,
where
$l$, $l'\sub M$
are
Ptolemy lines,
we have
$|xy|=|x'y'|=:a$, $|xx'|=|yy'|=:b$,
$|xy'|=|x'y|=:c$,
and
$c^2\le a^2+b^2$.
\end{lem}

\begin{proof} We have
$|xx'|=|yy'|$
because $\C$-lines in
$M$
are equidistant, and
$|xy|=|x'y'|$
because Ptolemy lines in
$M$
are also equidistant. By Proposition~\ref{pro:flip_extension}, 
there exists a flip
$\phi:M\to M$
that permutes
$x$, $y$
and
$x'$, $y'$
respectively,
$\phi(x)=x'$, $\phi(y)=y'$.
Thus
$|xy'|=|x'y|$.
Applying the Ptolemy inequality, we obtain
$c^2\le a^2+b^2$.
\end{proof}

For any two fibers
$F_b$, $F_{b'}$
of
$\pi_\om$,
by ($1_\C$) and Lemma~\ref{lem:equidist_lines}, we have
the canonically determined isometry
$\mu_{bb'}:F_b\to F_{b'}$.

\begin{lem}\label{lem:continuity_project_fiber} The isometries
$\mu_{bb'}:F_b\to F_{b'}$
depend continuously of
$b$, $b'\in B_\om$,
that is, for
$b_i\to b'$
and for any
$x\in F_b$,
we have
$\mu_{bb_i}(x)\to\mu_{bb'}(x)$.
\end{lem}

\begin{proof} If a sequence of geodesic segments in a
metric space pointwise converges, then the limit is
also a geodesic segment. Together with
uniqueness of Ptolemy lines in
$X_\om$
and compactness of
$X$,
this implies the claim.
\end{proof}

Now, we fix an order of
$F_b$
and define the order of
$F_{b'}$
via the isometry
$\mu_{bb'}$.
This gives a simultaneously
determined order
$O$
on all the fibers of
$\pi$.

\begin{lem}\label{lem:fiber_orient} The order
$O$
is well defined and independent of the choice of
$b\in B_\om$.
\end{lem}

\begin{proof} The base
$B_\om$
is contractible. Using Lemma~\ref{lem:continuity_project_fiber},
we see that the order of
$F_{b''}$
induced by
$\mu_{bb''}$
coincides with the order induced by
$\mu_{b'b''}\circ\mu_{bb'}$
for each
$b'$, $b''\in B_\om$.
Hence, the claim.
\end{proof}

Let
$G_\om$
be the group of all M\"obius automorphisms of
$X$
fixing
$\om$
and preserving the order
$O$.
Every
$\phi\in G_\om$
acts on
$X_\om$
as a homothety. By Corollary~\ref{cor:pi_submetry}, we have
a homomorphism
$\pi_\ast$
of
$G_\om$
into the group of homotheties of
$B_\om$.
We denote
$H=\pi_\ast(G_\om)$.

\section{Isometry group of $X_\om$}
Fix
$\om\in X$
and a metric from the M\"obius structure of
$X$
such that
$\om$
is the infinitely remote point. Recall that
$N_\om\sub\aut X$
is the group of shifts of
$X_\om$,
see sect.~\ref{subsect:group_shifts}.

\begin{lem}\label{lem:shift_preserve_order} Every shift
$\eta\in N_\om$
preserves the order
$O$
on $\C$-lines in
$X_\om$.
\end{lem}

\begin{proof} Let
$\phi:X\to X$
be a space inversion that permutes
$\om$, $\om'\in X$.
By Corollary~\ref{cor:3c_4c}($3_\C$), 
$\phi$
preserves the (unique) $\C$-circle
$\wh F$
through
$\om$, $\om'$, $\phi(\wh F)=\wh F$.
Since
$\phi$
has no fixed point,
$\phi$
also preserves orientations of
$\wh F$.
Let
$\ga:X_\om\to X_\om$
be a pure homothety with 
$\ga(\om')=\om'$.
Recall that
$\ga=\phi'\circ\phi$
for some space inversions
$\phi$, $\phi'$
which both permute
$\om$, $\om'$.
It follows that 
$\ga(\wh F)=\wh F$
and
$\ga$ 
preserves the order
$O$
of the $\C$-line
$F=\wh F\sm\om$
through
$\om'$. 

By Lemma~\ref{lem:pure_homothety},
$\ga$
preserves every foliation of
$X_\om$
by Busemann parallel Ptolemy lines, thus
$\ga(F')$
is a $\C$-line for every $\C$-line
$F'\sub X_\om$.
Now, by continuity
$\ga$
preserves the order
$O$
on $\C$-lines in
$X_\om$.
Then it follows from the definition of shifts that every shift
$\eta:X_\om\to X_\om$
preserves the order
$O$.
\end{proof}

\begin{lem}\label{lem:horizontal_shift} Given a Ptolemy line
$l\sub X_\om$, $x$, $x'\in l$,
the shift
$\eta:X_\om\to X_\om$
with
$\eta(x)=x'$
preserves every Ptolemy line
$l'\sub M$, $\eta(l')=l'$,
where
$M\sub X_\om$
is the \semi-plane containing
$l$.
In particular, every isometry
$\mu_{bb'}:F_b\to F_{b'}$
between $\C$-lines
$F_b$, $F_{b'}\sub X_\om$
extends to shift
$\eta:X_\om\to X_\om$.
\end{lem}

\begin{proof} Let
$F$, $F'\sub X_\om$
be the $\C$-lines
through
$x$, $x'$
respectively. Then
$F$, $F'\sub M$.
We put
$y=l'\cap F$, $y'=l'\cap F'$
and note that
$|xy|=|x'y'|$
by Lemma~\ref{lem:equidist_lines}.

We have
$\eta(M)=M$,
thus
$\eta(l')\sub M$.
By Lemma~\ref{lem:shift_preserve_order},
$\eta$
preserves the order
$O$,
thus
$\eta(y)=y'$
and
$y'\in\eta(l')$.
Since the Ptolemy line in
$M$
through
$y'$
is unique, we obtain
$\eta(l')=l'$. 
\end{proof}

\subsection{Isometries of $\C$-lines}
\label{subsect:iso_clines}

A shift
$\eta:X_\om\to X_\om$
is said to be {\em vertical} if it induces the identity
of the base 
$B_\om$, $\pi_\ast(\eta)=\id$.

\begin{pro}\label{pro:const_displacement_cline} The displacement
function of every vertical shift
$\eta:X_\om\to X_\om$
is constant.
\end{pro}

Let
$T\sub B_\om$
be a {\em pointed} oriented triangle,
i.e., we assume that an orientation and a vertex
$o$
of
$T$
are fixed. Then
$T$
determines a map
$\tau_T:F\to F$,
where
$F\sub X_\om$
is the fiber of
$\pi_\om$
over
$o$, $F=\pi_\om^{-1}(o)$.
Given
$x\in F$,
we lift to
$X_\om$
the sides of
$T$
in the cyclic order according to the orientation and starting
with
$o$
which is initially lifted to
$x$.
Then
$\tau_T(x)\in F$
is the resulting lift of the triangle sides.

\begin{lem}\label{lem:triangle_lift} The map
$\tau_T:F\to F$
is an isometry that preserves the order and has the constant
displacement function.
\end{lem}

\begin{proof} The map
$\tau_T$
is obtained as a composition of three
$\C$-line
isometries of type
$\mu_{bb'}$,
see sect.~\ref{subsect:both_foliation_equidistant}. Any isometry
$\mu_{bb'}$
preserves the order
$O$,
see Lemma~\ref{lem:fiber_orient}. Thus
$\tau_T:F\to F$
is an isometry preserving the order.

Given
$x$, $x'\in F$,
we let
$y=\tau_T(x)$, $y'=\tau_T(x')$.
There is a shift
$\eta:X_\om\to X_\om$
with 
$\eta(x)=x'$.
By Lemma~\ref{lem:homothety_vert_shift}, 
$\eta$
is vertical, and by Lemma~\ref{lem:shift_preserve_order} it
preserves the order
$O$.
Thus
$\eta$
preserves
$F$,
every \semi-plane
$M\sub X_\om$
and maps isometrically every Ptolemy line
$l\sub M$
to an equidistant Ptolemy line
$l'\sub M$.
Then it follows from definition of
$\tau_T$
that
$\eta(y)=y'$.
Therefore
$|x'y'|=|xy|$,
and the displacement function of
$\tau_T$
is constant.
\end{proof}

\begin{lem}\label{lem:unclosed_triangle} For every
$b\in B_\om$
there is a pointed oriented triangle
$T\sub B_\om$,
for which the map
$\tau_T:F\to F$, $F=\pi_\om^{-1}(b)$,
is not identical.
\end{lem}

\begin{proof} By Corollary~\ref{cor:nonintegrability}, the
canonical distribution 
$\cD$
on
$X_\om$
is not integrable. This means that there are
$o\in X_\om$, $o'$, $x\in D_o$
such that
$x\not\in D_{o'}$.
Then the points
$b'=\pi_\om(o)$, $\pi_\om(o')$, $\pi_\om(x)\in B_\om$
are vertices of a pointed oriented triangle
$T'$
for which the map
$\tau_{T'}:F'\to F'$
is not identical, where
$F'=\pi_\om^{-1}(b')$.
There is a shift 
$\eta:X_\om\to X_\om$
with
$\eta(F')=F$.
The shift
$\eta$
induces a shift
$\pi_\ast(\eta):B_\om\to B_\om$
of the base. Then
$T=\pi_\ast(\eta)(T')$
is a required triangle in
$B_\om$.
\end{proof}

\begin{proof}[Proof of Proposition~\ref{pro:const_displacement_cline}]
We fix a $\C$-line
$F\sub X_\om$
and first show that the displacement function of
$\eta$
is constant along 
$F$.

Let
$\la(T)$
be the perimeter of a pointed oriented triangle
$T\sub B_\om$
with the base point
$b=\pi_\om(F)$.
By triangle inequality, we have
$|x\tau_T(x)|\le\la(T)$
for every
$x\in F$.
Note that
$\tau_T^k=\tau_{T^k}$
for every
$k\in\Z$,
where
$T^k$
means that the triangle
$T$
is passed around
$|k|$
times in the direction prescribed by the sign of
$k$.
Thus the displacement function of
$\tau_T^k$
is also constant by Lemma~\ref{lem:triangle_lift}.

Using Lemma~\ref{lem:unclosed_triangle} and applying
appropriate homotheties of
$X_\om$,
we can find a pointed oriented triangle
$T$
with the base point
$b$,
arbitrarily small perimeter and non-identical isometry
$\tau_T:F\to F$.
By construction,
$\tau_T$
is a composition of three isometries of type
$\mu_{bb'}:F_b\to F_{b'}$.
Then by Lemma~\ref{lem:horizontal_shift},
$\tau_T$
extends to an isometry from the group
$N_\om$,
which is a vertical shift. We use the same notation
$\tau_T$
for this extension, 
$\tau_T:X_\om\to X_\om$.

Given
$x\in F$,
there is a sequence
$T_i$
of pointed oriented triangles with the base point
$b$
and
$\la(T_i)\to 0$,
and a sequence
$k_i\in\Z$
such that
$\tau_i(x)=\tau_{T_i}^{k_i}(x)\to\eta(x)$.
Then
$\tau_i:X_\om\to X_\om$
subconverges to a vertical shift
$\tau:X_\om\to X_\om$
with
$\tau(x)=\eta(x)$
that preserves the order and has the constant
displacement along
$F$.
Using Lemma~\ref{lem:simply_transitivity_shifts} we conclude that
$\eta=\tau$ 
has constant displacement along
$F$.

Now, the displacement of
$\eta$
is constant because
$\eta$
preserves every $\C$-line in
$X_\om$
and Ptolemy lines in \semi-planes are equidistant by
Lemma~\ref{lem:equidist_lines}.
\end{proof}

\subsection{Existence of an unclosed parallelogram}
\label{subsect:unclosed_par}

Let
$P\sub B_\om$
be a {\em pointed} oriented parallelogram,
i.e., we assume that an orientation and a vertex
$o$
of
$P$
are fixed. Similarly to the map
$\tau_T$
discussed in sect.~\ref{subsect:iso_clines}, we have a map
$\tau_P:F\to F$,
where
$F\sub X_\om$
is the fiber of
$\pi$
over
$o$, $F=\pi^{-1}(o)$.
Namely, given
$x\in F$,
we lift the sides of
$P$
to
$X_\om$
in the cyclic order according the orientation and starting
with
$o$
which is initially lifted to
$x$.
Then
$\tau_P(x)\in F$
is the resulting lift of the parallelogram sides.
We say that an pointed oriented parallelogram
$P\sub B_\om$
is {\em closed}, if the map
$\tau_P:F\to F$
is identical,
$\tau_P=\id_F$.
This property depends of the choice neither of
the base vertex nor of the orientation of
$P$.
Thus we can speak about closed or unclosed parallelograms
as well as closed or unclosed triangles in
$B_\om$.
By Lemma~\ref{lem:unclosed_triangle}, there exists
an unclosed triangle.

We need the following fact.

\begin{lem}\label{lem:transitive_action} Let
$G$
be a closed subgroup of the orthogonal group
$\bO(n)$, $n\ge 2$,
which acts transitively on the sphere
$S^{n-1}\sub\R^n$.
Then for any 2-dimensional subspace
$L\sub\R^n$
there is
$g\in G$
such that
$g(v)=-v$
for every
$v\in L$.
\end{lem}

\begin{proof} A list of all compact connected Lie groups acting
transitively and effectively on the sphere
$S^{n-1}$
is obtained in \cite{MS}, \cite{Bor}. It consists of
$SO(n)$, $U(m)$, $SU(m)$
with
$n=2m$, $Sp(m)Sp(1)$, $Sp(m)U(1)$, $Sp(m)$
with
$n=4m$, $G_2$
with
$n=7$, $\text{Spin}(7)$
with
$n=8$, $\text{Spin}(9)$
with
$n=16$,
see \cite[sect.~7.13]{Bes}. The required property is obvious for
$SO(n)$.
The groups
$U(m)$, $SU(m)$
with
$n=2m$
include the element
$-\id$.
This is also true for
$Sp(m)$
with
$n=4m$
and
$\text{Spin}(7)$, $\text{Spin}(9)$.

The group
$G_2$
is the automorphism group of the octonions
$\bO$
which acts on imaginary octonions
$\im\bO=\R^7$
in a way that any {\em basic triple}
$e_1$, $e_2$, $e_3\in\im\bO$
is moved to any other basic triple by uniquely
determined automorphism
$g\in G_2$,
see \cite{Ba}. Here
$e_1$
can be chosen arbitrarily with
$e_1^2=-1$,
$e_2$
with
$e_2^2=-1$
anticommutes with
$e_1$, $e_1e_2=-e_2e_1$,
and
$e_3$
with
$e_3^2=-1$
anticommutes with
$e_1$, $e_2$
and
$e_1e_2$.
We can take basic triples
$(e_1,e_2,e_3)$
and
$(e_1',e_2',e_3')$
so that
$e_1,e_2\in L\cap S^6\sub\im\bO$,
$e_1'=-e_1$, $e_2'=-e_2$, $e_3'=e_3$
and put
$g(e_1)=e_1'$, $g(e_2)=e_2'$, $g(e_3)=e_3'$.
This defines a required
$g\in G_2$
with
$g|L=-\id_L$.
\end{proof}

\begin{lem}\label{lem:unclosed_parall} There exists a
unclosed parallelogram in
$B_\om$.
\end{lem}

\begin{proof} 
Using property~($\rm{E}_2$) and argue as in the proof of 
Proposition~\ref{pro:flip_extension}, one shows that 
the stabilizer of any
$o\in X_\om$
in the isometry group of
$X_\om$
preserving the order acts transitively on the set
of the directed Ptolemy lines through
$o$.
Let
$T\sub B_\om$
be an unclosed triangle. By Lemma~\ref{lem:transitive_action},
there is an isometry
$\ov g:X_\om\to X_\om$
preserving the order such that
$g=\pi_\ast(\ov g):B_\om\to B_\om$
is the central symmetry of the plane
$L\sub B_\om$
containing
$T$
with respect to the midpoint
$\ov o$
of some side of
$T$.

Let
$b$
be a common vertex of the triangles
$T$, $g(T)$
and the parallelogram
$P=T\cup g(T)$, $F=F_b\sub X_\om$
the fiber of
$\pi_\om$
over
$b$.
The isometries
$\tau_T$
and
$\tau_{g(T)}$
of
$F$
coincide,
$\tau_{g(T)}=\tau_T$,
because the isometry
$\ov g:X_\om\to X_\om$
preserves the order. Then one easily sees that
the isometry
$\tau_P:F\to F$
satisfies
$\tau_P=\tau_T^2\neq\id_F$.
Therefore, the parallelogram
$P$
is unclosed.
\end{proof}

\subsection{The maximal unipotent subgroup}
\label{subsect:max_unipotent}

Recall that the group
$N_\om$
of the shifts
$X_\om\to X_\om$
acts on
$X_\om$
simply transitively (Lemma~\ref{lem:simply_transitivity_shifts})
and that the subgroup
$Z_\om\sub N_\om$
of vertical shifts acts simply transitively of every
$\C$-line 
$F\sub X_\om$
(Proposition~\ref{pro:fiber_nilpotent_lie_group}).
In context of rank one symmetric spaces 
$Y=\K\hyp^n$
of noncompact type the group 
$N_\om$
is called a {\em maximal unipotent subgroup} of the isometry group 
of the space
$Y$.

We have
$[N_\om,N_\om]\sub Z_\om$
because any two shifts of the base
$B_\om$
commute.

\begin{lem}\label{lem:z_in_center} The group
$Z_\om$
lies in the center of
$N_\om$.
\end{lem}

\begin{proof} By Proposition~\ref{pro:const_displacement_cline}, 
the displacement function
$\de_\zeta$
of
$\zeta$
is constant for every
$\zeta\in Z_\om$.
Thus for every
$\al\in N_\om$, $x\in X_\om$,
we have
$\de_\zeta(x)=\de_\zeta(\al(x))$.
The points
$x$, $\al(x)$, $\zeta\al(x)$, $\al\zeta(x)$
lie in one and the same \semi-plane, thus
$\zeta\al(x)=\al\zeta(x)$.
It follows that
$[\zeta,\al]=\id_{X_\om}$,
i.e.,
$\zeta$
lies in the center of
$N_\om$.
\end{proof}

\begin{lem}\label{lem:commute_shift_prescribed} Given
$a>0$
and a Ptolemy line
$l\sub X_\om$,
there are
$\al$, $\be\in N_\om$
such that
$\al(l)=l$
and the displacement of
$\ga=[\al,\be]$
equals
$a$.
\end{lem}

\begin{proof} By Lemma~\ref{lem:unclosed_parall},
there is an unclosed parallelogram
$P\sub B_\om$.
Applying an appropriate homothety
$\phi:X_\om\to X_\om$
if necessary, we can assume that the Ptolemy line
$l\sub X_\om$
projects down to the line
$\pi_\om(l)\sub B_\om$
that contains a side
$b''b\sub P$,
and the displacement of
$\tau_P:F\to F$
is
$a$, $|x\tau_P(x)|=a$
for every
$x\in F$,
where
$F\sub X_\om$
is the fiber of
$\pi_\om$
over a vertex of
$P$,
cp. Lemma~\ref{lem:triangle_lift}.

Let
$b'\in P$
be the vertex adjacent to
$b''$
and opposite to
$b$.
We consider the Ptolemy line
$l'\sub X_\om$
through
$o=l\cap F_{b''}$
that project down to the line
$\pi_\om(l')\sub B_\om$
containing the side
$b''b'$
of
$P$.
There are shifts
$\al$, $\be\in N_\om$,
which leave invariant the Ptolemy lines
$l$, $l'$
respectively, such that
$\pi_\ast(\al)(b'')=b$, $\pi_\ast(\be)(b'')=b'$.

Since the parallelogram
$P$
is unclosed, we have
$\al\be\neq\be\al$,
and
$\de_\ga(o)=a$,
where
$\de_\ga:X_\om\to\R$
is the displacement
of
$\ga=[\al,\be]$.
By Proposition~\ref{pro:const_displacement_cline}, the displacement
$\de_\ga$
is constant, thus
$\de_\ga(x)=a$
for every
$x\in X_\om$.
\end{proof}

\begin{pro}\label{pro:nil_transitive} The group
$N_\om$
is nilpotent and the group
$Z_\om$
is the center of
$N_\om$.
Moreover
$Z_\om=[N_\om,N_\om]$.
\end{pro}

\begin{proof}  Assume
$\al'\in N_\om$
commutes with every
$\be\in N_\om$.
We show that
$\al'\in Z_\om$.
Together with Lemma~\ref{lem:z_in_center} this
implies that
$Z_\om$
is the center of
$N_\om$.

Composing
$\al'$
with an appropriate
$\nu\in Z_\om$
if necessary, we can assume that
$\al'(l)=l$
for some Ptolemy line
$l\sub X_\om$.
It suffices to show that
$\al'(o)=o$
for some
$o\in l$.
By Lemma~\ref{lem:commute_shift_prescribed}, there are
$\al$, $\be\in N_\om$
such that
$\al(l)=l$
and the displacement
$\de_\ga=a$
for a given
$a>0$,
where
$\ga=[\al,\be]$.
Suppose that
$|o\al'(o)|\neq 0$.
Conjugating
$\al'$
by an appropriate pure homothety of
$X_\om$
preserving 
$o$,
we can assume that
$|o\al(o)|=|o\al'(o)|$.
Replacing
$\al'$
with
$(\al')^{-1}$
if necessary, we can also assume 
$\al'(o)=\al(o)$.
Then
$\al'=\al$.
But this contradicts the assumption that
$\al'$
commutes with every
$\be\in N_\om$.
Hence
$\al'(o)=o$.

Let
$a>0$
be the displacement of some
$\zeta\in Z_\om$, $\zeta\neq\id$.
Again by Lemma~\ref{lem:commute_shift_prescribed},
there are
$\al$, $\be\in N_\om$
such that the displacement of
$\ga=[\al,\be]$
equals
$a$.
Then
$\zeta$
coincides with
$\ga$
or
$\ga^{-1}$.
Thus
$Z_\om=[N_\om,N_\om]$.
\end{proof}

\section{Area law of lifting and metrics of
$\C$-lines}

\subsection{Lifting of polygons}
\label{subsect:lift_polygons}

Let
$P\sub B_\om$
be an oriented parallelogram. Fixing a vertex
$b\in P$,
we obtain a preserving the order isometry
$\tau_B:F\to F$,
where
$F\sub X_\om$
is the fiber of
$\pi_\om:X_\om\to B$
over
$b$,
see sect.~\ref{subsect:unclosed_par}.
It is convenient to associate with
$P$
an extension of
$\tau_P$
to
$X_\om$
which is defined as
$\tau_P(x)=\nu(x)$
for every
$x\in X_\om$,
where 
$\nu:X_\om\to X_\om$
is a vertical shift, 
$\pi_\ast(\nu)=\id_{B_\om}$,
while restricted to
$F$
coincides with
$\tau_P$.
The isometry
$\nu$
exists by Lemma~\ref{lem:horizontal_shift}, and it
is unique by Lemma~\ref{lem:simply_transitivity_shifts}. 
We use the same notation for the extension
$\tau_P:X_\om\to X_\om$
and call it a {\em lifting isometry}.

Furthermore, for every isometry
$\phi\in G_\om$
(recall that such an isometry preserves
$\om$
and the order
$O$)
we have according to Proposition~\ref{pro:const_displacement_cline},
$\tau_{P'}=\tau_P$,
where
$P'=\pi_\ast(\phi)(P)$.
In particular, the map
$\tau_P$
is not changed if we replace the parallelogram
$P$
by any its shifted copy
$P'\sub B$.

The vertical shift
$\tau_P:X_\om\to X_\om$
lies in the group
$Z_\om$,
see sect.~\ref{subsect:max_unipotent}, for every parallelogram
$P\sub B_\om$.
Since
$Z_\om$
is commutative, we have
$\tau_P\circ\tau_{P'}=\tau_{P'}\circ\tau_P$
for any parallelograms and even for any closed
oriented polygons
$P$, $P'\sub B_\om$.

Let
$Q\sub B_\om$
be a closed, oriented polygon. Adding a segment
$qq'\sub B_\om$
between points
$q$, $q'\in Q$
we obtain closed, oriented polygons
$P$, $P'$
such that
$Q\cup qq'=P\cup P'$,
the orientations of
$P$, $P'$
coincide with that of
$Q$
along
$Q$,
and the segment
$qq'=P\cap P'$
receives from
$P$, $P'$
opposite orientations. In this case we use notation
$Q=P\cup P'$.

\begin{lem}\label{lem:adding_lifts} In the notation above we have
$\tau_Q=\tau_{P'}\circ\tau_P$.
\end{lem}

\begin{proof} We fix
$q\in Q\cap P\cap P'$
as the base point. Moving from
$q$
along
$Q$
in the direction prescribed by the orientation of
$Q$,
we also move along one of
$P$, $P'$
according to the induced orientation. We assume W.L.G.
that this is the polygon
$P$.
In that way, we first lift
$P$
to
$X_\om$
starting with some point
$o\in F$,
where
$F$
is the fiber of the projection
$\pi_\om:X_\om\to B_\om$
over
$q$,
such that the side
$q'q\sub P$
is the last one while lifting
$P$.
Now, we lift
$P'$
to
$X_\om$
starting with
$o'=\tau_P(o)\in F$
moving first along the side
$qq'\sub P'$.
Then clearly the resulting lift of
$P'$
gives
$\tau_Q(o)=\tau_{P'}(o')\in F$.
Thus
$\tau_Q=\tau_{P'}\circ\tau_P$.
\end{proof}

\subsection{Area law of lifting}
\label{subsect:area_law_lift}

Let
$P_0\sub B_\om$
be a parallelogram. Applying if necessary
a homothety from the group
$H=\pi_\ast(G_\om)$,
we assume W.L.G. that
$\area P_0=1$.
Furthermore, cutting
$P_0$
by a line in the plane of
$P_0$
and gluing back the obtained pieces shifted
appropriately, one easily transforms
$P_0$
to a rectangle
$P_1$,
which therefore satisfies
$\tau_{P_1}=\tau_{P_0}$
by Lemma~\ref{lem:adding_lifts}. 

Let
$L\sub B_\om$
be the 2-dimensional subspace containing
$P_1$.
We fix two mutually orthogonal directions in
$L$
such that the sides of
$P_1$
are parallel to them,
and call one of them the {\em horizontal} direction and the other one
{\em vertical} direction.

We denote by
$\de(P)$
the displacement of any parallelogram
$P\sub B_\om$, $\de(P)=|x\tau_P(x)|$
for every
$x\in X_\om$,
see Proposition~\ref{pro:const_displacement_cline}, and put
$c_0:=\de(P_0)=\de(P_1)\ge 0$.

We denote by
$\cP_1$
the class of all the rectangles in
$L$
of area 1 with horizontal and vertical sides such that
$\de(P)=c_0$
for every
$P\in\cP_1$.
We write this equality as
\begin{equation}\label{eq:area_law}
\de(P)^2=c_0^2\cdot\area P
\end{equation}
and call
$c_0$
the {\em lifting constant} of the class
$\cP_1$.
Equality~(\ref{eq:area_law}) is called the {\em area law
of lifting}. Note that
$P_1\in\cP_1$.

\begin{lem}\label{lem:operations} The class
$\cP_1$
is closed under the following operations with rectangles:
\begin{itemize}
\item[(a)] a shift in
$L$;
\item[(b)] cutting by finitely many parallel horizontal or
vertical lines and gluing back the
shifted pieces;
\item[(c)] taking the limit of a convergent sequence of rectangles.
\end{itemize}
\end{lem}

\begin{proof} Operation (a) preserves the class
$\cP_1$
because every shift
$\ga:B_\om\to B_\om$
is of the form
$\ga=\pi_\ast(\zeta)$,
where
$\zeta:X_\om\to X_\om$
is a shift, and thus
$\tau_P=\tau_{\ga(P)}$
for every parallelogram
$P\sub B_\om$.
Using Lemma~\ref{lem:adding_lifts}, we see that
operation~(b) preserves the class
$\cP_1$.
Operation~(c) preserves the class
$\cP_1$
because the map
$\tau_P:X_\om\to X_\om$
depends continuously on
$P$.
\end{proof}

\begin{lem}\label{lem:unit_square} The class
$\cP_1$
includes a unit square
$Q_0\sub L$.
\end{lem}

\begin{proof} Given a rectangle
$P\sub L$
with horizontal and vertical sides, and integer
$m$, $n\ge 1$,
we construct a new rectangle
$P'\sub L$
using operation~(b) as follows. First, we subdivide
$P$
into
$m$
pairwise congruent rectangles cutting it by
$(m-1)$
horizontal lines and gluing back the shifted pieces
into a horizontal row. Second, we subdivide the obtained
rectangle
$\wt P$
into
$n$
pairwise congruent rectangles cutting
$\wt P$
by
$(n-1)$
vertical lines and gluing back the shifted pieces into
a vertical column. This gives
the resulting rectangle
$P'=:\Psi_{m,n}(P)$.
Note that if
$P\in\cP_1$,
then
$\Psi_{m,n}(P)\in\cP_1$
for each integer
$m$, $n\ge 1$
by Lemma~\ref{lem:operations}. Furthermore, if
$a=a(P)$
is the length of the horizontal sides of
$P$
and
$b=b(P)$
is the length of the vertical sides of
$P$,
then
$$a'=a(P')=\frac{m}{n}a,\ b'=b(P)=\frac{n}{m}b.$$

Assume W.L.G. that
$a<b$.
Then there are integer
$m>n\ge 1$
such that
$$\frac{1}{2}\left(1+\frac{b}{a}\right)\le\frac{m^2}{n^2}
<\frac{b}{a}.$$
Thus
$$1<\frac{b'}{a'}\le\la\frac{b}{a}<\frac{b}{a},$$
where
$\la^{-1}=\frac{1}{2}(1+\frac{b}{a})>1$.
This generates a sequence of rectangles
$P_i$, $P_{i+1}=\Psi_{m_i,n_i}(P_i)$.
By the choice of
$m_i$, $n_i$
this sequence cannot have accumulation points different from
$Q_0$,
thus
$P_i\to Q_0$.
Therefore
$Q_0\in\cP_1$
by operation~(c).
\end{proof}

\begin{cor}\label{cor:each_par_p} Every rectangle
$P\sub L$
of area 1 with horizontal and vertical sides is in the class
$\cP_1$, $P\in\cP_1$.
\end{cor}

\begin{proof} According to Lemma~\ref{lem:operations},
it suffices to show that
$P$
can be obtained from
$Q_0$
by operations (a)--(c). By the proof of
Lemma~\ref{lem:unit_square}, there is a sequence of
integer pairs
$(m_i,n_i)$
such that the sequence of rectangles
$P_i$,
where
$P_1=P$, $P_{i+1}=\Psi_i(P_i)$
and
$\Psi_i=\Psi_{m_i,n_i}$,
converges to
$Q_0$, $P_i\to Q_0$.
For every integer
$i\ge 1$
we inductively define a rectangle
$Q_i$
such that
$Q_i=\Phi_i(Q_{i-1})$,
where
$\Phi_i=(\Psi_i\circ\dots\circ\Psi_1)^{-1}$.
Since
$\Phi_i^{-1}(P)=P_i\to Q_0$,
we have
$Q_i\to P$.
Furthermore,
$Q_i\in\cP_1$
for every
$i\ge 1$
because
$Q_0\in\cP_1$
by Lemma~\ref{lem:unit_square}. Therefore
$P\in\cP_1$.
\end{proof}

We denote with
$\cP$
the class of all the rectangles
$P\sub L$
with horizontal and vertical sides that satisfy
the area law of lifting,
$\de(P)^2=c_0^2\cdot\area P$,
with the lifting constant
$c_0$.

\begin{pro}\label{pro:area_law_all} Every rectangle
$P\sub L$
with horizontal and vertical sides is in the class
$\cP$.
\end{pro}

\begin{proof} By property~(H) and Lemma~\ref{lem:pure_homothety}
for every
$o\in X_\om$
and every
$\la>0$
there is a pure homothety
$h_\la:X_\om\to X_\om$
with 
$h_\la(o)=o$
and with coefficient
$\la$.
Then
$h_\la$
preserves every foliation of 
$X_\om$
by Busemann parallel Ptolemy lines. Thus its projection to the base,
$\ov h_\la=\pi_\ast(h_\la):B_\om\to B_\om$,
is a homothety with coefficient
$\la$
that fixes 
$\ov o=\pi_\om(o)$,
and
$\ov h_\la(l)=l$
for every (geodesic) line 
$l\sub B_\om$
through
$\ov o$.

We take
$o\in X_\om$
with 
$\ov o\in L$.
Then
$\ov h_\la$
preserves 
$L$
and horizontal and vertical directions on it.
Furthermore, for every
$P'=\ov h_\la(P)$, $P\sub L$
is a polygon, we have
$\de(P')=\la\de(P)$.
Taking
$P\in\cP_1$
we obtain
$$\de(P')^2=\la^2\de(P)^2=c_0^2\cdot\area P'.$$
Thus
$P'\in\cP$.
However, by Corollary~\ref{cor:each_par_p} every rectangle
$P'\sub L$
with horizontal and vertical sides can be represented as
$P'=\ov h_\la(P)$
for some
$P\in\cP_1$
and some
$\la>0$.
\end{proof}

\subsection{Metrics on
$\C$-lines}

Recall that every
$\C$-line
$F\sub X_\om$
is a fiber of the projection
$\pi_\om:X_\om\to B_\om$.

\begin{pro}\label{pro:euclid_square_fiber} Given
$x$, $y$, $z\in F$, $y$
lies between
$x$
and
$z$
with respect to the order on a $\C$-line
$F$,
we have
$$|xz|^2=|xy|^2+|yz|^2.$$
\end{pro}

\begin{proof} By Lemma~\ref{lem:unclosed_parall} there is
an unclosed parallelogram
$P\sub B_\om$.
Thus by Proposition~\ref{pro:area_law_all} there is 
a two-dimensional subspace 
$L\sub B_\om$
with horizontal and vertical directions satisfying 
the area law of lifting with a lifting constant 
$c_0>0$.
W.L.G. we can assume that 
$F=\pi_\om^{-1}(o)$
for some point 
$o\in L$.
By Proposition~\ref{pro:area_law_all}, for every
$a>0$
a rectangle
$P_a\sub L$
with the horizontal sides of length 1 and the vertical
sides of length
$a$
belongs to the class
$\cP$, $\de(P_a)^2=c_0^2\cdot a$.
We put
$a=|xy|^2/c_0^2$, $b=|yz|^2/c_0^2$
and represent the rectangle
$P_{a+b}$
as the union of
$P_a$
and
$P_b$
with a common horizontal side. Then
$\tau_{P_{a+b}}=\tau_{P_b}\circ\tau_{P_a}$
by Lemma~\ref{lem:adding_lifts}. Since
$\de(P_a)=c_0\sqrt a=|xy|$,
we have
$\tau_{P_a}(x)=y$
and similarly
$\tau_{P_b}(y)=z$.
Therefore
$\tau_{P_{a+b}}(x)=z$,
and we obtain
$$|xz|^2=\de(P_{a+b})^2=c_0^2(a+b)=\de(P_a)^2+\de(P_b)^2
=|xy|^2+|yz|^2.$$
\end{proof}

\section{Canonical complex structure on the base}
\label{sect:complex_structure}

In this section we show that the base
$B_\om$
possesses a complex structure uniquely determined by
the geometry of the space 
$X_\om$.
This complex structure is said to be {\em canonical}.

\subsection{Functional $\xi_u$}
\label{subsect:functional_xi}

We fix a base point
$o\in B_\om$
and regard
$B_\om$
as an Euclidean vector space identifying
$u\in B_\om$
with the vector
$\overrightarrow{ou}$.
Given a unit vector
$u\in B_\om$, $|u|:=|ou|=1$,
we let
$u^\bot\sub B_\om$
be the orthogonal complement to
$u$.
For every vector
$v\in u^\bot$
we denote with
$u\wedge v\sub B_\om$
the oriented rectangle spanned by
$u$, $v$,
and with
$\tau_{u\wedge v}:X_\om\to X_\om$
the respective lifting isometry, see sect.~\ref{subsect:lift_polygons}.
We define the sign of
$v$
with respect to
$u$
as
$$\sign_u(v)=\begin{cases}
+1,&x<\tau_{u\wedge v}(x)\\
0,&x=\tau_{u\wedge v}(x)\\
-1,&\tau_{u\wedge v}(x)<x
          \end{cases}$$
for some and hence any
$x\in X_\om$.
Now, we define a function
$\xi_u:u^\bot\to\R$
by
$$\xi_u(v)=\sign_u(v)\cdot\de(u\wedge v)^2,$$
where
$\de(u\wedge v)$
is the displacement of
$\tau_{u\wedge v}$, $\de(u\wedge v)=|x\tau_{u\wedge v}(x)|$
for some and hence any
$x\in X_\om$.

The following lemma will be used in the proof of
additivity of
$\xi_u$.

\begin{lem}\label{lem:lift_additivity} Let
$P=xyzu$, $Q=xyz'u'$, $Q'=u'z'zu$
be oriented rectangles in
$B_\om$
such that
$z'u'=Q\cap Q'$
is the common side of
$Q$
and
$Q'$.
Then for the closed oriented polygon
$Q\cup Q'=xyz'zuu'x\sub B_\om$
the lifting isometries
$\tau_P$, $\tau_{Q\cup Q}:X_\om\to X_\om$
coincide,
$\tau_P=\tau_{Q\cup Q}$.
\end{lem}

\begin{proof} The lifting isometries
$\tau_P$
and
$\tau_{Q\cup Q'}$
differ by lifting isometries
$\tau_{\De}$, $\tau_{\De'}$,
$$\tau_{Q\cup Q'}^{-1}\circ\tau_P=\tau_{\De'}\circ\tau_\De,$$
where
$\De=uxu'$, $\De'=zz'y$.
However, the triangle
$\De'$
is a shifted copy of the triangle
$\De$,
and
$\De$, $\De'$
enter the formula above with opposite orientations.
Thus their contributions cancel out, and we have
$\tau_P=\tau_{Q\cup Q}$.
\end{proof}

\begin{pro}\label{pro:linear_functional}
The function
$\xi_u:u^\bot\to\R$
is a nonzero linear functional for every unit vector
$u\in B_\om$.
\end{pro}

\begin{proof}
First, we show that
$\xi_u$
is additive,
$\xi_u(v+v')=\xi_u(v)+\xi_u(v')$
for all
$v$, $v'\in u^\bot$.
We denote
$w=v+v'$.
It follows from Lemma~\ref{lem:lift_additivity} and
Lemma~\ref{lem:adding_lifts} that
$\tau_{u\wedge w}=\tau_{u\wedge v'}\circ\tau_{u\wedge v}$.
We fix a fiber
$F\sub X_\om$,
a point
$x\in F$,
and assume W.L.G. that
$x<\tau_{u\wedge w}(x)$
(this assumption depends neither on
$F$
nor on
$x\in F$).
If
$x<\tau_{u\wedge v}(x)<\tau_{u\wedge w}(x)$,
then
$\sign_u(w)=\sign_u(v)=\sign_u(v')=1$
by definition and
$\de(u\wedge w)^2=\de(u\wedge v)^2+\de(u\wedge v')^2$
by Proposition~\ref{pro:euclid_square_fiber}.
In the opposite case we have W.L.G. that
$\tau_{u\wedge v'}(x)<x<\tau_{u\wedge w}(x)<\tau_{u\wedge v}(x)$
and thus
$\sign_u(w)=\sign_u(v)=-\sign_u(v')$,
$\de(u\wedge w)^2=\de(u\wedge v)^2-\de(u\wedge v')^2$.
In both cases this gives
$\xi_u(w)=\xi_u(v)+\xi_u(v')$.

Next, we show that 
$\xi_u$
is homogeneous. We have 
$\sign_u(\la v)=\sign(\la)\cdot\sign_u(v)$
for each
$v\in u^\bot$, $\la\in\R$,
because the orientation of the rectangle
$u\wedge(\la v)$
depends on the sign of
$\la$.
Thus using the area law of lifting, which holds 
by Proposition~\ref{pro:area_law_all} in the plane
$L$
spanned by
$u$, $v$,
with some constant
$c_0\ge 0$,
we have
$$\xi_u(\la v)=\sign(\la)\sign_u(v)c_0^2|\la||v|=
\la\sign_u(v)c_0^2|v|=\la\xi_u(v)$$
for every
$v\in u^\bot$
and every
$\la\in\R$.
Thus
$\xi_u:u^\bot\to\R$
is linear.

Finally, there are unclosed rectangles in
$B_\om$
and the subgroup of isometries in
$H=\pi_\ast(G_\om)$
preserving
$o$
acts transitively on the unit sphere
$S(o)$. 
Then there is a unit vector
$v\in u^\bot$ 
such that
$\sign_u(v)=1$.
Thus 
$\xi_u\neq 0$.
\end{proof}

We put
$c=\sup_Q\de(Q)$,
where the supremum is taken over all the unit squares
$Q\sub B$,
and call
$c$
the {\em lifting constant} of
$B$.

\begin{lem}\label{lem:xi_norm}
For every unit vector
$u\in B_\om$
the norm of the linear functional
$\xi_u:u^\bot\to\R$
is
$|\xi_u|=c^2$.
\end{lem}

\begin{proof}
By Proposition~\ref{pro:linear_functional},
$\xi_u\neq 0$.
Thus there is a unique unit vector
$v\in u^\bot$
with
$\xi_u(v)=|\xi_u|\le c^2$.

The lifting constant
$c$
can be computed by taking the supremum
$\sup_Q\de(Q)$
over the compact set of all unit squares
$Q\sub B_\om$
having the base point
$o$
as a vertex. By continuity of the lift
$Q\mapsto\tau_Q$,
there is a unit square
$Q_0$
with
$\de(Q_0)=c$.
Applying to
$Q_0$
an isometry of type
$\pi_\ast(\phi):B_\om\to B_\om$
if necessary, we can assume that
$Q_0=u\wedge v'$,
where
$v'\in u^\bot$, $\sign_u(v')=1$.
Then
$\xi_u(v')=c^2$.
It follows that
$v'=v$
and
$|\xi_u|=c^2$.
\end{proof}

\subsection{Complex structure on the base}
\label{subsect:complex_structure}

By Lemma~\ref{lem:xi_norm}, for every unit vector
$u\in B_\om$
there is a unique
$v\in u^\bot$, $|v|=1$, $\sign_u(v)=1$,
such that
$\xi_u(v)=c^2$.
This define a map
$J:S(o)\to S(o)$, $J(u)=v$,
where
$S(o)\sub B_\om$
is the unit sphere centered at
$o$.
We put
$J(o)=o$
and extend
$J$
on
$B$
by homogeneity,
$J(u)=|u|J(u/|u|)$
for every
$u\in B_\om$.

\begin{lem}\label{lem:commute_const_displacement}
Let
$h:S(o)\to S(o)$
be a map commuting with a subgroup 
$G\sub\bO(k)$, $k=\dim B_\om$,
acting transitively on
$S(o)$
(we do not require that
$h$
is an isometry). Then the displacement of
$h$
is constant. 
\end{lem}

\begin{proof}
Given
$x$, $y\in S(o)$,
there is 
$g\in G$
with
$g(x)=y$.
Then
$|yh(y)|=|g(x)h(g(x))|=|g(x)g(h(x))|=|xh(x)|$.
\end{proof}

\begin{pro}\label{pro:complex_structure} The map
$J:B_\om\to B_\om$
is a complex structure on
$B_\om$,
that is,
$J$
is a linear isometry with
$J^2=-\id$.
Moreover,
every M\"obius automorphism
$\phi:X_\om\to X_\om$
respecting the order of the
$\C$-lines in
$X_\om$
preserves
$J$, $\pi_\ast(\phi)\circ J=J\circ\pi_\ast(\phi)$.
\end{pro}

\begin{proof} 
For 
$u\in S(o)$
let
$K_u\sub u^\bot$
be the kernel of
$\xi_u$, $K_u=\ker\xi_u$.
By definition,
$v=J(u)=\grad\xi_u/|\grad\xi_u|$,
thus
$K_u\sub u^\bot\cap v^\bot$.
Since
$\dim K_u=\dim u^\bot-1=\dim(u^\bot\cap v^\bot)$,
we have
$K_u=u^\bot\cap v^\bot=K_v$,
that is, the kernel
$K_u$
is invariant under
$J$, $J(K_u)=K_v=K_u$.
Furthermore,
$\sign_v(u)=-\sign_u(v)$
because the rectangles
$u\wedge v$
and
$v\wedge u$
have opposite orientations. We conclude that
$J(v)=-u$,
i.e.,
$J^2=-\id$.

Let
$\phi:X_\om\to X_\om$
be a M\"obius automorphism that respects the order
$O$
of the
$\C$-lines in
$X_\om$.
Recall that then
$\phi$
is a homothety, because
$\phi$
preserves the infinitely remote point
$\om$.
Applying if necessary a shift, we can assume W.L.G. that the homothety
$\ov\phi=\pi_\ast(\phi):B_\om\to B_\om$
also preserves the base point
$o$, $\ov\phi(o)=o$.
Given
$u$, $v\in S(o)$, $v\in u^\bot$,
we denote by
$u'$, $v'\in S(o)$
the unit vectors
$u'=\ov\phi(u)/|\ov\phi(u)|$, $v'=\ov\phi(v)/|\ov\phi(v)|$.
Then
$v'\in u'^\bot$.
Furthermore, for the displacements
$\de(u\wedge v)$, $\de(u'\wedge v')$
we have
$$|\ov\phi(u)|\cdot|\ov\phi(v)|\cdot\de(u'\wedge v')=\de(\ov\phi(u)\wedge\ov\phi(v))
=\la^2\de(u\wedge v),$$
where
$\la>0$
is the homothety coefficient of
$\phi$,
the first equality follows from 
Proposition~\ref{pro:area_law_all} and the second one 
follows from the fact that the homothety
$\ov\phi:B_\om\to B_\om$
is the projection of the homothety
$\phi:X_\om\to X_\om$.
Thus
$\de(u'\wedge v')=\de(u\wedge v)$.
Since
$\phi$
preserves the order
$O$
of the
$\C$-lines
in
$X_\om$,
we have
$\sign_{u'}(v')=\sign_u(v)$
and thus
$\xi_{u'}(v')=\xi_u(v)$.
It follows
$\ov\phi(K_u)=K_{u'}$.
This means that
$\ov\phi\circ J(u)=\la J(u')$,
and we obtain
$$J\circ\ov\phi(u)=J(\la u')=\la J(u')=\ov\phi\circ J(u),$$
that is,
$J$
is preserved by any M\"obius automorphism
$\phi:X_\om\to X_\om$.

It remains to show that
$J:B_\om\to B_\om$
is a linear isometry. For every 
$u\in S(o)$,
the 2-dimensional subspace
$L_u\sub B_\om$
spanned by
$u$, $J(u)$,
as well as its orthogonal complement
$K_u=L_u^\bot\sub B_\om$,
is
$J$-invariant,
$J(L_u)=L_u$, $J(L_u^\bot)=L_u^\bot$.
By induction over the dimension of
$B_\om$,
we construct an orthonormal basic
$b=\{u_1,v_1,\dots,u_k,v_k\}$
of
$B_\om$,
where
$u_1=u$, $v_i=J(u_i)$, $i=1,\dots,k$.
In particular, the dimension of
$B_\om$
is even,
$\dim B=2m$.
Now, we define a linear map
$J_0:B_\om\to B_\om$
by
$J_0(u_i)=v_i$, $J_0(v_i)=-u_i$
for
$i=1,\dots,k$,
that is,
$J_0$
coincides with 
$J$
on the basis
$b$.

Let
$G$
be a subgroup of isometries of type
$\ov\phi=\pi_\ast(\phi):B_\om\to B_\om$, $\ov\phi(o)=o$,
where
$\phi:X_\om\to X_\om$
is a M\"obius automorphism that respects the order
$O$
of the 
$\C$-lines 
in
$X_\om$.
Since
$J_0$
coincides with
$J$
on the complex lines spanned by
$u_i$, $v_i$, $i=1,\dots,k$,
and
$J$
commutes with every
$\ov\phi\in G$,
we obtain that the isometry
$$g=\ov\phi\circ J_0\circ\ov\phi^{-1}\circ J_0^{-1}:
B_\om\to B_\om$$
is identical on every complex line above. These complex
lines span
$B_\om$,
thus
$g=\id$
for every
$\ov\phi\in G$,
that is, the group
$G$
centralizes
$J_0$.
By property (${\rm E}_2$),
$G$
acts transitively on
$S(o)$.

The map 
$h=J\circ J_0^{-1}:S(o)\to S(o)$
commutes with every
$g\in G$
because
$J$
and
$J_0$
do. Then by Lemma~\ref{lem:commute_const_displacement}
the displacement of
$h$
is constant. Since
$h(u_i)=u_i$
for every
$i=1,\dots,k$,
we obtain
$h=\id$
and
$J=J_0$.
\end{proof}

\begin{rem}\label{rem:complex_structure}
A {\em complex line} in
$B_\om$
is a 2-dimensional subspace
$L$
invariant under
$J$, $J(L)=L$.
By definition of
$J$,
the lifting constant for
$L$
takes the maximal value over all 2-dimensional subspaces.
\end{rem}

\section{Coordinates in 
$X_\om$}
\label{sect:coordinates}

For a given
$\om\in X$
we fix as usual a metric from the M\"obius structure
with infinitely remote point
$\om$.
We also fix an order
$O$
on
$X_\om$, 
a base point
$o\in X_\om$,
and identify the base
$B_\om$
with Euclidean space
$\R^k$, $k=\dim B_\om$,
with origin
$\pi_\om(o)$.
With
$\mu:X_\om\to F_o$
we denote the projection onto the fiber
$F_o=\pi_\om^{-1}(\pi_\om(o))$
of
$\pi_\om$.
It follows from Lemma~\ref{lem:continuity_project_fiber}
that
$\mu$
is continuous. We define the {\em standard coordinates} 
of every point
$x\in X_\om$
as
$(z,h)\in\R^k\times\R=\R^{k+1}$,
where
$z=\pi_\om(x)$, $h=\frac{1}{4}\sign\mu(x)|o\mu(x)|^2$,
$$\sign\mu(x)=\begin{cases}
+1&o<\mu(x)\\
0&o=\mu(x)\\
-1&\mu(x)<o.
          \end{cases}$$
The coefficient
$\frac{1}{4}$
in front of the expression for the coordinate 
$h$
is introduced to provide the property that the standard generator
$c=[a,b]$
of the center of the classical Heisenberg group
$\bH=\bH^1$
has coordinates
$c=(0,1)$.

\subsection{Multiplication law in coordinates}
Since the group
$N_\om$
acts on
$X_\om$
simply transitively, see Lemma~\ref{lem:simply_transitivity_shifts},
every isometry
$g\in N_\om$
can be written as
$g=(z,h)$,
where
$(z,h)$
are the coordinates of
$g(o)$.

\begin{lem}\label{lem:multi_law} For
$g=(z,h)$, $g'=(z',h')\in N_\om$
we have
$$g\cdot g'=(z+z',h+h'+h\circ\tau_T(o)),$$
where
$T=\frac{1}{2}(z\wedge z')=\pi_\om(o)z(z+z')\sub B_\om$
is the oriented triangle, 
$\tau_T:X_\om\to X_\om$
the respective lifting isometry, see sect.~\ref{subsect:lift_polygons}.
\end{lem}

\begin{proof} The center
$Z_\om$
of
$N_\om$
acts on
$X_\om$
by vertical shifts, and by Proposition~\ref{pro:euclid_square_fiber} 
we have 
$g\cdot g'=(0,h+h')$
for
$g=(0,h)$, $g'=(0,h')\in Z_\om$.

The group
$\pi_\ast(N_\om)=N_\om/Z_\om$
acts on the base
$B_\om$
by shifts, and we have
$\pi_\ast(g)\cdot\pi_\ast(g')=z+z'$
for
$g=(z,h)$, $g'=(z',h')\in N$.
 
Let
$l\sub X_\om$
be the Ptolemy line 
through
$o$
such that 
$z:B_\om\to B_\om$
is the shift along the line
$\pi_\om(l)$.
There is an isometry
$\wt g\in N_\om$
which preserves
$l$
and projects down to
$z=\pi_\ast(\wt g)$.
Then
$\wt g=(z,0)$,
and
$g=(z,h)$
can be written as
$(z,h)=(z,0)\cdot(0,h)=\wt g\cdot(0,h)$.
Similarly we have
$g'=(z',h')=(z',0)\cdot(0,h')=\wt g'\cdot(0,h')$,
and the isometry
$g'\in N_\om$
preserves a line
$l'\sub X_\om$
through
$o$.
Then
$l''=\wt g(l')$
is the Ptolemy line in
$X_\om$
through
$\wt g(o)=(z,0)$
and
$\wt g\cdot\wt g'(o)$.
The 
$z$-coordinate of the point 
$\wt g\cdot\wt g'(o)$
is
$z+z'$.

Now, we compute the 
$h$-coordinate 
of
$\wt g\cdot\wt g'(o)$.
Let
$z\wedge z'\sub B_\om$
be the oriented parallelogram spanned by
$z$, $z'$,
$T=\frac{1}{2}(z\wedge z')$
the oriented triangle
$\pi_\om(o)z(z+z')$.
Then
$\mu(\wt g\cdot\wt g'(o))=\tau_T(o)$.
Thus
$\wt g\cdot\wt g'=(z+z',h\circ\tau_T(o))$,
and we obtain
$g\cdot g'=\wt g\cdot(0,h)\cdot\wt g'\cdot(0,h')=
(z+z',h+h'+h\circ\tau_T(o))$.
\end{proof}

\begin{lem}\label{lem:coord_lift_triangle}
For any triangle
$T=\frac{1}{2}z\wedge z'\sub B_\om$,
we have
$$h\circ\tau_T(o)=\frac{c^2}{8}\langle J(z),z'\rangle,$$
where
$c>0$
is the lifting constant of
$B_\om$.
\end{lem}

\begin{proof}
By Lemma~\ref{lem:adding_lifts},
$\tau_{z\wedge z'}=\tau_T^2$.
Thus by Proposition~\ref{pro:euclid_square_fiber}
$|o\tau_T(o)|^2=\frac{1}{2}\de(z\wedge z')^2$,
where
$\de(z\wedge z')$
is the displacement of the isometry
$\tau_{z\wedge z'}$.
By results of section~\ref{subsect:area_law_lift}
the lifting isometry
$\tau_{z\wedge z'}$
does not change when we replace the parallelogram
$z\wedge z'$
by a rectangle of equal area that has a side of length one
in the same 2-dimensional subspace. Thus we can assume that 
$|z|=1$
and
$z\bot z'$.
Since
$\tau_T(o)\in F_o$,
we have
$\mu\circ\tau_T=\tau_T$
and 
$\sign\mu\circ\tau_T=\sign_z(z')$,
see section~\ref{sect:complex_structure}.
It follows that 
$h\circ\tau_T(o)=\frac{1}{8}\xi_z(z')$,
where the linear functional
$\xi_z:z^\bot\to\R$
is used in the definition of the complex structure
$J$
(for the definition of 
$\xi_z$
see sect.~\ref{subsect:functional_xi}).

By Lemma~\ref{lem:xi_norm},
$|\xi_z|=c^2$.
Then
$\xi_z(z')=c^2\langle J(z),z'\rangle$
and
$h\circ\tau_T(o)=\frac{c^2}{8}\langle J(z),z'\rangle$.
\end{proof}

\subsection{The distance function D}
\label{subsect:distance_function}

For two points
$x,y\in X_{\om}$
let
$F_x$
and
$F_y$
be the
$\C$-lines
through
$x,y$,
and let
$\mu_{xy}:F_x\to F_y$
be the projection, see sect.~\ref{subsect:both_foliation_equidistant}.
We denote with
$$|xy|_\om^\R=|\pi_\om(x)\pi_\om(y)|,\quad |xy|_\om^\C=|\mu_{xy}(x)y|.$$
Note that
$|yx|_\om^\C=|xy|_\om^\C$
because Ptolemy lines are equidistant on every \semi-plane in
$X_\om$,
see Lemma~\ref{lem:equidist_lines}.

\begin{lem}\label{lem:univ_dist_function}
There exists some function
$D:[0,\infty)\times [0,\infty)\to [0,\infty)$,
such that for all
$\om \in X$
and all
$x,y \in X_{\om}$
$$|xy|_\om=D(|xy|_\om^\R,|xy|_\om^\C).$$
\end{lem}

\begin{proof}
To prove this, we have to show, that given triples of points
$\om,x,y$
and
$\om',x',y'$
in
$X$
with
$|xy|_\om^\R=|x'y'|_{\om'}^\R$
and
$|xy|_\om^\C=|x'y'|_{\om'}^\C$
we have
$|xy|_\om=|x'y'|_{\om'}$.

We can assume that
$|xy|_\om^\R\neq 0$, since otherwise the claim is trivial.
Let
$z=\mu_{xy}(x)$
and
$z'=\mu_{x'y'}(x')$.
By assumption
$|xz|_\om=|x'z'|_{\om'}\ne 0$
and
$|zy|_\om=|z'y'|_{\om'}$.
By (${\rm E}_2$) there exists a M\"obius map
$\phi:X\to X$
which maps
$\om\to\om'$, $x\to x'$, $z\to z'$
because each triple of points
$(\om,x,z)$
and
$(\om',x',z')$
belongs to a respective
Ptolemy circle in
$X$.
Note that
$\phi:X_\om\to X_{\om'}$
is an isometry that maps the fibers
$F_x$, $F_y=F_z$
to
$F_{x'}$, $F_{y'}=F_{z'}$.
Then
$\phi$
maps
$y$
either to
$y'$
or to
$y'' \in F_{y'}$,
the points symmetric to
$y'$
with
$|y''z'|_{\om'}=|z'y'|_{\om'}$.
In the last case, applying an isometry of
$X_{\om'}$
that order preserves the
$\C$-lines $F_{x'}$, $F_{y'}$
and maps
$y'$
to
$z'$,
and Lemma~\ref{lem:rectangle_side_length}, we have
$|xy|_\om=|x'y''|_{\om'}=|x'y'|_{\om'}$.
This proves our claim.
\end{proof}

\begin{lem}\label{lem:homogeneous_dist_function}
The distance function
$D$
is homogeneous,
$D(\la a,\la b)=\la D(a,b)$
for every
$a$, $b\ge 0$, $\la>0$.
\end{lem}

\begin{proof}
Let
$\phi_\la:X_\om\to X_\om$
be a homothety with coefficient 
$\la$.
Then for each
$x$, $x'\in X_\om$
we have
$|\phi_\la(x)\phi_\la(x')|_\om^\R=\la|xx'|_\om^\R$
and
$|\phi_\la(x)\phi_\la(x')|_\om^\C=\la|xx'|_\om^\C$.
Together with
$|\phi_\la(x)\phi_\la(x')|=\la|xx'|$
this implies the claim.
\end{proof}

\subsection{Existence of a vertical flip}
\label{subsect:exist_vert_flip}

Results of this section have important
applications in sect.~\ref{sect:rcircles}.

A {\em vertical flip} w.r.t. 
$o\in X_\om$
is an isometry
$j:X_\om\to X_\om$
that reverses the order
$O$
and preserves
$o$, $j(o)=o$.
In particular,
$j(F)=F$
for the 
$\C$-line 
$F$
through
$o$
and
$j|F$
reverses the order of
$F$.

Let
$J:B_\om\to B_\om$
be the canonical complex structure on
$B_\om$.
The orthonormal basis
$b=\{u_1,v_1,\dots,u_k,v_k\}$,
where
$v_i=J(u_i)$, $i=1,\dots,k$,
is called a {\em canonical basis} of
$B_\om$
for the complex structure
$J$.
We define a linear map
$\conj:B_\om\to B_\om$
by
$\conj(u_i)=u_i$, $\conj(v_i)=-v_i$
for 
$i=1,\dots, k$.
Then
$\conj$
is an isometry that anticommutes with
$J$, $J\circ\conj=-\conj\circ J$.
In particular, we have
\begin{equation}\label{eq:conjugate}
 \langle J\circ\conj(z),\conj(z')\rangle
=-\langle J(z),z'\rangle
\end{equation}
for each 
$z$, $z'\in B_\om$,
where
$\langle z,z'\rangle$
is the inner product of
$z$, $z'$
of the Euclidean space
$B_\om$.

\begin{pro}\label{pro:vert_flip} For every
$o\in X_\om$
and a Ptolemy line
$l\sub X_\om$
through
$o$
there exists a vertical flip of
$X_\om$
with respect to
$o$
that fixes
$l$
pointwise.
\end{pro}

\begin{proof} 
We introduce standard coordinates
$x=(z,h)$
in
$X_\om$
with the origin 
$o$
and take a canonical basis
$b=\{u_1,v_1,\dots,u_k,v_k\}$
of the complex structure
$J$
such that the vector
$u_1$
generates 1-dimensional subspace
$\pi_\om(l)\sub B_\om$.
Let
$\conj:B_\om\to B_\om$
be the conjugation isometry associated with
$b$.
We define 
$j:X_\om\to X_\om$
as
$$j(x)=j(z,h)=(\conj(z),-h).$$
Then
$j(o)=o$, $j$
fixes
$l$
pointwise by the choice of
$b$, $j(F)=F$,
where
$F$
is the 
$\C$-line
through
$o$,
and
$j|F$
reverses the order. We only have to show that
$j$
is an isometry. 

Recall that 
$X_\om$
is identified with the group
$N_\om$
by
$x=g(o)$, $x\in X_\om$, $g\in N_\om$.
For 
$g\in N_\om$, $g=(z,h)$, 
we put
$\|g\|:=|og(o)|=D(|z|,2\sqrt{|h|})$,
where we used Lemma~\ref{lem:univ_dist_function}
in the last equality. Since
$N_\om$
acts on
$X_\om$
by isometries, for
$x=g(o)$, $x'=g'(o)$,
we have
$|xx'|=|g(o)g'(o)|=|og^{-1}\cdot g'(o)|=\|g^{-1}\cdot g'\|$.
Using Lemma~\ref{lem:multi_law}, we obtain for
$g=(z,h)$, $g'=(z',h')$
$$g^{-1}\cdot g'=(-z,-h)\cdot(z',h')=
(z'-z,h'-h+h\circ\tau_T(o)),$$
where
$T=\frac{1}{2}(-z\wedge z')=\pi(o)(-z)(z'-z)\sub B_\om$
is the oriented triangle. By Lemma~\ref{lem:coord_lift_triangle},
$h\circ\tau_T(o)=\frac{c^2}{8}\langle J(-z),z'\rangle
=-\frac{c^2}{8}\langle J(z),z'\rangle$.
Therefore
\begin{equation}\label{eq:distance_formula}
|xx'|=D\left(|z-z'|,2|h-h'+\frac{c^2}{8}\langle J(z),z'\rangle|^{1/2}\right).
\end{equation}

Similarly for 
$j(x)=(\conj(z),-h)$, $j(x')=(\conj(z'),-h')$
we have
\begin{align*}
|j(x)j(x')|&=\|\left(\conj(z),-h\right)^{-1}\cdot\left(\conj(z'),-h'\right)\|\\
&=\|\left(\conj(z')-\conj(z),h-h'+h\circ\tau_{T'}(o)\right)\|, 
\end{align*}
where
$T'=\frac{1}{2}(-\conj(z)\wedge\conj(z'))=\pi(o)(-\conj(z))(\conj(z'-z))$
is the oriented triangle in
$B_\om$.
Then
$h\circ\tau_{T'}(o)=-\frac{c^2}{8}\langle J\circ\conj(z),\conj(z')\rangle
=\frac{c^2}{8}\langle J(z),z'\rangle$
by Equality~(\ref{eq:conjugate}). Using that
$\conj:B_\om\to B_\om$
is an isometry, we obtain
$$|j(x)j(x')|=D\left(|z-z'|,2|h-h'+\frac{c^2}{8}\langle J(z),z'\rangle|^{1/2}\right).$$
Comparing the formulae for the distances
$|xx'|$, $|j(x)j(x')|$
we see that
$j$
is an isometry.
\end{proof}

\section{Ptolemy circles in $X_\om$}
\label{sect:rcircles}

Here we study shape of Ptolemy circles in
$X_\om$.
Results of this section are used to compute the lifting constant
$c$.

\subsection{Ptolemy circles meeting a $\C$-line twice}
Let now
$\om \in X$,
and let
$F\sub X_\om$
be a
$\C$-line.
With
$\mu=\mu_\om:X_{\om} \to F$
we denote the projection onto
$F$.

\begin{lem}\label{lem:common_rline} Let
$z\in F$
and assume
$\mu(u)=z$
and
$\mu(v)=z$
for
$u$, $v\in X_\om$
such that
$|uv|=|uz|+|zv|$.
Then
$u$, $v$, $z$
lie on a common
Ptolemy line
in
$X_\om$.
\end{lem}

\begin{proof} 
We can assume that both
$u$, $v$
are not on
$F$,
since otherwise the claim is obvious. Let
$l$
be the
Ptolemy line through
$z$
and
$u$, and let
$l'$
be the
Ptolemy line
through
$z$ and
$v$.
Let
$c,c':\R\to X_{\om}$
be unit speed parameterizations of
$l,l'$,
such that
$c(0)=c'(0)=z$,
$c(t_0)=u$,
for $t_0=|zu|$
and
$c'(s_0)=v$,
for
$s_0=-|vz|$.
By assumption
$|c(t_0)c'(s_0)|=t_0-s_0$.
For every
$\la>0$
there is a pure homothety
$\phi_\la:X_\om\to X_\om$, $\phi_\la(z)=z$,
with coefficient
$\la$.
Then
$\phi_\la\circ c(t)=c(\la t)$ 
for every
$t\in\R$,
and
$\phi_\la\circ c'(s)=c'(\la s)$
for every
$s\in\R$.
Thus
$|c(t)c'(s)|=t-s$
for all
$t\ge 0$, $s\le 0$.
This implies
$b(c'(s))=-s$
for all
$s\le 0$,
where
$b:X_\om\to\R$
is a Busemann function, associated with
$l$, $b(y)=\lim_{t\to \infty}(|c(t)y|-|c(t)z|)$, $y\in X_\om$.
Thus 
$b(c'(s))=-s$
for all
$s\in\R$,
since
$b$
is affine. Therefore
$b\circ c'=b\circ c$.
Then
$l'=l$
by Lemma~\ref{lem:unique_line}.
\end{proof}

\begin{lem}\label{lem:rline_through_three_points}
Let
$\si\sub X_\om$
be a
$\R$-circle 
intersecting the
$\C$-line
$F$ 
in two points
$x$ and $y$.
Let
$u,v\in\si$
be points in different components of
$\si\sm\{x,y\}$
such that
$\mu(u)=\mu(v)=z\in F$.
Then the points
$u,z,v$
are contained in a Ptolemy line.
\end{lem}

\begin{proof} Let
$l$
be the
Ptolemy line through the
points
$u$ and $z=\mu(u)$.
Consider the point
$v'$ on
$l$ with order
$v'<z<u$ such that
$|v'z|=|vz|$.
We show that
$v=v'$.
By the Ptolemy equality on $\si$ we have
$$|vu| |xy| = |xu| |yv| + |uy| |v x|,$$
where we use that the pair
$u$, $v\in\si$
is separated by the pair
$x$, $y\in\si$.
On the other hand, we have
$$|v'u| |xy|\le |xu| |yv'| + |uy| |v' x|.$$
Note also that
$|vy|=|v'y|, |vx|=|v'x|$
by Lemma~\ref{lem:univ_dist_function}. Then
$|v'u|\le |vu|$.
On the other hand,
$|vu|\le|v'u|$
by definition of
$v'$.
Thus
$|vu|=|v'u|$
and
$v'=v$
by Lemma~\ref{lem:common_rline}.
\end{proof}

\begin{cor}\label{cor:two_points_circle} Let
$F\sub X_\om$
be a $\C$-line, and
$\si\sub X_\om$
a Ptolemy circle intersecting
$F$ 
in two distinct points
$x$, $y\in F$.
Let
$z\in F$
a point between
$x$ 
and 
$y$, $x<z<y$.
Then
$|\mu^{-1}(z)\cap \si| =2$.
\end{cor}

\begin{proof} Note that
$\om\not\in\si$
since otherwise
$\si$
is a Ptolemy line in
$X_\om$
intersecting
$F$
in
$x$, $y$
in contradiction with Corollary~\ref{cor:3c_4c}($4_\C$).
By the continuity of
$\mu$
we have
$|\mu^{-1}(z)\cap \si|\ge 2$,
since at least one point of every arc of
$\si$
from
$x$ to $y$
projects to
$z$.
Assume that there are 3 points projecting to 
$z$, $u$ 
from one arc, and
$v$, $v'$ 
from the other, then by Lemma~\ref{lem:rline_through_three_points} 
the points
$u$, $v$, $v'$ 
are on the Ptolemy line through
$\mu(u)$
and
$u$.
But they are also on
$\si$.
Thus
$\si$ 
and this Ptolemy line coincide by Corollary~\ref{cor:weak_unique}.
This is a contradiction.
\end{proof}

\begin{lem}\label{lem:one_circle} Given a $\C$-line 
$F\sub X_\om$,
points
$x$, $y$, $z\in F$
such that
$x<z<y$
and a Ptolemy line
$l\sub X_\om$
through
$z$,
there is at most one Ptolemy circle through
$x$, $y$
that meets
$l$.
\end{lem}

\begin{proof} 
It follows from Lemma~\ref{lem:rline_through_three_points}
and Corollary~\ref{cor:two_points_circle} that every
Ptolemy circle 
$\si$
through
$x$, $y$
that meets
$l$
intersects
$l$
in two points
$u$, $v$
separated by
$z$, $u<z<v\sub l$.
If
$\si'\neq\si$
is another Ptolemy circle through
$x$, $y$
that meets
$l$
in points
$u'$, $v'$, $u'<z<v'\sub l$,
then the points
$u$, $v$, $u'$, $v'\sub l$
are pairwise distinct by Corollary~\ref{cor:weak_unique}.
Consider now the situation in a metric of the M\"obius structure 
with infinitely remote point
$y$.
In this metric
$l$
is a Ptolemy circle intersecting
$F$ 
in
$z$, $\om$,
and
$x\in F$
is a point between
$z$
and
$\om$.
The Ptolemy circles
$\si$, $\si'$
are Ptolemy lines through
$x$.
Let
$\mu_y:X_y\to F$
be the projection in this metric. Then
$\mu_y(u)=\mu_y(v)=\mu_y(u')=\mu_y(v')=x$,
in contradiction to Corollary~\ref{cor:two_points_circle}. Thus
$\si=\si'$.
\end{proof}

\begin{lem}\label{lem:equal_distances} Under the assumptions of
Lemma~\ref{lem:rline_through_three_points} we have
$|uz|=|zv|$.
\end{lem}

\begin{proof} 
We denote by
$l$ 
the Ptolemy line containing
$u$, $z$, $v$.
By Proposition~\ref{pro:flip_extension}, there exists an isometry
$\phi$
of
$X_\om$
preserving the order of $\C$-lines
that extends the flip of
$l$
at
$z$.
Thus
$\phi(x)=x$, $\phi(y)=y$, $\phi(z)=z$, $\phi(\om)=\om$.
Then
$\si'=\phi(\si)$
is a Ptolemy circle through
$x$, $y$
that meets
$l$
at the points
$\phi(u)$, $\phi(v)$.
By Lemma~\ref{lem:one_circle},
$\si'=\si$,
hence
$\phi(u)=v,\phi(v)=u$
and thus
$|uz|=|zv|$.
\end{proof}

\begin{cor}\label{cor:circle_cover_twice} The projection
$\mu(\si)$
is the segment in
$F$
from
$x$
to
$y$.
Let
$\si_1$
and
$\si_2$
be the two segments of
$\si$
with endpoints
$x$
and
$y$,
then
$\mu|\si_i:\si_i\to \mu(\si_i)$
is a homeomorphism for
$i=1,2$.
\end{cor}

\begin{proof} Let
$z\in F$
be a point between
$x$
and
$y$, $x<z<y$.
There are
$u_i\in\si_i$
with
$\mu(u_i)=z$, $i=1,2$,
and by Corollary~\ref{cor:two_points_circle} the points
$u_1$, $u_2$
are uniquely determined. By Lemma~\ref{lem:equal_distances}
we have
$|u_1z|=|u_2z|$.
Then
$|u_1z|=|u_2z|\to 0$
as
$z\to x$
or
$y$.
Indeed otherwise we find
$u_1\in\si_1$, $u_2\in\si_2$
distinct from
$x$, $y$
such that
$\mu(u_1)=\mu(u_2)=x$
or
$y$.
Then by Lemma~\ref{lem:rline_through_three_points} there is a
Ptolemy line $l\sub X_\om$
through
$u_i$, $\mu(u_i)$, $i=1,2$.
These points are common for
$l$
and
$\si$,
hence
$l=\si$
by Corollary~\ref{cor:weak_unique}, a contradiction since
$\om\not\in\si$.

If follows that
$\mu^{-1}(x)\cap\si=x$
and
$\mu^{-1}(y)\cap \si =y$.
Furthermore
$\mu^{-1}(w)\cap\si =\es$,
if
$w\in F$
is not in the segment from
$x$
to
$y$,
since
$\si\sub X_\om$
is an embedded circle. Therefore both the restrictions
$\mu|\si_i:\si_i\to\mu(\si_i)$, $i=1,2$,
are continuous bijections and thus homeomorphisms.
\end{proof}

\begin{lem}\label{lem:mean_geometric} Given points
$x<z<y$
in the 
$\C$-line $F$
and a Ptolemy circle 
$\si\sub X_\om$
through
$x$, $y$,
we have
$|xz|\cdot|zy|=|zu|^2$,
where
$u\in\si$
is a point with
$\mu(u)=z$.
\end{lem}

\begin{proof}
We denote
$|xz|=a$, $|zy|=d$.
In a first step we show that the distance 
$r=|uz|$
depends only on
$a$, $d$.
Let
$F'$
be a $\C$-line
with infinitely remote point
$\om'$, $x'<z'<y'$
be points in
$F'$
with 
$|x'z'|=a$, $|z'y'|=d$, $\si'\sub X_{\om'}$
a  Ptolemy circle through
$x'$, $y'$,
a point
$u'\in\si'$
is projected to
$z'$
by the respective projection
$\mu':X_{\om'}\to F'$,
where the distances are taken in a metric
$d_{\om'}$
of the M\"obius structure. 

There are Ptolemy circles
$l$
and 
$l'$
in
$X$
through
$\om$, $z$, $u$
and
$\om'$, $z'$, $u'$
respectively. By (${\rm E}_2$), a M\"obius map
$\phi:l\to l'$
with
$\phi(\om)=\om'$, $\phi(z)=z'$, $\phi(u)=u'$,
is extended to a M\"obius automorphism 
$X\to X$
for which we use the same notation
$\phi$.
Then
$\phi:(X,d_\om)\to (X,d_{\om'})$
is a homothety with
$\phi(F)=F'$.
Thus applying if necessary a homothety w.r.t. the metric
$d_{\om'}$
that fixes
$z'$
and leaves invariant
$l'$,
we can assume that
$|z'\phi(x)|=|z'x'|=a$
and
$|z'\phi(y)|=|z'y'|=d$.
Then
$\phi:(X,d_\om)\to (X,d_{\om'})$
is an isometry. Next, applying if necessary a vertical flip 
$j:X_{\om'}\to X_{\om'}$
which fixes the Ptolemy line 
$l'$
pointwise, we can assume that
$\phi(x)=x'$, $\phi(y)=y'$.
Then by Lemma~\ref{lem:one_circle}, we have
$\phi(\si)=\si'$.
Hence using Lemma~\ref{lem:equal_distances}, we obtain
$|z'u'|=|zu|=r$.

In a second step we show that
$r^2=ad$.
Let
$d_z$
be the inversion of the metric
$\la d_\om(p,q)=\la|pq|$
with respect to
$z$,
where
$\la=1/ad$, $d_z(p,q)=\frac{ad|pq|}{|zp|\cdot|zq|}$.
Then
$d_z(\om,x)=d$, $d_z(\om,y)=a$,
and 
$\si$
is still a Ptolemy circle through the points
$x$, $y$
in the $\C$-line 
$F$
with infinitely remote point 
$z$.
The Ptolemy line 
$l$
in the metric 
$d_\om$
through
$z$, $u$
with infinitely remote point 
$\om$
is the Ptolemy line in the metric
$d_z$
through
$\om$, $u$
with infinitely remote point 
$z$.
Hence
$d_z(\om,u)=r$
by the first step. On the other hand, we have
$d_z(\om,u)=ad/r$.
Thus
$ad=r^2$.
\end{proof}

\begin{pro}\label{pro:unit_rcircle}
Let
$\si\sub X_\om$
be a Ptolemy circle that intersects the $\C$-line
$F$
in distinct points
$x$, $y$
and meets a Ptolemy line 
$l\sub X_\om$
in points
$u$, $v$
with
$|uz|=|zv|=1$,
where
$z\in l\cap F$
is the midpoint between
$x$, $y$.
Then
$|zw|=1$
for every
$w\in\si$.
\end{pro}

\begin{proof} Let
$\phi:l\to l$
be a M\"obius involution with
$\phi(u)=u$, $\phi(v)=v$
and
$\phi(z)=\om$, $\phi(\om)=z$.
By (${\rm E}_2$),
$\phi$
extends to a M\"obius automorphism of
$X$
for which we use the same notation,
$\phi:X\to X$.
Then
$\phi$
preserves the $\C$-circle
$F$, $\phi(F)=F$,
because by Corollary~\ref{cor:3c_4c}($3_\C$),
$F$
is the unique $\C$-circle
through
$z$, $\om$.

Let
$d_z$
be the metric inversion of 
$d_\om$
w.r.t.
$z$, 
where
$d_\om(p,q)=|pq|$
for
$p$, $q\in X$,
that is,
$d_z(p,q)=\frac{|pq|}{|zp|\cdot|zq|}$.
Then
$\phi:(X,d_\om)\to(X,d_z)$
is an isometry because
$$d_z(\phi(u),\phi(z))=d_z(u,\om)=\frac{1}{d_\om(z,u)}=1.$$
We have
$|xz|=|zy|$,
and
$|xz|\cdot|zy|=1$
by Lemma~\ref{lem:mean_geometric}. Thus
$|xz|=|zy|=1$.
Then
$d_z(x,\om)=1/|zx|=1$
and similarly
$d_z(y,\om)=1$.
It follows that
$\phi$
preserves the set 
$\{x,y\}$, $\phi(\{x,y\})=\{x,y\}$
and hence the Ptolemy circle 
$\si$
is invariant under
$\phi$, $\phi(\si)=\si$.

By Proposition~\ref{pro:vert_flip}, there exists a vertical flip 
$j:X_\om\to X_\om$
that fixes the Ptolemy line 
$l$
pointwise. Then
$j$
flips
$F$, $j(x)=y$, $j(y)=x$
and thus
$j(\si)=\si$.
Applying
$j$
if necessary, we can assume that
$\phi(x)=x$, $\phi(y)=y$.
Then
$\phi$
fixes
$\si$
pointwise,
$\phi(w)=w$
for every
$w\in\si$.
Now, we have
$$|xw|=d_z(\phi(x),\phi(w))=d_z(x,w)
=\frac{|xw|}{|zw|\cdot|zx|}=\frac{|xw|}{|zw|},$$
which implies
$|zw|=1$
for every
$w\in\si$.
\end{proof}

\subsection{Computing the distance function $D$}
Now, we are able to find the distance function
$D$,
see section~\ref{subsect:distance_function}.
\begin{pro}\label{pro:comp_dist_function}
We have
$D(a,b)^4=a^4+b^4$
for each
$a$, $b\ge 0$.
\end{pro}

\begin{proof}
Let
$\si\sub X_\om$
be a Ptolemy circle through points
$x$, $y$
in the $\C$-line 
$F$, $o\in F$
the midpoint between
$x$, $y$, $|xo|=|oy|=1$.
By rescaling with coefficient
$\la>0$,
which has to be found, we assume that
$|oz|=\la b$, $|zu|=\la a$
for some
$u\in\si$
with
$\mu(u)=z$, 
where
$z\in F$
is a point between
$o$
and
$y$.
Using Proposition~\ref{pro:euclid_square_fiber} and notations
$|xz|=c$, $|zy|=d$,
we obtain
$d^2=1-\la^2b^2$, $c^2=1+\la^2b^2$.
By Lemma~\ref{lem:mean_geometric}, we have
$c^2d^2=\la^4a^4$.
Hence
$\la^4=1/(a^4+b^4)$.
By Proposition~\ref{pro:unit_rcircle},
$|ou|=1=\la D(a,b)$.
Therefore,
$D(a,b)^4=1/\la^4=a^4+b^4$.
\end{proof}

\subsection{The lifting constant}
\label{subsect:lifting_constant}

\begin{lem}\label{lem:complex_line} Let 
$\si\sub X_\om$
be a Ptolemy circle centered at
$z=l\cap F$,
where the Ptolemy line 
$l\sub X_\om$
intersects
$\si$
in
$u$, $v$,
the $\C$-line 
$F\sub X_\om$
intersects
$\si$
in
$x$, $y$.
Let
$l'\sub X_\om$
be the tangent Ptolemy line to
$\si$
at the point
$u$.
Then the lines
$\ov l=\pi_\om(l)$, $\ov l'=\pi_\om(l')\sub B_\om$
span a complex line, that is,
$J(\ov l)\parallel \ov l'$
for the canonical complex structure
$J:B_\om\to B_\om$.
\end{lem}

\begin{proof}
We regard
$B_\om$
as an Euclidean space with origin
$o=\pi_\om(z)=\pi_\om(x)=\pi_\om(y)$,
and 
$J$
as an isometry of
$B_\om$
with
$J(o)=o$. 

By Proposition~\ref{pro:vert_flip}, there is a vertical flip 
$j:X_\om\to X_\om$
with respect to 
$z$
that fixes
$l$
pointwise. Then
$j(u)=u$, $j(v)=v$, $j(x)=y$
and
$j(y)=x$.
Thus
$j$
preserves
$\si$, $j(\si)=\si$.
Moreover 
$j$
preserves
$l'$, $j(l')=l'$,
by uniqueness the tangent Ptolemy line,
see Proposition~\ref{pro:tangent_rcircle}, and it acts on
$l'$
as a flip because
$j$
flips the Ptolemy circle 
$\si$.

Let
$\ov l''\sub B_\om$
be the line through the origin
$o$
parallel to
$\ov l'$.
Recall that
$\pi_\ast(j)=\conj$,
see sect.~\ref{subsect:exist_vert_flip}. Since
$j$
preserves
$l$
pointwise and flips
$l'$,
the line
$\ov l'$
and thus the line
$\ov l''$
is orthogonal to
$\ov l$.
It suffices to show that
$J(\ov l)=\ov l''$.
To this end, we use a freedom in the construction of
the vertical flip 
$j$.
Namely, assume that 
$J(\ov l)\neq\ov l''$.
Since
$J(\ov l)\perp\ov l$,
the projection
$m\sub B_\om$
of
$\ov l''$
to the orthogonal complement
$L^\bot\sub B_\om$
of the subspace
$L$
spanned by
$u_1$, $v_1=J(u_1)$
is nontrivial. We construct a canonical basis
$b=\{u_1,v_1,\dots,u_k,v_k\}$
of
$B_\om$
for the complex structure
$J$,
see sect.~\ref{subsect:exist_vert_flip}, so that
$u_1$
generates
$\ov l$
and
$u_2$
generates 
$m$.
Then for the respective conjugation isometry
$\conj:B_\om\to B_\om$
we have
$\conj(\ov l'')\neq\ov l''$
unless
$\ov l''=m$
because by definition
$\conj(u_i)=u_i$, $\conj(v_i)=-v_i$
for
$i=1,\dots,k$.

Since
$j(l')=l'$,
it follows from definition of the respective flip
$j:X_\om\to X_\om$,
see Proposition~\ref{pro:vert_flip}, that
$\conj(\ov l'')=\ov l''$
and hence
$\ov l''=m$.
But then the isometry
$\conj$
preserves
$\ov l''$
pointwise and thus
$j$
preserves the tangent Ptolemy line 
$l'$
pointwise. This contradicts the property that
$j$
acts on
$l'$
as a flip. Therefore
$J(\ov l)=\ov l''$,
and
$\ov l$, $\ov l'$
span the complex line 
$L\sub B_\om$.
\end{proof}

\begin{lem}\label{lem:mu_distortion}
Let
$\mu=\mu_F:X_\om\to X_\om$
be the projection onto a $\C$-line 
$F\sub X_\om$.
Then for each
$u$, $v\in X_\om$
we have
$$|\mu(u)\mu(v)|^2\le|uv|^2+\de(T)^2,$$
where
$T=\ov u\,\ov z\,\ov v$
is the (oriented) triangle in the base
$B_\om$
with vertices
$\ov z=\pi_\om(F)$, $\ov u=\pi_\om(u)$, $\ov v=\pi_\om(v)$,
and
$\de(T)$ 
is the displacement of the lifting isometry
$\tau_T:X_\om\to X_\om$.
\end{lem}

\begin{proof}
Let
$F'=\pi^{-1}(\ov u)$
be the
$\C$-line
in
$X_\om$
through
$u$, $u'=\tau_T(u)\in F'$, $v'=\mu_{F'}(v)\in F'$.
By Proposition~\ref{pro:euclid_square_fiber}, we have
$|u'v'|^2=||uv'|^2\pm|uu'|^2|$,
and
$|uu'|=\de(T)$.
It follows from Proposition~\ref{pro:comp_dist_function} that
$|uv'|\le|uv|$.
Thus
$$|\mu(u)\mu(v)|^2=|u'v'|^2\le|uv|^2+\de(T)^2.$$
\end{proof}

\begin{pro}\label{pro:lift_const_2}
Let
$c>0$
be the supremum of the lifting constants, taken over all 
the unit squares
$Q$
in
$B_\om$, $c=\sup_Q\de(Q)$,
see section~\ref{sect:complex_structure}. Then
$c=2$.
\end{pro}

\begin{proof}
Let
$\si\sub X_\om$
be a Ptolemy circle of radius 1 centered at
$z\in F$,
where
$F\sub X_\om$
is a $\C$-line,
$x$, $y\in\si\cap F$
such that
$x<z<y$
with respect to the order
$O$, $l$
the Ptolemy line through
$z$
that intersects
$\si$.
We denote with
$j:X_\om\to X_\om$
a vertical flip which preserves
$l$
pointwise, see sect.~\ref{subsect:exist_vert_flip}. Then recall
$j(\si)=\si$.

Given
$t\in F$, $x<t<y$,
we denote with
$a=a(t)=|xt|$, $d=d(t)=|ty|$.
Then by Proposition~\ref{pro:euclid_square_fiber}  we have
$a^2+d^2=2$.
For
$v\in\si$
with
$\mu(v)=t$,
we let
$r=r(t)=|tv|$.
Then by Lemma~\ref{lem:mean_geometric},
$r^2=a\cdot d$.
Denote with
$b^+=b^+(t)=|xv|$
and with
$b^-=b^-(t)=|yv|$.
Then by Proposition~\ref{pro:comp_dist_function}
$$(b^+)^4=a^4+r^4=a^2(a^2+d^2)=2a^2$$
and similarly
$(b^-)^4=2d^2$.
Thus
$(b^+)^2=a\sqrt 2$
and
$(b^-)^2=d\sqrt 2$.

Now, we assume that
$a<d$
and take a point
$w\in\si$
on the arc
$xvy$
of
$\si$
with
$\mu(w)=s\in F$
such that
$|sy|=a=|xt|$,
that is,
$w=j(v)$.
Then
$|yw|=|xv|=b^+$, $|xw|=|yv|=b^-$,
and using the Ptolemy equality applied to the quadruple
$(x,v,w,y)\sub\si$,
we obtain
$(b^-)^2-(b^+)^2=|vw|\sqrt{2}$.
This gives
$$d-a=|vw|.$$
Furthermore,
$|st|^2=d^2-a^2$
again by Proposition~\ref{pro:euclid_square_fiber}.

The Ptolemy line 
$l$
intersects the Ptolemy circle 
$\si$
in two points, and we denote with
$u\in l\cap\si$
the point between
$v$, $w$
on the arc
$xvwy$
of
$\si$.
Then
$j(u)=u$
and 
$j$
preserves the tangent Ptolemy line 
$l'$
to
$\si$
through 
$u$.
We denote with
$v'$, $w'\in l'$
points closest in
$l'$
to
$v$, $w$
respectively. Then
$|vv'|=o(|uv|)$, $|ww'|=o(|uw|)$.
Note that
$|uv|=|uw|$
by
$j$-symmetry.
Thus
$|v'w'|=|vw|+o(|uv|)=d-a+o(|uv|)$.

Since
$\ov v'\,\ov w'\sub\ov l'$,
by Lemma~\ref{lem:complex_line} the triangle
$T=\ov z\,\ov v'\,\ov w'$
lies a complex line in
$B_\om$,
where
$\ov z=\pi_\om(z)=\pi_\om(s)=\pi_\om(t)$, $\ov v'=\pi_\om(v')$, 
$\ov w'=\pi_\om(w')$, $\ov l'=\pi_\om(l')$,
and thus the canonical complex structure
$J$
with 
$J(\ov z)=\ov z$
preserves the 2-subspace of
$B_\om$
that contains
$T$.
Hence the area law of lifting applied to
$T$
gives
$$\de(T)^2=c^2\cdot\area T,$$
see Remark~\ref{rem:complex_structure}, where
$\de(T)$
is the displacement of the lifting isometry
$\tau_T:X_\om\to X_\om$.
We have
$|\ov v'\ov w'|=|v'w'|=d-a+o(|uv|)$,
$|\ov z\,\ov v'|=|\pi_\om(t)\pi_\om(v)|+o(|uv|)=
|tv|+o(|uv|)=r+o(|uv|)$
and similarly
$|\ov z\,\ov w'|=r+o(|uv|)$.
Thus
$\area T=\frac{1}{2}r(d-a)+o(|uv|)$.

Denote with
$t'=\mu(v')$
and
$s'=\mu(w')$
points in the $\C$-line 
$F$.
Then
$\de(T)=|t's'|$.
By Lemma~\ref{lem:mu_distortion},
$|tt'|^2\le|vv'|^2+\de(T_t)^2$, $|ss'|^2\le|ww'|^2+\de(T_s)^2$,
where
$T_t=\ov z\,\ov v\,\ov v'$, $T_s=\ov z\,\ov w\,\ov w'$
are triangles in
$B_\om$.
We have
$|\ov v\,\ov v'|$, $|\ov w\,\ov w'|=o(|uv|)$
and
$|\ov z\,\ov v|$, $\ov z\,\ov w|\le 1$.
Thus
$\area T_t$, $\area T_s=o(|uv|)$,
and hence
$\de(T_t)^2$, $\de(T_s)^2=o(|uv|)$
by the area law of lifting. Therefore
$|tt'|^2$, $|ss'|^2=o(|uv|)$,
and using Proposition~\ref{pro:euclid_square_fiber} we obtain
$$||t's'|^2-|ts|^2|\le|tt'|^2+|ss'|^2=o(|uv|).$$
Then
$$\de(T)^2=|t's'|^2=|ts|^2+o(|uv|)=d^2-a^2+o(|uv|)$$
and therefore
$$d^2-a^2=\frac{c^2}{2}r(d-a)+o(|uv|).$$
Since
$\si$
is a Ptolemy circle of radius one, we have
$a$, $d\to 1$
and
$$c^2=\frac{2(a+d)}{\sqrt{ad}}+o(1)\to 4$$
as
$|uv|\to 0$.
Thus
$c=2$.
\end{proof}

\section{The model space $\di\K{\rm H}^k$}
\label{sect:model_space}

Every rank one symmetric space
$M$
of non-compact type is a hyperbolic space
$\K\hyp^k$
over a normed division algebra
$\K$.
The only possibilities are the real numbers
$\K=\R$, $\dim\K=1$;
the complex numbers
$\K=\C$, $\dim\K=2$; 
the quaternions
$\K=\bH$, $\dim\K=4$;
and the octonions
$\K=\C a$, $\dim\K=8$.
Then
$\dim M=k\cdot\dim\K$,
where
$k\ge 2$
and
$k=2$
in the case
$\K=\C a$.
We choose a normalization of the metric so that in the case
$\K=\R$
the sectional curvatures of
$M$
are
$K_\si\equiv-1$
and in the case
$\K\neq\R$
the sectional curvatures of
$M$
are pinched as
$-4\le K_\si\le -1$.

We use the standard notation
$TM$
for the tangent bundle of
$M$
and
$UM$
for the subbundle of the unit vectors.
For every unit vector
$u\in U_oM$,
where
$o\in M$,
the eigenspaces
$E_u(\la)$
of the {\em curvature operator}
$\cR(\cdot,u)u:u^\bot\to u^\bot$,
where
$u^\bot\sub T_oM$
is the subspace orthogonal to
$u$,
are parallel along the geodesic
$\ga(t)=\exp_o(tu)$, $t\in\R$, and the respective
eigenvalues
$\la=-1,-4$
are constant along
$\ga$.
The dimensions of the eigenspaces are
$\dim E_u(-1)=(k-1)\dim\K$, $\dim E_u(-4)=\dim\K-1$,
$u^\bot=E_u(-1)\oplus E_u(-4)$.

\subsection{The M\"obius structure of $\di\K{\rm H}^k$}
\label{subsect:model_moebius_structure}

We let
$Y=\di M$
be the geodesic boundary at infinity of
$M$.
For every
$o\in M$
the function
$d_o(\xi,\xi')=e^{-(\xi|\xi')_o}$
for
$\xi$, $\xi'\in Y$
is a (bounded) metric on
$Y$,
where
$(\xi|\xi')_o$
is the Gromov product based at
$o$.
For every
$\om\in Y$
and every Busemann function
$b:M\to\R$
centered at
$\om$
the function
$d_b(\om,\om):=0$
and
$d_b(\xi,\xi')=e^{-(\xi|\xi')_b}$,
except the case
$\xi=\xi'=\om$,
is an (unbounded) metric on
$Y$,
where
$(\xi|\xi')_b$
is the Gromov product with respect to
$b$.
Since
$M$
is a
$\CAT(-1)$-space, 
the metrics
$d_o$, $d_b$
satisfy the Ptolemy inequality and furthermore
all these metrics are pairwise M\"obius equivalent,
see \cite{FS1}.

We let
$\cM$
be the {\em canonical} M\"obius structure on
$Y$
generated by the metrics of type
$d_o$, $o\in M$.
Recall that
$\cM$
is the class of metrics on
$Y$
each of which is M\"obius equivalent to any
$d_o$.
Then
$Y$
endowed with
$\cM$
is a compact Ptolemy space. Every metric
$d\in\cM$
is of type
$d=d_b$
for some Busemann function
$b:M\to\R$,
or
$d=\la d_o$,
for some 
$o\in M$
and 
$\la>0$.
We emphasize that in general metrics of
$\cM$
are neither Carnot-Carath\'eodory metrics nor length
metrics.

The following lemma is a modification of 
Lemma~\ref{lem:sinversion_minversion}.

\begin{lem}\label{lem:invariant_sphere} Let
$\phi:X\to X$
be a M\"obius involution,
$\phi^2=\id$,
of a Ptolemy space 
$X$
with
$\phi(\om)=\om'$
for distinct
$\om$, $\om'\in X$.
Then there is a unique metric sphere
$S\sub X$
between
$\om$, $\om'$
invariant for
$\phi$, $\phi(S)=S$.
\end{lem}

\begin{proof} Let
$d$
be a metric of the M\"obius structure with infinitely
remote point 
$\om'$.
Since
$\phi(\om)=\om'$,
the point 
$\om$
is infinitely remote for the induced metric
$\phi^\ast d$.
Thus for some
$\la>0$
we have
$$(\phi^\ast d)(x,y)=\frac{\la d(x,y)}{d(x,\om)d(y,\om)}$$
for each
$x$, $y\in X$
which are not equal to
$\om$
simultaneously. We let
$S=S_r^d(\om)\sub X$
be a metric sphere between
$\om$, $\om'$
with
$r^2=\la$, $d(x,\om)=r$
for every
$x\in S$.
Then
$$d(\phi(x),\om)=d(\phi(x),\phi(\om'))=(\phi^\ast d)(x,\om')
  =\la/d(x,\om)=r$$
for every
$x\in S$.
Hence
$\phi(S)=S$.
For any
$x\in X$
with 
$d(x,\om)\lessgtr r$
the same argument shows that
$d(\phi(x),\om)\gtrless r$,
thus an invariant sphere between
$\om$, $\om'$
is unique.
\end{proof}

\begin{pro}\label{pro:rank_one_basic_axioms} Let
$Y=\di M$
be the boundary at infinity of a rank one symmetric
space
$M$
of noncompact type. Then regarded as a compact Ptolemy space 
with respective M\"obius structure
$\cM$, $Y$
satisfies properties (E) and (I).
\end{pro}

\begin{proof} Every rank one symmetric space
$M$
of noncompact type contains geodesic subspaces 
$L$
isometric to
$\hyp^2$.
Then
$\si=\di L\sub Y$
is a Ptolemy circle. Hence, the property (E) is fulfilled for 
$Y$.

Given distinct
$\om$, $\om'\in Y$
and a metric sphere
$S\sub Y$
between
$\om$, $\om'$,
we show that there is a unique space inversion
$\phi=\phi_{\om,\om',S}:Y\to Y$
w.r.t.
$\om$, $\om'$, $S$.

Let
$b:M\to\R$
be a Busemann function of the geodesic ray
$x\om\sub M$, $x\in l=\om\om'\sub M$.
Then the Gromov product
$(\xi|\xi')_b\in\R$
w.r.t.
$b$
is well defined for
$\xi$, $\xi'\in Y\sm\om$, 
see e.g. \cite[sect.3.2]{BS}, \cite{BK}, and the function
$d(\xi,\xi')=e^{-(\xi|\xi')_b}$
is a metric of the M\"obius structure with infinitely remote point 
$\om$.
Let
$r>0$
be the radius of
$S$, $S=S_r^d(\om')$.
We take 
$o\in l$
with
$b(o)=-\ln r$.
Let
$u\in U_oM$
be a tangent vector to
$l$.
For every
$\xi\in Y$
such that the direction
$v_\xi\in U_oM\cap u^\bot$
of the ray
$o\xi\sub M$
is an eigenvector of the curvature operator,
$v_\xi\in E_u(-1)\cup E_u(-4)$,
the ideal triangle
$\om\xi\om'$
lies in a geodesic subspace of
$M$
isometric to
$\hyp^2$
(the case
$v_\xi\in E_u(-1)$)
or to
$\frac{1}{2}\hyp^2$
(the case
$v_\xi\in E_u(-4)$).
By symmetry,
$o$
is an equiradial point of
$\om\xi\om'$,
see \cite[Proposition~2.5]{BK}, thus
$(\xi|\om')_b=b(o)=-\ln r$
and
$d(\xi,\om')=r$,
that is,
$\xi\in S$.

The central symmetry
$f:M\to M$
at
$o$, $d_of=-\id_{T_oM}$,
is an isometry that induces a M\"obius involution 
$\phi=\di f:Y\to Y$
without fixed points and with
$\phi(\om)=\om'$.
Then for every
$\xi\in S$
with
$v_\xi\in E_u(-i^2)$
we have
$v_{\phi(\xi)}=-v_\xi\in E_u(-i^2)$, $i=1,2$.
Thus
$\phi(\xi)\in S$.
It follows from Lemma~\ref{lem:invariant_sphere} that the sphere
$S$
is invariant for 
$\phi$, $\phi(S)=S$.
Furthermore, every Ptolemy circle
$\si\sub Y$
through
$\om$, $\om'$
is the boundary at infinity of a geodesic subspace
$L\sub M$
containing
$l$
and isometric to
$\hyp^2$,
see \cite{FS1}. Since
$L$
is invariant for 
$f$,
we have
$\phi(\si)=\si$.
Therefore,
$\phi=\phi_{\om,\om',S}:Y\to Y$
is a space inversion w.r.t.
$\om$, $\om'$, $S$.

Assume there is a M\"obius involution
$\psi:Y\to Y$, $\psi^2=\id$,
without fixed points and with 
$\psi(\om)=\om'$, $\psi(S)=S$,
such that
$\psi(\si)=\si$
for every Ptolemy circle
$\si\sub Y$
through
$\om$, $\om'$.
We show that
$\psi=\phi$.

Let
$B\sub Y$
be the union of all the Ptolemy circles through
$\om$, $\om'$.
Then
$\psi|B=\phi|B$
because 
$S$
intersects every arc between
$\om$, $\om'$
of a Ptolemy circle through
$\om$, $\om'$
at a unique point, and any M\"obius automorphism of
$\si$
is uniquely determined by values at three distinct points.

Note that the existence of the fibration
$\pi_\om:Y_\om\to B_\om$
with the base 
$B_\om$
isometric to
$\R^{(k-1)\dim\K}$, $\dim M=k\dim\K$,
is well known, and it follows from consideration of a model,
see \cite{Gol} for the case
$M=\C\hyp^k$.

We put
$\eta=\psi^{-1}\circ\phi$
and note that
$\eta|B=\id_B$.
Then
$\eta:Y_\om\to Y_\om$
is an isometry w.r.t. the metric
$d$
that induces the identity map of the base
$B_\om$.
Thus
$\eta$
preserves any foliation of
$Y_\om$
by Busemann parallel Ptolemy line. Then 
$\eta=\id$
by the same argument as in the proof of 
Lemma~\ref{lem:simply_transitivity_shifts}. Hence,
$\psi=\phi$,
and the property (I) is fulfilled for 
$Y$.
\end{proof}

\subsection{Isomorphism with the model space $\di\C{\rm H}^k$}
\label{subsect:isomorphism_model}

For a given
$\om\in X$
we fix as usual a metric from the M\"obius structure
with infinitely remote point
$\om$.
As in sect.~\ref{sect:coordinates} we introduce standard coordinates in
$X_\om$
with origin
$o\in X_\om$
identifying
$X_\om$
with
$\R^n\times\R=\R^{n+1}$
by
$x=(z,h)$.
Moreover, we regard
$X_\om$
as the nilpotent group
$N_\om$,
see sect.~\ref{subsect:max_unipotent}, identifying every
$x\in X_\om$
with
$g\in N_\om$
via the rule
$x=g(o)$.
Then the multiplication law of 
$N_\om$
in the coordinates for
$g=(z,h)$, $g'=(z',h')$
is given by
$$g\cdot g'=\left(z+z',h+h'+\frac{1}{2}\langle J(z),z'\rangle\right),$$
where we used Lemmas~\ref{lem:multi_law}, \ref{lem:coord_lift_triangle}
and Proposition~\ref{pro:lift_const_2}. Next, we fix a canonical basis
$b=\{u_1,v_1,\dots,u_k,v_k\}$
of
$B_\om=\R^n$, $n=2k$,
for the complex structure
$J$, 
where
$v_i=J(u_i)$, $i=1,\dots,k$.
Then
$z=(z_1,\dots,z_k)$
for every
$z\in B_\om$,
where
$z_i=x_iu_i+y_iv_i$
with
$x_i$, $y_i\in\R$
and
$z=\sum_iz_i$,
that is, we identify
$B_\om$
with the complex coordinate space
$\C^k$.
Note that
$\conj z_i=x_iu_i-y_iv_i$
for every
$z_i=x_iu_i+y_iv_i$,
and we use the standard notation
$\conj z=\ov z$
for the complex conjugation.

Let
$(z,z')=\sum_iz_i\ov z_i'$
be the standard Hermitian form on
$\C^k$.

\begin{lem}\label{lem:mult_hermitian} For each
$g=(z,h)$, $g'=(z',h')\in N_\om$
we have
$$g\cdot g'=\left(z+z',h+h'-\frac{1}{2}\im(z,z')\right).$$
\end{lem}

\begin{proof}
This follows from the equality
$\langle J(z),z'\rangle=-\im(z,z')$. 
\end{proof}

It follows that
$N_\om$
is the higher-dimensional Heisenberg group
$\bH^k$,
see \cite[sect.~2.1.2]{CDPT}.

\begin{proof}[Proof of Theorem~\ref{thm:complex_hyperbolic}]
We only need to consider the case
$p=1$
and to show that
$X$
is M\"obius equivalent to
$\di\C\hyp^k$
taken with the canonical M\"obius structure,
$\dim X=n=2k-1$.

It follows from Proposition~\ref{pro:complex_structure} that
the dimension of the base 
$B_\om$
is even,
$\dim B_\om=n-1=2(k-1)$.
By Corollary~\ref{cor:onedimensional_base},
$k\ge 2$.
In a given dimension a compact Ptolemy space that
satisfies (E), (I) with
$p=1$
is uniquely determined up to a M\"obius isomorphism 
by the lifting constant
$c$,
and this constant itself is uniquely determined by these properties,
$c=2$.

We give more detail for this argument, which are however
straightforward. Let
$Y=\di\C\hyp^k$
be the boundary at infinity of the complex hyperbolic space
$\C\hyp^k$
taken with the canonical M\"obius structure. Since
$\C\hyp^k$
is a
$\CAT(-1)$-space,
the M\"obius structure of
$Y$
is a compact Ptolemy space, and
$\dim Y=2k-1$.
By Propositions~\ref{pro:rank_one_basic_axioms},
$Y$
satisfies properties (E) and (I).

We fix metrics from the M\"obius structures of
$X$, $Y$
with infinitely remote points
$\om\in X$
and
$\zeta\in Y$
respectively and introduce standard coordinates in
$X_\om$, $Y_\zeta$
identifying these spaces with respective nilpotent groups
of isometries
$N_\om$, $N_\zeta$.
By Lemma~\ref{lem:mult_hermitian}, the multiplications laws in
$N_\om$, $N_\zeta$
are identical, thus the groups are isomorphic via the isomorphism 
which associates to each other the elements with equal coordinates.
Moreover, this isomorphism is an isometry between
$X_\om$, $Y_\zeta$
because for any
$x=(z,h)$, $x'=(z',h')\in X_\om$
we have
\begin{equation}\label{eq:koranyi_gauge}
|xx'|^4=|z-z'|^4+16|h-h'-\frac{1}{2}\im(z\ov z')|^2
\end{equation}
by the distance formula~(\ref{eq:distance_formula}) from Proposition~\ref{pro:vert_flip} together with
Proposition~\ref{pro:comp_dist_function}.
\end{proof}

\begin{rem}\label{rem:koranyi_gauge}
The distance on a Heisenberg group 
$\bH^k$
given by Eq.~(\ref{eq:koranyi_gauge})
is called the {\em Kor\'anyi gauge}, see \cite{CDPT}. In our approach,
this metric is a member of the canonical M\"obius structure
on the boundary at infinity
$Y=\di\C\hyp^k$,
cp. \cite[sect.~3.4.5]{CDPT}.
\end{rem}


\end{document}